\documentclass[a4paper,10pt]{amsart}
\usepackage[utf8]{inputenc}
\usepackage{enumerate}
\usepackage{stmaryrd}
\usepackage{amsthm}
\usepackage{amsmath}
\usepackage{amsfonts}
\usepackage{amssymb}
\usepackage{mathtools}
\usepackage{mathrsfs}
\usepackage{setspace}
\usepackage{xcolor}
\usepackage{textgreek}
\usepackage{todonotes}
\usepackage{soul}
\usepackage[a4paper, left=3cm, right=3cm, top=3cm, bottom=3cm]{geometry}
\usepackage{thmtools} 
\usepackage{bm}

\usepackage{esint}

\usepackage[pdfdisplaydoctitle,colorlinks,breaklinks,urlcolor=blue,linkcolor=blue,citecolor=blue]{hyperref} 

\usepackage[nameinlink,capitalise]{cleveref}

\newtheorem{thm}{Theorem}[section]

\newtheorem*{thm*}{Theorem}
\newtheorem{definition}[thm]{Definition}
\newtheorem{cor}[thm]{Corollary}

\newtheorem{lem}[thm]{Lemma}

\newtheorem{prop}[thm]{Proposition}

\newtheorem{ass}[thm]{Assumption}

\newtheorem{example}[thm]{Example}

\newcommand{\one}{\mathbf{1}}

\newcommand{\eps}{\varepsilon}

\newcommand{\N}{\mathbb{N}}
\newcommand{\Z}{\mathbb{Z}}
\newcommand{\R}{\mathbb{R}}

\newcommand{\PP}{\mathbb{P}}
\newcommand{\norm}[1]{\left\|#1\right\|}
\newcommand{\T}{\mathbb{T}}

\newcommand{\EE}{\mathbb{E}}

\newcommand{\dd}{\mathrm{d}}

\newcommand{\mathd}{\,\mathrm{d}}
\renewcommand{\leq}{\leqslant}
\renewcommand{\geq}{\geqslant}

\newcommand{\seminorm}[1]{\left\llbracket {#1}\right\rrbracket}

\newcommand{\lllbracket}{\llbracket\mkern-2.5mu\lbrack}
\newcommand{\rrrbracket}{\rbrack\mkern-2.5mu\rrbracket}
\newcommand{\trinorm}[1]{\left\lllbracket {#1}\right\rrrbracket}

\theoremstyle{remark}
\newtheorem{rmk}[thm]{Remark}

\numberwithin{equation}{section}
\allowdisplaybreaks

\title[Anomalous properties of stochastic active scalars]{Anomalous properties of 2D active scalars \\ perturbed by rough transport noise}

\author[L. Galeati]{Lucio Galeati}
\address{Dipartimento di Ingegneria e Scienze dell’Informazione e Matematica, Università degli Studi
dell’Aquila, Italy} 
\email{lucio.galeati@univaq.it}

\author[E. Luongo]{Eliseo Luongo}
\address{Institut für Mathematik, Technische Universität Berlin, 10623 Berlin, Germany} 
\email{eliseo.luongo@tu-berlin.de  }
\author[U. Pappalettera]{Umberto Pappalettera}
\address{Departement Mathematik und Informatik, Universit\"at Basel, Spiegelgasse 1, 4051 Basel, Switzerland} 
\email{umberto.pappalettera@unibas.ch}

\date\today

\begin{document}

\begin{abstract}
We study SPDEs associated with $2$D active scalars driven by incompressible transport noise of Kraichnan type, with regularity exponent $\alpha\in (0,1)$; our examples include the Euler, SQG and IPM systems.
We investigate whether a number of ``turbulent'' phenomenologies, which are well understood in the linear Kraichnan model, persist in this nonlinear setting, uniformly in vanishing viscosity approximations.
First, for suitable values of $\alpha$ and initial data in $L^p_x$, we establish \emph{anomalous regularization} estimates, measured in appropriate endpoint Besov-type spaces of regularity $\beta=\beta(\alpha,p)>0$. Remarkably, these results allow for some scaling supercritical regimes of the parameters $\alpha,p$; on the other hand, for (sub)critical parameters, we recover the same regularity exponent $\beta=1-\alpha$ as in the linear case.
Second, in the (sub)critical case, we further prove \emph{anomalous integrability}, namely solutions becoming instantaneously $L^\infty_x$-valued at positive times, uniformly in the viscosity; moreover, in this case we establish strong existence and pathwise uniqueness of solutions to the inviscid SPDE, which are recovered as the unique vanishing viscosity limit.
Finally, in the $2$D Euler case, for $\alpha\in (0,1/2)$, we establish \emph{anomalous dissipation} of enstrophy and sharpness of anomalous regularization.\\[1ex]



\end{abstract}

\maketitle


\section{Introduction}\label{sec:intro}
In this paper we investigate properties of nonlinear stochastic active scalar equations on $\R^2$, driven by a rough transport noise and written in the general form
\begin{equation}\label{eq:nonlinear_transport}
\begin{cases} 
\mathd \theta + u \cdot \nabla \theta \mathd t + \circ\, \mathrm{d} W \cdot \nabla \theta =
  f \mathd t, \\
  u := \mathscr{R} \theta,
  \\
  \theta |_{t = 0} = \theta_0 \in L^r_x \cap L^p_x, \quad 
  f \in L^1_T(L^r_x \cap L^p_x),
\end{cases}
\end{equation}
for suitable integrability exponents\footnote{Since we are working on the full space $\R^2$, operators like $\mathscr{R}_{BS}=-\nabla ^\perp (-\Delta)^{-1}$ are not continuous from $L^p_x$ to itself, which is why we need an additional exponent $r\leq p$ (see also \autoref{subsec:nonlocal_operators}). Its presence is technical and we will mostly focus on regularity regimes associated to the higher exponent $p$ (see e.g. \autoref{ass:exponent.beta} below); to simplify, the reader may think of $r=1$ at a first reading.} $1< r \leq 2 \leq p< +\infty$ and where the velocity field $u$ is reconstructed from the active scalar $\theta$ via some linear operator $\mathscr{R}$. Henceforth we will always assume that $\mathscr{R}$ is a singular integral operator of convolutional type, associated to some Fourier symbol $\widehat {\mathscr{R}}$ (i.e. $\widehat{ (\mathscr{R} \theta)}=\widehat {\mathscr{R}} \hat \theta$) which is smooth outside of the origin and such that $\mathscr{R}$ is divergence-free in the sense of distributions.
The main examples we have in mind are the following:
\begin{itemize}
    \item the Biot--Savart operator $\mathscr{R}=\mathscr{R}_{BS}:=-\nabla ^\perp (-\Delta)^{-1}$ (where $\nabla^\perp=(-\partial_2,\partial_1)$) associated to the 2D Euler equations in vorticity form;
    \item divergence-free Caldéron-Zygmund operators, denoted by $\mathscr{R}=\mathscr{R}_{CZ}$ in the following; relevant examples are the Surface Quasi-Geostrophic (SQG) equations for $\mathscr{R}=\nabla^\perp (-\Delta)^{-1/2}$, and the Incompressible Porous Media (IPM) equations for $\mathscr{R}=\nabla^\perp (-\Delta)^{-1} \partial_1$.
\end{itemize}

The SPDE \eqref{eq:nonlinear_transport} is considered on a given time interval $[0,T]$. The shorthand notation $L^m_T L^q_x := L^m([0,T];L^q(\R^2;\R^N))$, for $m,q \in [1,\infty]$ and $N$ depending on the context, specifies the time-space integrability of scalars, vectors, or matrices. A similar convention is used for $L^m_\omega$-integrability with respect to the randomness parameter $\omega \in \Omega$, where $(\Omega,\mathcal{F},\PP)$ denotes the probability space where $(\theta,W)$ are defined and $\EE$ denotes expectation with respect to $\PP$; we refer to \autoref{subsec:notation} for additional details concerning other functional spaces.

Throughout the paper we will mostly take $\mathd W$ to be a white-in-time, incompressible, Gaussian noise of Kraichnan type, with spatial regularity $\alpha \in (0,1)$; namely, its space covariance $C=C_\alpha$ satisfies condition \eqref{eq:isotropic.covariance} below.
However, in several of our results, we can allow slightly larger classes of $W$, as better clarified in \cref{rmk:anisotropy} and in the main body of the paper. 
We refer to \cref{sec:preliminaries} for more details on the covariance of $\mathd W$ and the notion of solutions to \eqref{eq:nonlinear_transport}. 

Let us shortly explain the interest in the study of \eqref{eq:nonlinear_transport}, referring to \cref{subsec:intro_literature} for a more extensive discussion.
The linear SPDE with this type of noise (namely with $\mathscr{R}\equiv f\equiv 0$) was first proposed by Kraichnan \cite{Kraichnan1968} as a synthetic model of passive scalar turbulence and has since then become rather popular in the physics community, cf. \cite{bernard1998slow,CFG2008}. 
This modelling perspective is also close to the celebrated Hasselmann's proposal \cite{Ha76} of representing the \emph{effective} influence of unresolved degrees of freedom on resolved variables through stochastic perturbations; see also \cite{Arnold2001} for a revisitation of this approach, \cite{crisan2023implementation} for a specific application to a class of idealized climate models, and \cite{DSF2024} for more a recent discussion in the context of stochastic fluid dynamics.
Stochastic model reduction (see for instance \cite{MTVE2001}) constitutes another mathematical realization of the same principle.

Moving our attention back to nonlinear SPDEs in the form of \eqref{eq:nonlinear_transport}, the introduction of general transport noise as a proxy for modelling the small, unresolved, turbulent scales of the fluid itself, is more recent and has been advocated by the mathematical community e.g. in \cite{BCF1991,Holm2015,Memin2014,FlaPap2022,DebPap2026}.
In the rough regime $\alpha\in (0,1)$, it was first noticed in \cite{CoMa23} that some features of the linear SPDE persist in nonlinear models (specifically for 2D Euler equations $\mathscr{R}=\mathscr{R}_{BS}$), drastically affecting the resulting solution theory for \eqref{eq:nonlinear_transport} compared to its deterministic counterpart (namely with $W\equiv 0$); generalizations have then been provided e.g. in \cite{jiao2025well,bagnara2025regularization,JiaLuo2026}.

Here, rather than examining well-posedness of \eqref{eq:nonlinear_transport}, we are interested in understanding whether solutions to \eqref{eq:nonlinear_transport} display several phenomenologies which are often expected in more realistic \emph{turbulent fluids}.
For this reason, rather than studying \eqref{eq:nonlinear_transport} directly, we regularize it by adding a small viscosity at the right-hand side and possibly approximating the external forcing $f$ and the initial condition $\theta_0$:
\begin{align}
\begin{cases} \label{eq:nonlinear_viscous} 
\mathd \theta^\nu + u^\nu \cdot \nabla \theta^\nu \mathd t + \circ\, \mathrm{d} W \cdot \nabla \theta^\nu =
  f^\nu \mathd t + \nu \Delta \theta^\nu \mathd t, \\
  u^\nu := \mathscr{R} \theta^\nu,\\
  \theta^\nu |_{t = 0} = \theta^\nu_0.
\end{cases}
\end{align}
We then consider vanishing viscosity limits of \eqref{eq:nonlinear_viscous} as the parameter $\nu \downarrow 0$ (cf. \cref{defn:vanishing_scheme}) and study properties of the solutions $\{ \theta^\nu\}_{\nu \in (0,1)}$ that hold uniformly in $\nu$, so to recover information on \emph{physically relevant} solutions to \eqref{eq:nonlinear_transport}.
We investigate the following problems:
\begin{itemize}
    \item 
    \emph{Anomalous regularization}: solutions gain positive regularity in a $L^2_x$-based Besov-type space, in mean square with respect to time and expectation;
    \item 
    \emph{Anomalous integrability}: solutions gain $L^\infty_x$ integrability at every positive time, with finite $p$-th probabilistic moment;
    \item 
    \emph{Anomalous dissipation}: the mean square $L^2_x$-norm of solutions is strictly dissipated, continuously in time.
\end{itemize}

Let us remark that none of the above properties is expected for smooth solutions of transport-like equations like \eqref{eq:nonlinear_transport}. Indeed, assuming $f \equiv 0$ for simplicity, the active scalar $\theta$ is formally transported by a measure-preserving, invertible map (the flow associated to $u$) and thus it cannot gain integrability over the initial condition $\theta_0$, nor lose $L^2_x$ mass; see also \autoref{prop:dissipation.implies.irregularity} for a conditional version of this statement.
For the same reason, any positive regularity gain resulting in $\theta_t\in L^q_x$ for some $q>p$ and time $t>0$ is formally incompatible with \eqref{eq:nonlinear_transport}.
Nonetheless, due to the spatial roughness of the noise, all these phenomena are actually possible. We refer to them as \emph{anomalous}, since we deal with properties that are true for $\theta^\nu$ for every fixed $\nu \in (0,1)$ and persist in the vanishing viscosity limit.

There are clear predecessors motivating the interest in understanding these anomalous phenomena for the nonlinear SPDEs \eqref{eq:nonlinear_transport}, as well as fluid dynamics models more generally.
Anomalous dissipation of energy plays a key role in Kolmogorov's theory, to the point of being regarded as the ``zero-th law of turbulence'' \cite{Frisch1995}.
The theories by Kolmogorov, Obukhov--Corrsin and Monin--Yaglom, encompassing both hydrodynamic and passive scalar turbulence, predict uniform-in-viscosity estimates for suitable structure functions at statistical equilibrium; these predictions suggest some persistence of regularity, as well as possibly some turbulent self-regularizing effect, see \cite{drivas2022self}.
In the mathematical community, anomalous regularization seems to have been clearly identified much more recently, in a research trend started with \cite{CoMa23}; it applies this idea at a dynamical level, showing that \emph{arbitrary} initial data in suitable Sobolev spaces gain better regularity over time (either at the inviscid level, or with uniform-in-viscosity bounds). There is a connection between this property and the construction of statistical equilibria satisfying the aforementioned predictions, see the discussions in \cite{GaGrMa24,Rowan25}.
In turn, anomalous regularization  can be used to infer anomalous integrability, as done in \cite{DrGaPa25}.
In the linear Kraichnan model, the aforementioned three anomalous properties are by now relatively well-understood, see \cite{Rowan2024,GaGrMa24,DrGaPa25,Rowan25} and the references therein. They are also satisfied by other simplified models of turbulence like the Burgers equations, see \cite{eyink2015spontaneous}; we refer to \autoref{subsec:intro_literature} for a more comprehensive bibliographic discussion.

\subsection{Anomalous regularization}
\label{ssec:intro.anomalous.reg}

Our first main result describes regimes where anomalous regularization holds.
In order to make the statements precise, following \cite{DrGaPa25}, for $\beta\in (0,1]$ we define the Besov-type space $\tilde{L}^2_{\omega,T} \tilde B^\beta_{2,\infty}$ as the collection of random functions $\theta\in L^2_{\omega,T,x}$ such that
\begin{align} \label{def:Besov-like.reg}
    \llbracket \theta \rrbracket_{\tilde{L}^2_{\omega,T} \tilde B^\beta_{2,\infty}}
    := \bigg( \sup_{z\neq 0} \frac{1}{|z|^{2\beta}} \int_0^T \EE[\| \delta_z \theta_t\|_{L^2_x}^2] \dd t \bigg)^{1/2} < \infty;
\end{align}
here we adopt the increment notation $\delta_z \theta_t (x) := \theta_t (x + z) - \theta_t (x)$.
The quantity in \eqref{def:Besov-like.reg} is closely related to the second-order structure function of $\theta$ and the parameter $\beta$ measures space regularity of solutions. 
For instance, if $\beta \in (0,1)$ and $\theta \in \tilde{L}^2_{\omega,T} \tilde B^\beta_{2,\infty}$, then $\theta \in L^2_{\omega,T} H^{\beta-\varepsilon}_x$ for every $\varepsilon \in (0,\beta)$.
We refer to \autoref{subsec:besov} for more details on such spaces.

In the linear case $\mathscr{R} \equiv0$, it was shown in \cite[Theorem 1.3]{DrGaPa25} that the unique solution $\theta$ associated to any initial condition $\theta_0\in L^2_x$ belongs to $\tilde{L}^2_{\omega,T} \tilde B^{1-\alpha}_{2,\infty}$, and that this regularity is optimal, e.g. in the sense that $\theta\notin\tilde{L}^2_{\omega,T} \tilde B^{\beta}_{2,\infty}$ for any $\beta>1-\alpha$.
As a consequence, we do not expect to obtain better regularity exponents in the nonlinear case.

Here, for nonlinear equations \eqref{eq:nonlinear_transport}, we identify both regimes in which solutions still gain regularity of order $1-\alpha$, and regimes in which we are only able to prove some partial regularity gain with $\beta\in (0,1-\alpha)$. 
As a consequence, anomalous regularization holds also in situations where uniqueness of solutions is not expected, see \cref{rmk:pathwise.uniquness} below.
More precisely, since we always work with viscous approximations $\theta^\nu$ solving \eqref{eq:nonlinear_viscous}, we obtain uniform-in-$\nu$ regularity estimates; vanishing viscosity limits then select particular solutions of the inviscid equation \eqref{eq:nonlinear_transport} which inherit the same anomalous regularity.

The exact amount of anomalous regularization $\beta\in (0,1-\alpha]$ we obtain depends on the choice of $\mathscr{R}$ and the values of $\alpha$, $p$. More precisely, we work under the following:
\begin{ass} \label{ass:exponent.beta}
    Suppose either of the following:
    \begin{itemize}
        \item
        $\mathscr{R}=\mathscr{R}_{BS}$, $r\in (1,2)$, $p\in [2,\infty)$ and\footnote{\label{foot:2}For 2D Euler, we could allow $p=\infty$ in \eqref{eq:definition.regularity}, using that $L^r_x\cap L^\infty_x\subset L^r_x\cap L^p_x$ for any $p\geq 3$.
        We only exclude $p=\infty$ for technical reasons: the vanishing viscosity scheme as defined in \autoref{defn:vanishing_scheme} may be inconsistent with the use of non-separable spaces like $L^\infty_x$. }
        \begin{align} \label{eq:definition.regularity} 
    \beta :=
    \begin{cases}
      1-\alpha, \quad &\mbox{if} \quad \ p \in [2,3),\ \alpha \leq 1-\frac{1}{p},
      \\
      1-\alpha, \quad &\mbox{if} \quad \ p \in [3,\infty),
      \\
      (1-\alpha)(p-2), \quad &\mbox{if} \quad \ p \in (2,3),\  1-\frac{1}{p}<\alpha<1;
    \end{cases}
\end{align}
\item 
$\mathscr{R}=\mathscr{R}_{CZ}$, $r\in (1,2]$, $p\in (2,\infty)$ and, adopting the convention $\frac{1}{0}=+\infty$,
\begin{align} \label{eq:definition.regularity_CZ}  
    \beta :=
    \begin{cases}
      1-\alpha, 
\quad &\mbox{if} \quad \ p \in (2,6),\ \alpha \leq \frac{1}{2}-\frac{1}{p},
\\
1-\alpha, 
\quad &\mbox{if} \quad \ p \in [6,\infty),\ \alpha \leq \frac{1}{2}-\frac{1}{2p-6},
\\
(\frac12-\alpha)(p-2) \quad &\mbox{if} \quad \ p \in [6,\infty),\ \frac{1}{2}-\frac{1}{2p-6}<\alpha<\frac{1}{2},
\\
(\frac12-\alpha)(p-2) \quad &\mbox{if} \quad \ p \in (2,6),\ \frac{1}{2}-\frac{1}{p}<\alpha<\frac{1}{2}.
    \end{cases}
\end{align}
    \end{itemize}
\end{ass}

The precise statement of our first main result is the following:

\begin{thm}[Anomalous regularization] \label{thm:regularization}
Let $\mathscr{R},r,p,\alpha,\beta$ be such that \cref{ass:exponent.beta} holds. Then for any $\theta_0\in L^r_x\cap L^p_x$, $f\in L^1_T(L^r_x\cap L^p_x)$ and any vanishing viscosity scheme $\{(\theta^\nu_0,f^\nu,\theta^\nu)\}_{\nu \in (0,1)}$ of \eqref{eq:nonlinear_transport} in the sense of \autoref{defn:vanishing_scheme}, the solutions $\{\theta^\nu\}_{\nu \in (0,1)}$ to \eqref{eq:nonlinear_viscous} are uniformly bounded in $L^{\infty}_{\omega,T}(L^r_x\cap L^p_x)\cap \Tilde{L}^2_{\omega,T}\Tilde{B}^{\beta}_{2,\infty}$, their laws are tight in suitable topologies, and any accumulation point as $\nu \downarrow 0$ is a weak solution of \eqref{eq:nonlinear_transport}, in the sense of \autoref{martingale_sol}, belonging to the space $L^{\infty}_{\omega,T}(L^r_x\cap L^p_x)\cap\Tilde{L}^2_{\omega,T}\Tilde{B}^{\beta}_{2,\infty}$ with almost surely weakly continuous paths in $L^2_x$. 
\end{thm}

The first part of \autoref{thm:regularization}, concerning uniform regularity estimates for $\{\theta^\nu\}_{\nu \in (0,1)}$, will be established in \autoref{sec:computation}.
The proof is then completed in \autoref{subsec:weak_existence}, where details on the precise topology with respect to which tightness and convergence hold are given, cf. the space $\mathcal{E}$ defined in \eqref{polish_space_tightness}.
We refer to \autoref{subsec:intro_existence_uniqueness} below for further properties of solutions $\theta$ to \eqref{eq:nonlinear_transport} obtained by vanishing viscosity, especially conditions under which they are unique.

\begin{figure}[htbp] 
    \centering
    \includegraphics[width=0.49\textwidth]{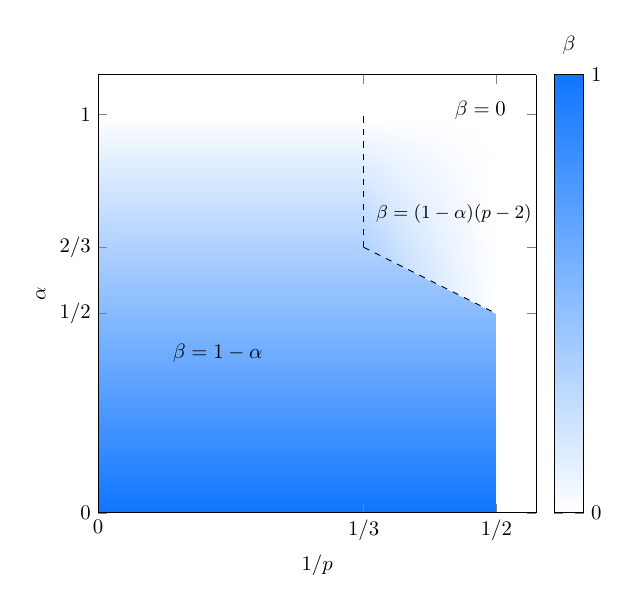}
    \includegraphics[width=0.49\textwidth]{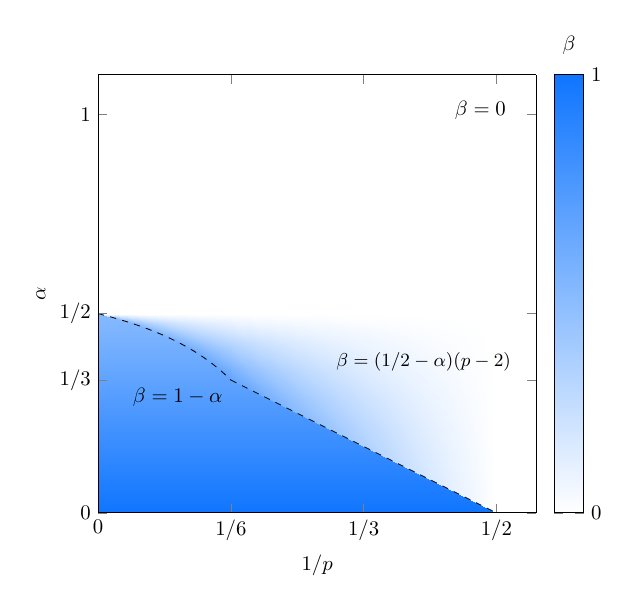}
    \caption{A visualization of the anomalous regularization exponent \eqref{eq:definition.regularity} for $\mathscr{R}=\mathscr{R}_{BS}$ (left) and $\mathscr{R}=\mathscr{R}_{CZ}$ (right).
    Dotted lines separate regions with different regularization regimes. }
    \label{fig.regularization}
\end{figure}
For the linear, incompressible Kraichnan model $\dd\theta + \circ \dd W \cdot \nabla \theta =0$ corresponding to $\mathscr{R}=0$ and no external forcing, the precise structure of the \emph{dissipation measure} governing the local energy balance has been studied in \cite[Section 5]{DrGaPa25}. There it is shown that, due to the specific structure of covariance of $W$, a representation of the dissipation measure \`{a} la Duchon--Robert \cite{DuRo00} produces a coercive term that carries information about the $\tilde{L}^2_{\omega,T} \tilde{B}^{1-\alpha}_{2,\infty}$ Besov-type regularity of solutions. 
More precisely, setting $\mathscr{R}\equiv f\equiv 0$, \cite[equation (1.11)]{DrGaPa25} in dimension $d=2$ yields
\begin{equation}\label{eq:dissipation_measure}
    \|\theta_0\|_{L^2_x}^2 - \EE[\| \theta_T\|_{L^2_x}^2] \propto \lim_{\eps \to 0} \int_0^T \int_{\mathbb{S}^1} \frac{\EE[\| \delta_{\varrho z}\theta_t\|_{L^2_x}^2]}{\varrho^{2(1-\alpha)}} \sigma(\dd z) \dd t
\end{equation}
where $\sigma(\dd z)$ denotes the surface measure on $\mathbb{S}^1$.

The proof of \cref{thm:regularization} passes through a general criterion (\autoref{lem:gron_increments}), leveraging analogous formulae for a stochastic process $\theta^\nu$ solving a nonlinear SPDE of the form \eqref{eq:nonlinear_viscous}.
Let us shortly discuss the main idea, restricting to $\beta=1-\alpha$ for simplicity.
By considering a well-chosen family of radial mollifiers $\{\chi^\eps\}_{\eps>0}$ (taken from \cite{DrGaPa25}, but also reminiscent of \cite{DuRo00,Novack2024}), we can use the structure of the SPDE to study energy-type balance governing the evolution of $\EE[\langle \theta_t,\chi^\eps\ast \theta_t\rangle]$ as $\eps\to 0^+$, see Equation \eqref{eq:energy_balance}.
Exploiting the divergence-free structure of $u^\nu$, one can see that $u^\nu\cdot\nabla \theta^\nu$ contributes to this balance with a trilinear term, reminiscent of the third-order longitudinal structure function (cf. \eqref{eq:prototype_longitudinal_function}):
\begin{equation*}
    \mathcal{T}_{u^\nu,\theta^\nu}(t,\eps) = \int_{\mathbb{R}^2} \int_{\{| z |  \leqslant 1\}} z \cdot \delta_{\eps z} u^\nu_t (y)  | \delta_{\eps z} \theta^\nu_t (y)
  |^2 \mathd y \mathd z;
\end{equation*}
instead the noise $W$ gives rise to a coercive-type term similar to the one appearing in the right-hand side of \eqref{eq:dissipation_measure}.
If $\mathcal{T}_{u^\nu,\theta^\nu}(t,\eps)$ is properly controlled by the Besov-type regularity of $\theta^\nu$ as $\eps\to 0^+$ (as encoded by condition \eqref{eq:key_bound_regularization} in \autoref{lem:gron_increments}), then the contribution from the noise dominates it, and so one can close an anomalous regularization estimate, resulting in uniform-in-$\nu$ bounds.
The case of general $\beta\in (0,1-\alpha]$ requires to track both contributions more precisely (introducing appropriate ``localized'' seminorms, cf. \eqref{eq:relevant_seminorms}, and a splitting argument, cf. \autoref{lem:data-splitting}), but follows the same spirit. 

We point out that the precise choice of the relation $u^\nu = \mathscr{R}\theta^\nu$ plays no crucial role in the argument, and in fact similar results can be obtained also for other constitutive laws (cf. \autoref{rem:gSQG_numerologies} below) and linear SPDEs with random drifts $b^\nu$ (cf. \autoref{thm:anomalous_linear}).

To better understand the parameters $(\alpha,\beta,p)$ appearing in \autoref{ass:exponent.beta}, let us introduce the following terminology, which is based on scaling arguments and \cite[Remark 1.7]{BaGa25}.

\begin{definition}
    Let $\alpha\in (0,1)$, $p\in [1,\infty]$, and let $\mathscr{R}$ be a singular integral operator associated to some $\mathfrak{h}$-homogeneous Fourier symbol $\widehat{\mathscr{R}}$, namely such that $\widehat{\mathscr{R}}(\lambda \xi) = \lambda^\mathfrak{h}\widehat{\mathscr{R}}(\xi)$ for every $\xi \in \R^2 \setminus \{0\}$ and $\lambda>0$. We say that that $(\alpha,p)$ belong to the \emph{scaling (sub)critical regime} if
    \begin{equation}\label{eq:subcritical_regime}
        \alpha \leq \frac{1-\mathfrak{h}}{2}-\frac{1}{p}.
    \end{equation}
    We say that $(\alpha,p)$ are \emph{critical} if equality holds in \eqref{eq:subcritical_regime}.
\end{definition}

For instance, in our setting one has $\mathfrak{h}(\mathscr{R}_{BS})=-1$ and $\mathfrak{h}(\mathscr{R}_{CZ})=0$, so that \eqref{eq:subcritical_regime} reads as
\begin{align*}
    \alpha \leq 1-\frac{1}{p},\quad \text{if } \mathscr{R}=\mathscr{R}_{BS}; 
    \qquad
    \alpha \leq \frac{1}{2}-\frac{1}{p},\quad \text{if } \mathscr{R}=\mathscr{R}_{CZ}.
\end{align*}
We see in particular that the (sub)critical range roughly corresponds to the first case in the respective definitions of $\beta$ given by \eqref{eq:definition.regularity}-\eqref{eq:definition.regularity_CZ}.
Let us also note that relation \eqref{eq:subcritical_regime} is non-empty for any value $\mathfrak{h}\in [-1,1)$.

\begin{rmk} \label{rmk:pathwise.uniquness}
As argued in \cite{BaGa25}, pathwise uniqueness for \eqref{eq:nonlinear_transport} is expected to hold if and only if condition \eqref{eq:subcritical_regime} holds.\footnote{There are by now several positive results in the (sub)critical range, see \cite{BaGa25} and the references therein for $\mathscr{R}=\mathscr{R}_{BS}$, as well as \cref{thm_pathwise_uniq_order_0} below for $\mathscr{R}=\mathscr{R}_{CZ}$; however, to the best of our knowledge, no explicit counterexample in the supercritical range is available yet.}
In light of this, \cref{thm:regularization} suggests that there is no direct connection between pathwise uniqueness for \eqref{eq:nonlinear_transport} and anomalous regularity of vanishing viscosity approximations \eqref{eq:nonlinear_viscous}.
\cref{thm:regularization} implies that vanishing viscosity solutions $\theta$ to \eqref{eq:nonlinear_transport} gain ``full'' Besov regularity of order $\beta=1-\alpha$ whenever $(\alpha,p)$ satisfy \eqref{eq:subcritical_regime}, but also for some scaling \emph{supercritical} parameters choices $(\alpha,p)$ (the second condition in \eqref{eq:definition.regularity} and \eqref{eq:definition.regularity_CZ}); on the other hand, there exist other supercritical regimes (the last condition in \eqref{eq:definition.regularity} and the last two conditions in \eqref{eq:definition.regularity_CZ}) where we only obtain some ``partial'' regularization with $\beta<1-\alpha$.
We do not know whether the latter is an artifact of the proof or is in fact sharp; we leave this question for future investigations.

This clear separation between regularization (and therefore existence) and uniqueness should not come as too surprising.
Already \cite[Theorems 2.11-2.12]{CoMa23} contained a similar distinction.
Another nonlinear setting displaying a similar distinction are the deterministic $2$D Euler equations: for any $\theta_0\in L^p_x$ with $p\in [1,\infty)$, solutions obtained by vanishing viscosity limits satisfy many nontrivial properties like renormalizability, Lagrangianity and compactness, see \cite{CiCrSp2021} and the references therein; yet these solutions are not expected to be unique (with explicit counterexamples in the presence of forcing, cf. \cite{vishik2018instability,vishik2018instabilityII,ABCDLGJK2024}) and the scaling critical case corresponds to $p=\infty$.
\end{rmk}

\subsection{Anomalous integrability}\label{subsec:intro_integrability}

Next, we focus on the problem of anomalous integrability. 
By \autoref{thm:regularization} and Sobolev embeddings, one can immediately obtain some additional integrability if $\beta>1-2/p$; bootstrapping this fact, we can obtain a much stronger result.

Let us briefly explain the heuristics behind this integrability gain, taking as prototypical example the Euler equations $\mathscr{R} = \mathscr{R}_{BS}$ with parameters $p \in [2,3)$ and $\alpha \leq 1-1/p$.
In this case, by \autoref{thm:regularization} we know that for every $\varepsilon>0$ and uniformly in $\nu \in (0,1)$, the viscous approximations $\{\theta^\nu\}_{\nu \in (0,1)}$ are uniformly bounded in
$$L^2_{\omega,T} H^{1-\alpha-\varepsilon}_x \hookrightarrow L^2_{\omega,T}  L^q_x\quad \text{for } q := \frac{2}{\alpha+\varepsilon},$$
and since $p<3$, $\alpha\leq 1-1/p$, we have $q>p$. Namely, starting from initial conditions $\theta_0^\nu \in L^p_x$ we end up with solutions $\theta^\nu_t(\omega) \in L^q_x$ for almost every $t \in [0,T]$ and $\omega \in \Omega$, with suitable bounds that are uniform in $\nu \in (0,1)$. At this point, we can iterate this procedure starting from $|\theta^\nu_t (\omega)|^{q/p} \in L^p_x$ and deduce $|\theta^\nu_{t'} (\omega')|^{q/p} \in L^q_x$ at a later time $t'$, or equivalently $\theta^\nu_{t'}(\omega') \in L^{q^2/p}_x$, and so on.

We make this heuristic quantitatively precise in \autoref{sec:integrability}, by keeping track of the time dependence of this integrability gain, as well as moments with respect to $\omega$.
In the actual argument, we use different integrability exponents and embeddings in the iteration in order to obtain sharper results, but the idea is the same explained above.
This strategy is reminiscent of the iteration performed in
\cite[Section 4.2]{DrGaPa25}.
However, the proof therein applied infinitely many interpolation bounds for the linear (possibly compressible) Kraichnan solution operator, which required careful control of the associated interpolation constants;
here, exploiting the incompressible transport structure instead
provides a more elementary argument, with direct control
of the constants.

Interestingly, this strategy does not work in the full generality of \autoref{ass:exponent.beta} and instead recovers (modulo additional technical conditions) the (sub)critical range \eqref{eq:subcritical_regime}:

\begin{ass}\label{ass:strong_regularization}
    $(\alpha,p,\mathscr{R})$ satisfy the (sub)critical condition \eqref{eq:subcritical_regime} and either:
    \begin{itemize}
        \item $\mathscr{R}=\mathscr{R}_{BS}$, $r\in (1,2)$ and $p \in [2,\infty)$;\footnote{\label{foot:3} Similarly to \cref{foot:2}, for 2D Euler, we could allow $p=\infty$ in \autoref{ass:strong_regularization} and we only exclude it for technical reasons. Indeed, since $\alpha<1$ by assumption, we can always find $\tilde{p}$ large enough such that $\alpha\leq 1-1/\tilde p$ and use that $L^r_x\cap L^\infty_x\subset L^r_x\cap L^{\tilde{p}}_x$.} or
        \item $\mathscr{R}=\mathscr{R}_{CZ}$, $r \in (1,2]$, $p \in (2,\infty)$.
    \end{itemize} 
\end{ass}

\begin{thm}[Anomalous integrability] \label{thm:integrability}
    Let $(\alpha,p,r,\mathscr{R})$ satisfy \autoref{ass:strong_regularization}.
    Let $\theta_0 \in L^r_x \cap L^p_x$, $f \in L^1_T(L^r_x \cap L^p_x)$ and set up a vanishing viscosity scheme for \eqref{eq:nonlinear_transport} that additionally satisfies
\begin{align} \label{eq:bound.forcing.anomalousintegrability}
\sup_{\nu \in (0,1)} \| \theta^\nu_0 \|_{L^p_x} =: M_\theta< \infty \quad \text{and} \quad
\sup_{\nu \in (0,1)} \| f^\nu \|_{L^1_T L^{p_\ast}_x} =: M_f< \infty
\quad
\mbox{ for some } p_\ast \in [p,\infty].
\end{align}
Then for every $t \in (0,1\wedge T]$ one has
\begin{equation}\label{eq:intro_anomalous_integrability}
  \sup_{\nu \in (0,1)}   \mathbb{E} \left[\sup_{s \in [ t ,T]}  \| {\theta}^{\nu}_s \|_{L^{p_\ast}_x}^p  \right]^{1/p}
   \lesssim 
   M_f+
   t ^{-\left( \frac{1}{1-\alpha} \right)\left(\frac{1}{p}-\frac{1}{p_\ast} \right)} M_\theta,
\end{equation}
where we adopt the convention $\frac{1}{\infty}=0$.
\end{thm}

\begin{rmk}\label{rmk:integrability_intro}
    A prototypical application of \autoref{thm:integrability} is the following: in the absence of forcing $f^\nu\equiv 0$, taking $p_\ast=\infty$ and a vanishing viscosity scheme that additionally satisfies \eqref{bound_viscoud_data} below (cf. \autoref{rmk:bound_viscoud_data}), one obtains
    \begin{align*}
        \sup_{\nu \in (0,1)} \EE[ \| \theta_t^\nu\|_{L^\infty_x}^p] \lesssim (1\wedge t)^{-\frac{1}{1-\alpha}} \| \theta_0\|_{L^p_x}^p \quad\forall t\in (0,T].
    \end{align*}
    For $t\ll 1$, this is the same asymptotic integrability gain satisfied by parabolic PDEs of the form $\partial_t v = -(-\Delta)^{1-\alpha}v$. Note that even in the case of linear transport SPDEs, the estimate above improves on \cite[Theorem 1.3, equation (1.8)]{DrGaPa25}, as the exponent $\frac{1}{1-\alpha}$ (with $d=2$) here is exact.
\end{rmk}

\begin{figure}[htbp]
    \centering
    \includegraphics[width=0.49\textwidth]{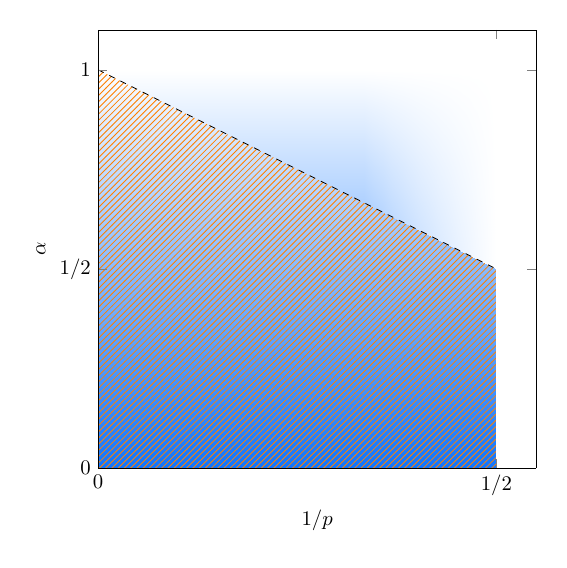}
    \includegraphics[width=0.49\textwidth]{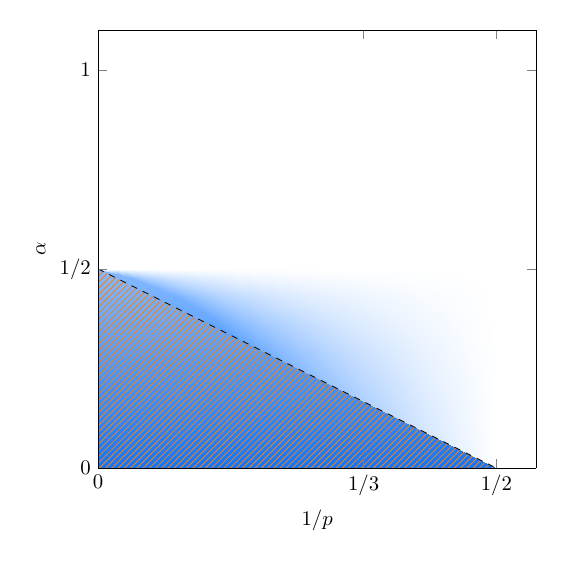}
    \caption{Hatched area: Region where \autoref{ass:strong_regularization} and anomalous integrability hold for $\mathscr{R}=\mathscr{R}_{BS}$ (left) and $\mathscr{R}=\mathscr{R}_{CZ}$ (right). The blue region, where \autoref{thm:regularization} applies, is as in \autoref{fig.regularization}. Notice that one can have anomalous regularization without anomalous integrability.}
\end{figure}

 \begin{rmk} 
It is worth comparing the setting of \cref{thm:integrability} with some of the supercritical regimes covered by \autoref{thm:regularization}, for instance the case
\begin{align*}
    \mathscr{R} =\mathscr{R}_{BS},\quad r<2,\quad p \in (2,3), \quad 1-\frac{1}{p} < \alpha < 1.
\end{align*}
Here, by Sobolev embeddings one has
\begin{equation*}
    \theta^\nu \in L^2_{\omega,T} H^{(1-\alpha)(p-2)-\eps}_x \hookrightarrow L^2_{\omega,T} L^q_x\quad\text{for}\quad q := \frac{2}{1-(1-\alpha)(p-2)+\varepsilon}<p.
\end{equation*}
That is, we do not gain any higher integrability than the one required on the initial condition $\theta_0$ in order to obtain anomalous regularization in the first place. The same dichotomy appears in the case $\mathscr{R}=\mathscr{R}_{CZ}$.
\end{rmk}

\begin{rmk} \label{rmk:anisotropy}
    \cref{thm:regularization} and \cref{thm:integrability} hold for a more general class of noises, see \cref{ass:noise} below, rather than just incompressible Kraichnan (namely $W$ with covariance $C=C_\alpha$ as in \eqref{eq:isotropic.covariance}).
    \cref{ass:noise} encodes some form of \emph{local isotropy} of the noise covariance around the origin, which is stable under more regular, possibly \emph{anisotropic} perturbations; see the discussion in \cref{ex:ass_noise,ex:ass_noise2}.
    In particular, our results readapt verbatim to the same SPDEs posed on the torus $\T^2$, conditionally on verifying \eqref{eq:local.expansion.noise}.
    
    Requiring exact isotropy of the noise was a major limitation of the approach from \cite{GaGrMa24,DrGaPa25} and earlier \cite{le2002integration}.
    Any kind of isotropy assumption, either exact or local, was recently overcome by Rowan in \cite{Rowan25}, with a different technique that covers a larger class of ``genuinely anisotropic'' noises.
    Concerning the application of Rowan's approach to nonlinear SPDEs on the torus, we refer to \cite{luotang2026}.
\end{rmk}

Similarly to the case of anomalous regularization, the mechanism underlying anomalous integrability is formulated more abstractly in
\autoref{subsec:integrability_strategy}.
It applies to a class of transport
SPDEs with random divergence-free drifts (see equations \eqref{eq:abstract_transport_SPDE}-\eqref{eq:general.viscous.SPDE}), i.e. without
requiring a constitutive relation between the drift $b$ and the advected quantity $\theta$, the latter being essentially treated as a ``passive'' scalar.
In particular, \autoref{thm:anomalous_linear} provides
sufficient conditions on such drifts to establish anomalous regularization and integrability for the advected scalar, in any dimension $d\geq 2$.

\subsection{Existence and uniqueness for inviscid SPDEs}\label{subsec:intro_existence_uniqueness}

Let us turn our attention back to the inviscid problem \eqref{eq:nonlinear_transport}.
Recall that the last part of \cref{thm:regularization} provides the existence of probabilistically weak solutions to \eqref{eq:nonlinear_transport} satisfying the anomalous regularity $\tilde{L}^2_{\omega,T} \tilde{B}^\beta_{2,\infty}$, obtained as limits in law (along subsequences) of vanishing viscosity schemes.

On the other hand, under the assumptions of \autoref{thm:integrability}, we also have the uniform-in-$\nu$ estimate \eqref{eq:intro_anomalous_integrability};
we can make the constants $M_f$, $M_\theta$ explicit by choosing ad hoc vanishing viscosity approximations, similarly to \autoref{rmk:bound_viscoud_data}.
Lower semicontinuity of $L^{p_\ast}_x$-norms and an application of the Portmanteau theorem then imply the following:

\begin{cor}\label{cor:existence}
Under the assumptions of \cref{thm:regularization}, there exists a probabilistically weak solution of \eqref{eq:nonlinear_transport} in the sense of \cref{martingale_sol}, obtained as  vanishing viscosity scheme and such that
\begin{align*}
\theta \in  L^{\infty}_{\omega,T}(L^r_x\cap L^p_x)\cap\Tilde{L}^2_{\omega,T}\Tilde{B}^{\beta}_{2,\infty}   
\end{align*}
with almost surely weakly continuous paths in $L^2_x$.
If additionally the assumptions of \cref{thm:integrability} hold, then the solution can be constructed so that it also satisfies
\begin{equation}\label{eq:corollary_anomalous_integrability}
   \mathbb{E} \left[\sup_{s \in [ t ,T]}  \| {\theta}_s \|_{L^{p_\ast}_x}^p  \right]^{1/p}
   \lesssim 
   \| f\|_{L^1_T L^{p_\ast}_x}+
   t ^{-\left( \frac{1}{1-\alpha} \right)\left(\frac{1}{p}-\frac{1}{p_\ast} \right)} \| {\theta}_{0} \|_{L^{p}_x},
   \quad
   \forall t \in (0,1\wedge T].
\end{equation}
\end{cor}

\autoref{cor:existence} solely relies on a priori estimates and tightness arguments, which do not require any assumption of uniqueness of solution to the limit equation \eqref{eq:nonlinear_transport}.
On the other hand, \autoref{thm:integrability} applies under \autoref{ass:strong_regularization}, and so in particular in the (sub)critical regime \eqref{eq:subcritical_regime}; here, for $\mathscr{R}=\mathscr{R}_{BS}$, uniqueness holds by \cite{BaGa25} under the additional assumption $\theta_0\in \dot H^{-1}_x$. We complement that work by replacing the assumption $\theta_0\in \dot H^{-1}_x$ with $\theta_0\in L^r_x$, and  providing a similar well-posedness result in the case of zero-th order operators; as a result, we obtain the following comprehensive statement.

\begin{thm}\label{thm:intro_wellposedness}
    Let $(\alpha,p,r,\mathscr{R})$ satisfy \autoref{ass:strong_regularization}.
    Then for any $\theta_0\in L^r_x\cap L^p_x$ and deterministic forcing $f\in L^1_T (L^r_x\cap L^{p}_x)$, strong existence and pathwise uniqueness hold for solutions of \eqref{eq:nonlinear_transport}, in the class $\theta\in L^\infty_{\omega,T}(L^r_x\cap L^p_x)\cap \mathcal{U}_p$.
    For $\mathscr{R}=\mathscr{R}_{BS}$, and in strictly subcritical regimes for $\mathscr{R}=\mathscr{R}_{CZ}$, pathwise uniqueness holds in $L^\infty_{\omega,T}(L^r_x\cap L^p_x)$, too.
    Moreover, the unique solution is a Markov process, it is the limit in probability (in appropriate topologies) of the solutions $\{\theta^\nu\}_{\nu \in (0,1)} $ associated to any vanishing viscosity scheme $\{(\theta^\nu_0,f^\nu,\theta^\nu)\}_{\nu \in (0,1)}$ of \eqref{eq:nonlinear_transport} in the sense of \autoref{defn:vanishing_scheme}, and it satisfies the properties of \autoref{cor:existence}.
\end{thm}

The proof will be presented in \autoref{sec:existence_uniqueness}.
Some finer details in the statement of \autoref{thm:intro_wellposedness} are left vague on purpose to avoid technicalities.
The spaces $\mathcal{U}_p$, only needed for pathwise uniqueness when $\mathscr{R}=\mathscr{R}_{CZ}$ and  critical $(\alpha,p)$, are slightly technical and we postpone their definition to \autoref{sec:existence_uniqueness}; nevertheless, in most cases, we obtain pathwise uniqueness in the natural class $L^\infty_{\omega,T}(L^r_x\cap L^p_x)$.
Convergence in probability is again in the space $\mathcal{E}$ defined in \eqref{polish_space_tightness}.
Let us also mention that the statements holds for a slightly larger class of noises $W$ than just Kraichnan, cf. \autoref{ass:noise2}.

\begin{rmk}
\autoref{cor:existence} and \autoref{thm:intro_wellposedness} allow to construct global-in-time, Sobolev regular solutions of the inviscid stochastic problem \eqref{eq:nonlinear_transport} via vanishing viscosity approximations. 
Similar results are notoriously difficult to achieve in the deterministic case, since equations of this form are often ill-posed in Sobolev spaces of low regularity \cite{cordoba2021non, bianchini2025non, cordoba2025instantaneous} and high regularity norms may blow-up in finite time \cite{KiselevSQG,cordoba2024finite}. 
Let us also note that previous results on restored well-posedness by Kraichnan noise for nonlinear active scalar models, cf. \cite{CoMa23,jiao2025well,bagnara2025regularization,BaGa25}, were not able to treat the case of Calderón--Zygmund operators, which we cover here.
\end{rmk}

\begin{rmk}[Further extensions]\label{rem:gSQG_numerologies}
    In this paper we restricted ourselves to the paradigmatic examples $\mathscr{R}=\mathscr{R}_{BS}$ and $\mathscr{R}=\mathscr{R}_{CZ}$, but it is clear that the same approach readapts to larger families of nonlinear SPDEs, up to properly modifying the relevant numerologies. For instance, one can treat the generalized SQG equations, $\mathscr{R} =-\nabla^{\perp}(-\Delta)^{-\frac{1-\mathfrak{h}}{2}}$ with $\mathfrak{h}\in (-1,0)$ and compute the appropriate values of $\beta\in (0,1-\alpha]$ similarly to what is done in \autoref{sec:computation}.
    In particular, by \eqref{eq:formula_beta_intermediate}, we expect the existence of supercritical ranges of parameters where anomalous regularization holds with exponent
    \begin{align*}
        \beta=(p-2)\left(\frac{1-\mathfrak{h}}{2}-\alpha \right).
    \end{align*}
    Moreover in this case the (sub)critical regime is exactly given by \eqref{eq:subcritical_regime}, so that \autoref{thm:regularization} (with $\beta=1-\alpha$) and \autoref{thm:integrability} are still expected to hold under an analogue of \autoref{ass:strong_regularization}.
    It would be interesting to understand whether the same phenomenologies extend to the more singular case $\mathfrak{h}\in (0,1)$; we leave this problem for future investigations.
\end{rmk}

\subsection{Anomalous dissipation}

Let us finally discuss anomalous dissipation and consequently sharpness of the regularization entailed by \cref{thm:regularization}. In the linear case $\mathscr{R}=0$ and $f^\nu \equiv 0$, it was shown in \cite[Section 3]{DrGaPa25} that the regularity threshold $1-\alpha$ is optimal; the argument is by contradiction, combining the the fact that any non-zero solution diplays anomalous dissipation of energy with a commutator argument à la Constantin--E--Titi \cite{CET1994}.
Loosely speaking, the latter show that, if the solution $\theta$ belongs to $L^2_{\omega,T}H^{1-\alpha}_x$, then anomalous dissipation cannot take place.
It turns out that this implication remains true in the nonlinear case, at least in certain regimes of parameters, see \cref{prop:dissipation.implies.irregularity}. 

For the Euler equations, in a suitable range of parameters, we are able to prove that anomalous dissipation of mean enstrophy happens for every non-zero initial condition; in view of the discussion above, this implying optimality of the regularization from \cref{thm:regularization}.
\begin{thm}[Enstrophy dissipation for stochastic Euler] \label{thm_anomalous_dissipation}
Let $\mathscr{R}=\mathscr{R}_{BS}$ and assume that
\begin{equation}\label{eq:parameters_anomalous_dissipation}
    \alpha\in (0,1/2), \quad r\in (1,2), \quad p\in [2,+\infty], \quad \theta_0 \in L^r_x\cap L^p_x\cap\dot{H}^{-1}_x \setminus \{0\}, \quad f \equiv 0;
\end{equation}
consider the unique solution $\theta$ to \eqref{eq:nonlinear_transport} given by \autoref{thm:intro_wellposedness}.
Then, enstrophy dissipation happens continuously in time on average:
\begin{align}\label{eq:an_diss_euler}
    \lim_{s\rightarrow t}\mathbb{E}\left[\norm{\theta_t-\theta_s}_{L^2_x}^2\right]=0,
    \quad \text{and}\quad
    \mathbb{E}\left[\norm{\theta_t}^2_{L^2_x}\right]<\mathbb{E}\left[\norm{\theta_s}^2_{L^2_x}\right]\quad \text{for all } 0\leq s<t\leq T.
\end{align}
As a consequence, regularity $\theta\in \Tilde{L}^2_{\omega,T}\Tilde{B}^{1-\alpha}_{2,\infty}$ is optimal, in the sense that for every $q \in [1,\infty)$ and any nonempty subinterval $(t_0,t_1)\subset [0,T]$ it holds $\theta \notin L^2_{\omega,[t_0,t_1]} B^{1-\alpha}_{2,q}$.
\end{thm}

We point out that the conclusions of \autoref{thm_anomalous_dissipation}
apply even to smooth initial data $\theta_0$; this is in stark contrast with the unforced deterministic two-dimensional Euler equations, in which case spatial smoothness of solutions and enstrophy conservation hold.

\subsection{Related literature}\label{subsec:intro_literature}

\subsubsection{State of the art for the deterministic PDEs}
The literature on the deterministic $2$D Euler equations in vorticity form is extremely vast and we will not attempt to entirely cover it here; concerning well-posedness and ill-posedness results for these equations, we refer the interested reader to the overviews given in \cite{ABCDLGJK2024,CoMa23,GalLuo25,BCK2026}.
Let us instead discuss more in detail the SQG and IPM equations.

The rigorous mathematical analysis of the SQG system started with \cite{CoMaTa1994}, due to its relevance for the formation of strong fronts in atmospheric flows and its mathematical analogy with $3$D Euler; therein local well-posedness of regular solutions was shown. 
Global existence of weak $L^2_x$-solutions was obtained by Resnick \cite{Resnick1995} and refined by Marchand \cite{marchand2008existence} for $L^p_x$ with $p>4/3$. Recently the endpoint value $p=4/3$ has been achieved in \cite{DRLP2026} and refined to Lorentz spaces in \cite{constantin2026weak}. The most general result on weak existence for $L^p_x$-valued weak solutions to the generalized SQG system, for any exponent $\mathfrak{h}\in (-1,1)$ (as defined in \autoref{rem:gSQG_numerologies}) is currently \cite{de2026hamiltonian}.

In comparison, first results about local well-posedness and blow-up criteria for regular solutions of IPM were obtained in \cite{CGO2007}, see also \cite[Section 2]{KisYao2023}; for the derivation of this PDE via Darcy's law and a discussion relevant applications we refer e.g. to the monograph \cite{Bear1972}.

The fundamental difference between SQG and IPM, as pointed out in \cite{shvydkoy2011convex,IsVi2015}, is that in the first case $\mathscr{R}=\nabla^\perp (-\Delta)^{-1/2}$ has odd Fourier symbol, while $\mathscr{R}=\nabla^\perp (-\Delta)^{-1} \partial_1$ is even. This makes it much simpler to construct weak solutions to SQG via vanishing viscosity and compactness arguments, which have not been successfully implemented for IPM so far. Conversely, convex integration schemes work much more easily for IPM, with first global existence and non-uniqueness of $L^2_x\cap L^\infty_x$-valued solutions obtained in \cite{cordoba2011lack}. For the same reason, the first successful implementation of convex integration to SQG from \cite{buckmaster2019nonuniqueness} required to rephrase the system in its momentum form; this has been bypassed in the more recent refinements \cite{isett2021direct} and then \cite{brue2026flexibility}, where non-uniqueness of $L^p_x$-valued solutions with $p>4/3$ is shown.

Overall, the current state of the art for IPM and SQG is in many aspects similar. Small scale formations for large times have been shown for both systems respectively in \cite{KisYao2023,KiselevSQG}.
The techniques developed by C{\'o}rdoba and Mart{\'\i}nez-Zoroa have also been applied to both, yielding a variety of non-existence, strong ill-posedness and finite-time singularity results (possibly in the presence of very regular forcing), see \cite{cordoba2021non,cordoba2024finite,cordoba2025instantaneous,bianchini2025non,ABC2026}. 
Let us also mention the recent work \cite{dembski2025singularity} on finite-time singularity formation for IPM, on wedge domains arbitrarily close to the half place, for Lipschitz continuous and compactly supported initial data.

\subsubsection{Phenomenologies of turbulence and related models}

As already mentioned, the Kraichnan model was first introduced in \cite{Kraichnan1968} in the '60s.
It witnessed a renewed interest in the physics community in the late '90s and early '00s; it is one of the few available models which displays \emph{anomalous dissipation}, \emph{intermittency} and Lagrangian \emph{spontaneous stochasticity}. The work \cite{GawVerg00} further generalized it and observed a phase transition depending on the compressibility of the noise; instead intermittency of solutions was first discussed in \cite{bernard1998slow}. See also the monograph \cite{CFG2008}.
A more rigorous mathematical analysis was then conducted by E--Vanden-Eijnden \cite{EVan2000,EVan2001} and most prominently LeJan--Raimond \cite{le2002integration,LeJRai2004}.
This is also the first model where an explicit link between anomalous dissipation and Lagrangian spontaneous stochasticity was established; the same principle holds in much higher generality, see \cite{drivas2017lagrangian,Rowan2024}.

Recently there have been numerous relevant contributions for this model. Rowan \cite{Rowan2024} gave a proof of anomalous dissipation on $\T^d$ and established exponential dissipation of energy.
Building on the intuition first developed in \cite{CoMa23} for the $\dot H^{-1}_x$-norm, \cite{GaGrMa24} gave a first comprehensive result of \emph{anomalous regularization} of solutions, in the sense that initial data instantaneously gain higher Sobolev regularity under the evolution of the linear SPDE, up to $H^{1-\alpha-}_x$; the proof therein was based on the study of the Fourier spectrum by means of radial symmetries and Mellin transform.
The work \cite{crippa2025zero} extended this result to allow the presence of an additional Sobolev drift $b$, established the existence of \emph{dissipation measures} associated to the unique solutions and studied the associated vanishing noise limit; instead \cite{bagnara2024anomalous} analyzed similar phenomena in the case of vector advection (the so called Kasantzev--Kraichnan model).
Another proof of anomalous regularization was given in \cite{DrGaPa25}, based on the study of the PDE satisfied by the $2$-point self-correlation function.
This was the also the first work to establish the borderline regularity $\tilde L^2_{\omega,T} \tilde B^{1-\alpha}_{2,\infty}$, prove its sharpness and link it to the dissipation measure, as well as deriving anomalous integrability as a byproduct of anomalous regularization. However, the arguments from \cite{DrGaPa25} still crucially rely on isotropy of the noise and hold on $\R^d$.
The next step was accomplished by Rowan \cite{Rowan25}, analyzing the evolution of the energy spectrum using lattice Poincarè inequalities (building on \cite{luo2024elementary}) and extrapolation techniques. This approach allows for a large class of genuinely anisotropic noises and in some cases even provide infinite smoothing in some directions, cf. \cite[Corollary 1.7]{Rowan25}.

Another celebrated model displaying similar features to the ones studied here are the 1D Burgers equations $\partial_t u + u \partial_x u =0$, and more generally conservation laws \cite{Dafermos2000}. It is well-known (see e.g. \cite{DLOW2004}) that for $u_0\in L^\infty_x$, entropy solutions are unique, coincide with vanishing viscosity limits and instantaneous become $BV$-regular at positive times; on the other hand, classical solutions can cease to exist in finite time due to the formation of shocks, which result in loss of energy. The link between anomalous dissipation and spontaneous stochasticity for this system has been discussed in \cite{eyink2015spontaneous}. General dispersive estimates akin to our \autoref{thm:integrability} for scalar conservations laws were obtained in \cite{SeSi2019}.

Starting with \cite{DrElIyJe22}, many works have recently constructed examples of \emph{deterministic} vector fields inducing anomalous dissipation of energy on passive scalars, see e.g. \cite{DrElIyJe22, CoCrSo23,ArVi23+,BuSzWu23+,ElLi23+,JoSo24+,hess2025universal}.
Moreover, in the work \cite{hess2025turbulent}, it is presented a drift whose associated passive scalars display \emph{total} dissipation of energy, intermittency and partial anomalous regularization.

Let us finally mention \cite{bagnara2026regularity}. It was noticed therein that in $2$D, the autonomous velocity fields obtained as a.s. realization of $W_1(\omega)$ satisfy the weak Sard property, and therefore passive scalars advected by them must conserve energy, in stark contrast with the non-autonomous case.

\subsubsection{Transport noise and its regularizing effect}

The use of transport noise in fluid dynamics equations has been advocated by many authors, with a variety of arguments ranging from stochastic Lagrangian characteristics to variational principles; it is also of great relevance in applications like Stochastic Parametrization, Data Assimilation and Uncertainty Quantification. Without aiming to be comprehensive, we refer the reader to \cite{BCF1991,Holm2015,Memin2014,CHLN2022,DebMem2025}.
When the noise is sufficiently regular, rigorous derivations of such SPDEs in an ideal limit of multiscale systems can be found in \cite{FlaPap2022,DebPap2026}.

The idea that transport noise, modelling the influence of the small turbulent scales on larger ones, may stabilize the nonlinear dynamics and restore its well-posedness, has been notably advocated by Flandoli \cite{Flandoli2015}.
In the ever growing field of \emph{regularization by noise}, let us only mention the contributions closest to the setting of \eqref{eq:nonlinear_transport}.
The first work understanding some partial anomalous regularization features of rough Kraichnan noise and applying it the the 2D Euler equations was \cite{CoMa23}, with refinements in \cite{JiaLuo2025}; similar ideas have then been successfully applied to gSQG  \cite{bagnara2025regularization,jiao2025well,BaGa25}, linear transport with drift \cite{crippa2025zero} and 2D Boussinesq with thermal diffusivity \cite{JiaLuo2026}.

While completing this work, we were informed of the ongoing project \cite{luotang2026} by Luo and Tang. Therein the authors similarly study the problems anomalous regularization and anomalous dissipation of solutions to nonlinear 2D fluid dynamics SPDEs with rough transport noise; differently from our case, they work on the torus $\T^2$, adopting an approach closer to the one developed in the linear setting by Rowan \cite{Rowan25}.

\subsection{Structure of the paper}
We start by recalling in \autoref{sec:preliminaries} useful preliminaries concerning the noise $W$, the SPDEs \eqref{eq:nonlinear_transport}-\eqref{eq:nonlinear_viscous} and the function spaces we use.
Then in \autoref{sec:general.criterion} we develop a general criterion for anomalous regularization, cf. \autoref{lem:gron_increments}; we verify its applicability to \eqref{eq:nonlinear_transport} under \autoref{ass:exponent.beta} in \autoref{sec:computation}, cf. \autoref{prop_cases}.
\autoref{sec:integrability} is devoted to the proof of \autoref{thm:integrability}, as well as the more abstract \autoref{thm:anomalous_linear}, valid for transport SPDEs with random drifts in any dimension.
In the first part of \autoref{sec:existence_uniqueness}, we present the tightness arguments needed to complete the proofs of \autoref{thm:regularization} and \autoref{cor:existence}; in the second part we establish novel pathwise uniqueness results for \eqref{eq:nonlinear_transport} and altogether prove \autoref{thm:intro_wellposedness}.
\autoref{Appendix:approx_kernels} collects some key technical lemmas concerning the flux functions associated to the noise $W$ and (conditional variants of) the Besov-type spaces $\tilde{L}^2_{\omega,T}\tilde{B}^{\beta}_{2,\infty}$ used throughout the paper.

\subsection*{AI Disclosure}
The authors used GPT-6 Astra, Extra High solely for copyediting, typing and proofreading.
All the mathematical results in the paper were obtained by the authors without AI assistance.
The authors reviewed and edited all generated text and accept full responsibility for the final content.

\subsection*{Acknowledgements}
LG acknowledges support from the Istituto Nazionale di Alta Matematica (INdAM) through the project GNAMPA 2026 ``Fluidodinamica stocastica: irregolarità, trasporto e fenomeni di regolarizzazione''.
This article is published with funding from Università degli Studi dell’Aquila under the Call for Proposals for Fundamental research and Early-career research grants - Year 2026. (Italian version: L’articolo
è pubblicato con il contributo dell’Università degli Studi dell’Aquila nell’ambito dell’Avviso per la presentazione di Progetti di Ateneo per la Ricerca di base e Avvio alla Ricerca - anno 2026).
EL and UP have received funding from the European Research Council (ERC) under the European Union’s Horizon 2020 research and innovation programme (grant agreement No. 949981). 
EL has received funding from the Deutsche Forschungsgemeinschaft (DFG, German Research Foundation) under Germany´s Excellence Strategy – The Berlin Mathematics Research Center MATH+ (EXC-2046/2, project ID: 390685689).
UP has received funding from the Swiss National Science Foundation under the SNSF Ambizione grant No. 233216.
LG acknowledges the hospitality from Universit\"at Bielefeld, for a visit in November 2025 during which part of this project was developed, funded by the aforementioned ERC grant agreement No. 949981.

\section{Notation and Preliminaries}\label{sec:preliminaries}

\subsection{Notation}\label{subsec:notation}
We provide here for the reader's convenience the main conventions we will adopt throughout the paper.

Whenever performing estimates, given $a,b\in [0,+\infty)$, we write $a\lesssim b$ to mean that there exists $C>0$ such that $a\leq C b$; we write $a\lesssim_\lambda b$ if we want to stress the dependence of $C$ on the set of parameters $\lambda$. Notation $a\sim b$ stands for $a\lesssim b$ and $b\lesssim a$.

We always work on $\R^d$, $d\geq 2$; $v\cdot w$ denotes scalar product on $\R^d$.

We will adopt several shorthand notations for functions spaces. $L^p_x=L^p(\R^d;\R^N)$ denote Lebesgue spaces with norm $\|\cdot\|_{L^p_x}$, where $N$ may vary depending on the context. Similarly, for $s>0$, $W^{s,p}_x=W^{s,p}_x(\R^d;\R^N)$ denote fractional Sobolev spaces, as defined via Gagliardo--Niremberg seminorms whenever $s$ is not an integer. $H^s(\R^d;\R^N)=H^s_x$ denote inhomogeneous fractional Sobolev spaces as defined via Fourier trasform, for $s\in\R$; recall that $H^s_x=W^{s,2}_x$ for $s>0$.
For $\gamma>0$, $C^\gamma_x=C^\gamma(\R^d;\R^N)$ denote inhomogeneous H\"older spaces.
Similar conventions apply to their homogeneous counterparts $\dot W^{s,p}_x$, $\dot H^s_x$, and $\dot C^\gamma_x$.

We will often use $\llbracket\cdot \rrbracket$ to denote seminorms associated with some of these functions spaces, which may correspond to the norms of their homogeneous counterparts. For instance, for $s\in (0,1)$,
\begin{align*}
    \llbracket f\rrbracket_{W^{s,p}_x}  
    := 
    \| f\|_{\dot W^{s,p}_x}
    = \left( \int_{\R^d\times \R^d} \frac{|f(x)-f(y)|^p}{|x-y|^{d+sp}} \right)^{\frac{1}{p}}
, \quad
    \llbracket f\rrbracket_{C^s_x} := \| f\|_{\dot C^s_x}
    = \sup_{x\neq y} \frac{|f(x)-f(y)|}{|x-y|^s},
\end{align*}
and we set $\| f\|_{\dot W^{1,p}_x}:=\| \nabla f\|_{L^p_x},\ \| f\|_{\dot W^{0,p}_x}:=\|  f\|_{L^p_x}$.
For any $h\in\R^d$, we adopt the incremental notation $\delta_h f(x):=f(x+h)-f(x)$; with this convention, for $s\in [0,1]$ and $p\in (1,\infty)$, it is well-known that
\begin{equation}\label{eq:basic_increments}
    \| \delta_h f\|_{L^p_x}=\| f(\cdot+h)-f\|_{L^p_x} \lesssim_{d,s,p} \| f\|_{\dot W^{s,p}_x} |h|^s,
    \quad\forall h\in \R^d.
\end{equation}
We also use $\langle f,g\rangle$ to denote duality pairings between functions (or possibly distributions) of the space variable. This applies e.g. to $f,g\in L^2_x$, $f\in L^p_x$ and $g\in L^{p'}_x$, but also $f\in \mathcal{S}$ and $g\in \mathcal{S}'$ (denoting respectively Schwartz functions and tempered distributions), etc.
We will replace the subscript $x$ by ${\rm loc}$ whenever we want to consider local spaces, e.g. $L^p_{{\rm loc}}$. $\hat f$ denotes the Fourier transform of $f$, whenever $f\in \mathcal{S}'$, and the variable $\xi\in\R^d$ is often used for Fourier modes.

Given two Banach spaces $E_1$, $E_2$, we denote by $E_1\cap E_2$ their intersection, as a Banach space with norm $\| \cdot\|_{E_1\cap E_2}=\| \cdot\|_{E_1}+\| \cdot\|_{E_2}$.

We always work on an interval $[0,T]$.
Given a Banach space $E$ and $[t_0,t_1]\subset [0,T]$, we use $L^q_{[t_0,t_1]} E$ as a shorthand for the Lebesgue--Bochner space $L^q([t_0,t_1];E)$; when $[t_0,t_1]=[0,T]$, we simply write $L^q_T E$.
Similarly, given a reference probability space $(\Omega,\mathcal{F},\PP)$, $L^m_\omega E=L^m(\Omega,\mathcal{F},\PP;E)$ denotes the corresponding Lebesgue--Bochner space.

These notations can be chained, e.g. we can write $L^2_\omega L^q_T H^s_x$ as a shorthand for $$L^2(\Omega,\mathcal{F},\PP; L^q([0,T];H^s(\R^d;\R^N))),$$ and so on.
Whenever some exponents coincide, we may contract them, so that e.g. $L^2_{\omega,T} L^p_x=L^2_\omega L^2_T L^p_x$ and $L^2_{\omega} L^q_{T,x}=L^2_\omega L^q_T L^q_x$.

Given a Polish space $F$ and $\gamma\in (0,1)$, we write $C^\gamma_T F$ for the space of $\gamma$-H\"older continuous, $F$-valued functions $C^\gamma([0,T];F)$.

\subsection{Structure of the noise}\label{subsec:structure_noise}

We work on a given filtered probability space $(\Omega,\mathcal{F},(\mathcal{F}_t)_{t\geq 0},\mathbb{P})$ satisfying the usual conditions. 
We set ourselves on $\R^d$, $d\geq 2$. 
The driving noise $W$ in the SPDEs we consider (e.g. \eqref{eq:nonlinear_transport}) is assumed to be a Brownian-in-time, colored-in-space homogeneous Gaussian process defined on $(\Omega,\mathcal{F},(\mathcal{F}_t)_{t\geq 0},\mathbb{P})$; hence its covariance matrix $C:\R^d\to\R^{d\times d}$ satisfies 
\begin{align}
    \mathbb{E}[W_t(x) \otimes W_s(y)] =: (t \wedge s) C(x-y),
    \quad
    t,s \in [0,T], \quad
    x,y \in \R^d.
\end{align}
Throughout the paper we will assume that
\begin{align} \label{eq:properties.C}
    \nabla \cdot C = 0
    \quad \text{and}
   \quad \hat{C} \in L^1_\xi \cap L^\infty_\xi,
\end{align}
where $\hat C$ denotes the Fourier transform of $C$. This implies in particular that $C$ is continuous and
\begin{equation}\label{eq:cancellation_fourier_covariance}
    \nabla\cdot W \equiv 0,\quad 
    \xi\cdot \hat C(\xi) \xi=0 \quad \text{for Lebesgue a.e. }\xi\in\R^d.
\end{equation}
Note that $C$ and $\hat C$ are both even, $\R^{d\times d}$-valued functions.
By \eqref{eq:properties.C} and \cite[Section 2.1]{GalLuo25}, one can represent $W$ by the series expansion
\begin{equation}\label{eq:noise_series_expansion}
    W_t(x) = \sum_{k\in \N} \sigma_k(x) W^k_t
\end{equation}
where $\{\sigma_k\}_k$ is a family of smooth, divergence-free vector fields such that
\begin{equation}\label{eq:covariance_series_expansion}
    C(x-y)=\sum_{k\in\N} \sigma_k(x)\otimes \sigma_k(y),
\end{equation}
and convergence in \eqref{eq:covariance_series_expansion} in uniform on compact sets. Representations \eqref{eq:noise_series_expansion}-\eqref{eq:covariance_series_expansion} are often useful in practical computations and we will use them systematically throughout the paper, see e.g. \autoref{sec:general.criterion}.

The majority of our results hold under the following key condition on $W$.

\begin{ass}\label{ass:noise}
    The noise covariance $C$ satisfies \eqref{eq:properties.C}.    
    Letting $Q(z):=C(0)-C(z)$, there exist $\alpha\in (0,1)$ and some constants $c_1$, $c_2$, $r_0>0$ such that
    \begin{align} \label{eq:local.expansion.noise}
        \left|Q(z)- c_1 | z |^{2 \alpha} \left[ P^{\|}_z + \left( 1 + \frac{2 \alpha}{d -
  1} \right) P^{\perp}_z \right]\right| \leq c_2 | z |^2, \quad \text{ for } | z | \leq r_0.
    \end{align}
\end{ass}
In the line above, we denoted by $P^{\|}_z := \hat{z} \otimes \hat{z}$ and $P^{\perp}_z := Id-\hat{z} \otimes \hat{z}$ respectively the longitudinal and perpendicular projections over the vector $\hat{z} := z/|z|$ for $z \neq 0$, and $Q(0)=0$.

\begin{example}[Kraichnan noise]\label{ex:kraichnan}
    Fix $\alpha\in (0,1)$ and let $C=C_\alpha$ be defined via its Fourier transform by
\begin{align}\label{eq:isotropic.covariance}
    \hat C_\alpha(\xi) = \frac{1}{\langle\xi\rangle^{d+2\alpha}} P^\perp_\xi,
    \quad
    \mbox{for some } \alpha \in (0,1).
\end{align}
Above we denoted by $\langle\xi\rangle$ the Japanese bracket, i.e. $\langle\xi\rangle=(1+|\xi|^2)^{1/2}$.
Up to multiplicative constants, this is the noise first introduced in \cite{Kraichnan1968}; \autoref{ass:noise} is satisfied by \cite[Appendix A]{DrGaPa25}.
\end{example}

\begin{example}\label{ex:ass_noise}
    Let $C$ satisfy \autoref{ass:noise} and consider $\tilde C= C+R$, where $R$ is another covariance satisfying \eqref{eq:properties.C} and such that $R\in W^{2,\infty}_x$.
    By Taylor expanding $R$ around the origin, using its evenness, it is easy to see that $\tilde C$ still satisfies \autoref{ass:noise}, for the same $\alpha$, $r_0$ and $\tilde c_1=c_2$, $\tilde c_2=c_2 + \| D^2 R\|_{L^\infty_x}$. 
    This applies in particular whenever the noise $W$ is of the form $W=W^1+W^2$, where $W^1$ has covariance $C$ and $W^2$ is Lipschitz regular.\footnote{See more generally \cite[Proposition 2.6]{GalLuo25} for similar regularity counting arguments between a given noise and its covariance.}
\end{example}

\begin{example}\label{ex:ass_noise2}
    Let $C_\alpha$ be as in \eqref{eq:isotropic.covariance}, $R>0$ fixed and consider the covariance $C$ given by
    \begin{align*}
        \hat C(\xi):= \hat C_\alpha(\xi) \mathbf{1}_{\{|\xi|>R\}};
    \end{align*}
    then \autoref{ass:noise} holds. Indeed, $R=C-C_\alpha$ has compactly supported Fourier transform and is therefore smooth by Bernstein's lemma, so we are in the setting of \autoref{ex:ass_noise}. The same applies to any other choice of perturbation of $C_\alpha$ which only modifies its Fourier transform on a bounded set.
\end{example}

\subsection{Nonlocal operators $\mathscr{R}$}\label{subsec:nonlocal_operators}
Recall that for us $\mathscr{R}$ is always a singular integral operator of convolutional type, associated to some Fourier symbol $\widehat {\mathscr{R}}$.
We collect here useful properties of such operators.
First of all, due to its convolutional structure, $\mathscr{R}$ commutes with translations; in particular, whenever $\mathscr{R} f$ is well-defined, it holds
\begin{equation}\label{eq:R_commute_translation}
    \mathscr{R}(\delta_z f)= \delta_z \mathscr{R} f,\quad \forall z\in\R^2.
\end{equation}

In the case $\mathscr{R}=\mathscr{R}_{CZ}$, by Caldéron--Zygmund theory,  for any $q\in (1,\infty)$ we have
\begin{equation}\label{eq:properties_CZ}
    \| \mathscr{R}_{CZ} f\|_{L^q_x}\lesssim_q \| f\|_{L^q_x},\quad \forall f\in L^q_x.
\end{equation}
This applies for instance to $\mathscr{R}=\nabla^\perp (-\Delta)^{-1} \partial_1$ (for IPM) and $\mathscr{R}=\nabla^\perp (-\Delta)^{-1/2}$ (for SQG).

The Biot--Savart operator $\mathscr{R}_{BS}$ is instead $(-1)$-homogeneous.
In this case, by standard properties of Riesz transforms and Riesz potentials of order $1$ (see \cite[Chapters 3-4]{stein1970singular}), for every choice of parameters ${q_1}\in (1,2)$ and $q_2\in (2,\infty)$ satisfying $\frac{1}{q_2}=\frac{1}{q_1}-\frac{1}{2}$. one has
\begin{equation}\label{eq:properties_BS_integrability}
    \| \mathscr{R}_{BS} f\|_{L^{q_2}_x}\lesssim_{q_1} \| f\|_{L^{q_1}_x},\quad \forall f\in L^{q_1}_x.
\end{equation}
Moreover, $\nabla\mathscr{R}_{BS}=-\nabla \nabla^{\perp}(-\Delta)^{-1}$ is a Caldéron--Zygmund operator, therefore
\begin{equation}\label{eq:properties_BS}
    \|\mathscr{R}_{BS} f\|_{\dot W^{1,q}_x}=\| \nabla \mathscr{R}_{BS}  f\|_{L^q_x}\lesssim_q \| f\|_{L^q_x},\quad \forall f\in L^1_x\cap L^\infty, \quad q\in (1,\infty).
\end{equation}
Estimate \eqref{eq:properties_BS} then extends to any $f$ such that $\mathscr{R}_{BS} f$ is well-defined; as a consequence, for any couple $(q_1,q_2)$ as in \eqref{eq:properties_BS_integrability}, $\mathscr{R}_{BS}$ is a bounded operator from $L^{q_1}_x\cap L^{q_2}_x$ to $W^{1,q_2}_x$.

Note in particular that $\mathscr{R}_{BS}$ is not a well-defined operator on $L^2_x$, which is why when dealing with 2D Euler we always consider functions in $L^r_x\cap L^p_x$ with $r\in (1,2)$.
However, $\mathscr{R}_{BS}$ does enjoy nice mapping properties between homogeneous Sobolev spaces:
\begin{equation}\label{eq:BS_Sobolev_spaces}
    \| \mathscr{R}_{BS} f\|_{\dot H^{s+1}_x} \lesssim \| f\|_{\dot H^s_x},\quad \forall s\in\R.
\end{equation}
As a consequence, for any $s\geq 1$, $\mathscr{R}_{BS}f\in H^s_x$ whenever $f\in \dot H^{-1}_x\cap H^{s-1}_x$.

\subsection{Solution concepts and vanishing viscosity approximations}\label{subsec:solutions}

As mentioned in the introduction, several results we obtain are valid even without any constitutive relation between the advected scalar $\theta$ and the drift $u$, as long as they satisfy suitable regularity and/or integrability assumptions.
To make this clear, let us consider more generally on $\R^d$ the stochastic transport equation
\begin{equation}\label{eq:abstract_transport_SPDE}
    \mathd \theta + b \cdot \nabla \theta \mathd t + \circ\, \mathrm{d} W \cdot \nabla \theta = f \mathd t, \quad  \theta |_{t = 0} = \theta_0.  
\end{equation}
For simplicity we always assume the forcing $f$ to be deterministic, but we allow $b$ in \eqref{eq:abstract_transport_SPDE} to be random.
As standard, in order to study solutions to \eqref{eq:abstract_transport_SPDE}, we pass to the (formally) equivalent It\^o form, which reads
\begin{equation*}
    \mathd \theta + b \cdot \nabla \theta \mathd t + \mathrm{d} W \cdot \nabla \theta = f \mathd t+\frac{1}{2} C(0):D^2 \theta \,\dd t, \quad  \theta |_{t = 0} = \theta_0,
\end{equation*}
where $C$ is the covariance of $W$.

We employ the following relatively standard notion of weak solution to \eqref{eq:abstract_transport_SPDE}, which is weak both in the analytical and in the probabilistic sense.

\begin{definition}\label{martingale_sol_linear}
    Let $C$ be a covariance satisfying \eqref{eq:properties.C}.
    We say that a tuple $(\Omega,\mathcal{F},(\mathcal{F}_t)_{t\geq 0},\mathbb{P}, \theta,b,W)$ is a solution of \eqref{eq:abstract_transport_SPDE} with deterministic initial condition $\theta_0\in L^2_x$ and deterministic forcing $f\in L^1_TL^2_x$ if:
    \begin{enumerate}
        \item $(\Omega,\mathcal{F},(\mathcal{F}_t)_{t\geq 0},\mathbb{P})$ is a filtered probability space satisfying the standard assumptions and $W$ is an $\mathcal{F}_t$-Wiener process with covariance $C$;
        \item $\theta:\Omega\times [0,T]\rightarrow L^2_x$  is a $\mathcal{F}_t$-progressively measurable process whose trajectories are $\mathbb{P}$-a.s. weakly continuous in the sense of distributions and such that $\theta\in L^2_T L^2_x$ $\mathbb{P}$-a.s.;
        \item $b:\Omega\times [0,T]\times \R^d\to \R^d$ is $\mathcal{F}_t$-progressively measurable, $\PP$-a.s. $\nabla\cdot b=0$ in the sense of distributions, and $\mathbb{P}$-a.s. $ b, b \theta\in L^1_{loc}([0,T]\times \R^d)$;
        \item for every $\phi\in C^{\infty}_c(\R^d)$, the following holds $\mathbb{P}$-a.s. for every $t\in [0,T]$
        \begin{align}\label{weak_formulation}
            \langle \theta_t,\phi\rangle-\langle \theta_0,\phi\rangle&=\int_0^t \left( \langle b_s\theta_s,\nabla\phi\rangle+\frac{1}{2}\langle \theta_s,C(0):D^2\phi\rangle +\langle f_s,\phi\rangle\right) \mathd{ s}+\int_0^t \langle \theta_s \nabla \phi
            ,  \mathd{ W}_s\rangle.
        \end{align}
    \end{enumerate}
\end{definition}

Under the assumptions in \autoref{martingale_sol_linear}, the integrals in \eqref{weak_formulation} are well defined, continuous semimartingales, see e.g. \cite[Lemma 2.2]{DrGaPa25}.
The nonlinear SPDE \eqref{eq:nonlinear_transport} can be regarded as a special case of \eqref{eq:abstract_transport_SPDE} with an additional constraint:

\begin{definition}\label{martingale_sol}
    Let $C$, $\theta_0$ and $f$ be given as in \autoref{martingale_sol_linear} and $d=2$. We say that a tuple $(\Omega,\mathcal{F},(\mathcal{F}_t)_{t\geq 0},\mathbb{P}, \theta,W)$ is a probabilistically weak solution to \eqref{eq:nonlinear_transport}
    if it satisfies \autoref{martingale_sol_linear} with $b=u=\mathscr{R}\theta$, where the latter is assumed to be $\PP$-a.s. a well-defined element of $L^1_{\rm loc}([0,T]\times\R^d)$.
    We say that $\theta$ is a probabilistically strong solution if the above holds with $(\mathcal{F}_t)_{t\geq 0}$ being the standard augmentation of the filtration generated by $W$.
\end{definition}

As discussed in the introduction, many of the properties of \eqref{eq:abstract_transport_SPDE}-\eqref{eq:nonlinear_transport} will be established through vanishing viscosity approximations.
Rather than studying directly \eqref{eq:abstract_transport_SPDE} directly, let us consider suitable viscous approximations $\{\theta^\nu\}_{\nu \in (0,1)}$ solving
\begin{align}\label{eq:general.viscous.SPDE}
    \mathd \theta^\nu + b^\nu \cdot \nabla \theta^\nu \mathd t + \mathrm{d} W \cdot \nabla \theta^\nu =
  \frac12 C(0) : D^2 \theta^\nu \mathd t + \nu \Delta\theta^\nu \mathd t + f^\nu \mathd t, \quad \theta^\nu\vert_{t=0}=\theta^\nu_0.
\end{align}

In what follows, we will always assume that, for any fixed $\nu \in (0,1)$,
\begin{equation}\label{eq:viscous_id_basic}
    \theta^\nu_0\in L^1_x\cap L^\infty_x, \quad f^\nu\in L^1_T (L^1_x\cap L^\infty_x).
\end{equation}
We will further assume that, once a filtered probability space $(\Omega,\mathcal{F},(\mathcal{F}_t)_{t\geq 0},\mathbb{P})$ is given, supporting an $\mathcal{F}_t$-Brownian motion $W$ with covariance $C$ and a $\mathcal{F}_t$-progressive divergence-free drift $b^\nu$, then existence and pathwise uniqueness of $\mathcal{F}_t$-progressive solutions $\theta^\nu$ to \eqref{eq:general.viscous.SPDE} holds, in the class
\begin{align*}
    \theta^\nu\in L^\infty_{\omega,T} (L^1_x\cap L^\infty_x)\cap L^\infty_\omega L^2_T \dot H^1_x.
\end{align*}
This is true for instance whenever $b^\nu\in L^\infty_{\omega,T} L^p_x$ for some $p\in [2,\infty]$; indeed, existence can be established by weak compactness as in \cite[Section 5]{bagnara2025regularization}, while uniqueness follows from the stochastic Lions--Magenes lemma (\cite[Theorem 2.13]{RozLot2018}) since $b^\nu \theta^\nu\in L^2_{\omega,T,x}$.
Thanks to the same result, $\theta^\nu$ satisfies $\PP$-a.s. the energy balance
\begin{align}\label{viscous_energy_eq}
    \norm{\theta^\nu_t}_{L^2_x}^2+2\nu\int_0^t\norm{\nabla \theta^{\nu}_s}_{L^2_x}^2 \mathd s&=\|\theta_0^{\nu}\|_{L^2_x}^2
    +
    2\int_0^t \langle f^\nu_s, \theta^\nu_s \rangle \mathd s,\quad \forall t\in [0,T],
\end{align}
and moreover $\theta^\nu$ has $\PP$-a.s. paths in $C_T L^2_x$.
Furthermore, mutatis mutandis in \cite[Lemma 4.7 and Proposition 5.3]{bagnara2025regularization}, one has the $\PP$-a.s. pathwise estimates
\begin{align}\label{viscous_bound_1}
   \sup_{t \in [0,T]} \norm{\theta^\nu_t}_{L^q_x}&\leq \norm{\theta^\nu_0}_{L^q_x}+\int_0^T \norm{f^\nu_t}_{L^q_x} \mathd t,
   \quad \forall q\in [1,+\infty],
   \\
    \label{viscous_bound_2}
    \nu\int_0^T \norm{\nabla\theta^\nu_t}_{L^2_x}^2\mathd s& \leq \frac{1}{2}\left(\norm{\theta^\nu_0}_{L^2_x}+\int_0^T \norm{f^\nu_t}_{L^2_x} \mathd t\right)^2;
\end{align}
One can improve \eqref{viscous_bound_1} to the following, sharper one: $\PP$-a.s.
\begin{align}\label{viscous_bound_variant}
   \sup_{t \in [s,T]} \norm{\theta^\nu_t}_{L^q_x}&\leq \norm{\theta^\nu_s}_{L^q_x}+\int_s^T \norm{f^\nu_t}_{L^q_x} \mathd t, \quad \forall s\in[0,T],\quad\forall q\in (1,+\infty).
\end{align}
Indeed, running the same argument on the time interval $[s,T]$, together with uniqueness, yields the $\PP$-a.s. bound \eqref{viscous_bound_1} with $0$ replaced by any fixed $s$; one can then note that, for $q\in (1,\infty)$, $\theta^\nu$ has $\PP$-a.s. paths in $C_T L^q_x$, by interpolating between $C_T L^2_x$ and $L^\infty_T (L^1_x\cap L^\infty_x)$. One can finally choose a countable dense family $\{s_n\}_{n\in \N}\subset [0,T]$ where the bound \eqref{viscous_bound_1} holds $\PP$-a.s. with $0$ replaced by any $s_n$, and extend it to any $s\in [0,T]$ by continuity, obtaining \eqref{viscous_bound_variant}.

All the aforementioned estimates still apply in the case when $d=2$ and $b^\nu=\mathscr{R}\theta^\nu$, namely when $\theta^\nu$ solves \eqref{eq:nonlinear_viscous}.
In this case, mutatis mutandis in \cite[Lemma 4.7]{bagnara2025regularization}, for any $\nu \in (0,1)$ and $(\theta^\nu_0,f^\nu)$ satisfying \eqref{eq:viscous_id_basic}, there exists a probabilistically strong solution $\theta^\nu$ to \eqref{eq:nonlinear_viscous}.
In this case, since the driving noise has independent stationary increments, the unique solution $\theta^\nu$ has the Markov property.
In order to obtain uniform-in-$\nu$ estimates, we will need some control on the quantities appearing on the right-hand sides of \eqref{viscous_bound_1}-\eqref{viscous_bound_2}-\eqref{viscous_bound_variant}.
The viscous approximations we consider are specified in the following definition.
\begin{definition}\label{defn:vanishing_scheme}
    Let $\theta_0$, $f$, $\mathscr R$, $p,r \in [1,\infty)$ be given. We say that a family $\{(\theta^\nu_0,f^\nu,\theta^\nu)\}_{\nu \in (0,1)}$ is a \emph{vanishing viscosity scheme} for the SPDE \eqref{eq:nonlinear_transport} on $[0,T]$ if
    \begin{align*}
        &\{(\theta^\nu_0,f^\nu)\}_{\nu \in (0,1)}\subset L^1_x\cap L^\infty_x \times L^1_T(L^1_x\cap L^\infty_x),\\ 
        & \sup_{\nu \in (0,1)}
        \| \theta^\nu \|_{L^\infty_{\omega,T} \left(L^r_x\cap L^p_x\right)} + \sup_{\nu \in (0,1)} \norm{f^\nu}_{L^1_T\left(L^r_x\cap L^p_x\right)}<+\infty,\\
        &(\theta_0^\nu,f^\nu)\to (\theta_0, f) \text{ in } L^r_x\cap L^p_x \times L^1_T(L^r_x\cap L^p_x) \text{ as }\nu\to 0^+,
    \end{align*}
    and for every $\nu \in (0,1)$ the process $\theta^\nu$ solves \eqref{eq:nonlinear_viscous} in a sense analogous to \autoref{martingale_sol}.
\end{definition}

\begin{rmk} \label{rmk:bound_viscoud_data} 
The class of vanishing viscosity schemes in \autoref{defn:vanishing_scheme} is non-empty, as can be readily checked using cut-offs and convolutions.
Indeed, the approximations can be constructed so that additionally
\begin{equation} \label{bound_viscoud_data}
    \sup_{\nu \in (0,1)} \norm{\theta_0^\nu}_{L^r_x\cap L^p_x} \leq  \norm{\theta_0}_{L^r_x\cap L^p_x}, \quad
    \sup_{\nu \in (0,1)} \norm{f^\nu}_{L^1_T\left(L^r_x\cap L^p_x\right)} \leq  \norm{f}_{L^1_T\left(L^r_x\cap L^p_x\right)}.
\end{equation}
\end{rmk}

In view of \eqref{viscous_bound_1}-\eqref{viscous_bound_2}, noting that under \autoref{ass:exponent.beta} we have $L^r_x\cap L^p_x \subset L^2_x $ by interpolation, any vanishing viscosity scheme in the sense of \autoref{defn:vanishing_scheme} satisfies
\begin{equation}\label{viscous_uniform_bound}\begin{split}
    \sup_{\nu \in (0,1)}
    \left( \| \theta^\nu \|_{L^\infty_{\omega,T} \left(L^r_x\cap L^p_x\right)}^2 + \nu\|\theta^\nu\|^2_{L^\infty_\omega L^2_T \dot{H}^1_x}
    \right) 
    \lesssim 
    \sup_{\nu \in (0,1)} \left(
     \| \theta^\nu \|_{L^\infty_{\omega,T} \left(L^r_x\cap L^p_x\right)}^2 + \norm{f^\nu}_{L^1_T\left(L^r_x\cap L^p_x\right)}^2\right)< \infty.
\end{split}\end{equation}

We finally recall one further property of the above viscous approximations, which will turn out convenient to cover critical regimes.
By \cite[Lemma A.3]{bagnara2025regularization}, for each fixed $\varepsilon>0$ we can decompose $\theta_0^{\nu},f^{\nu}$ as
\begin{align}
\theta_0^{\nu}=:\theta_0^{\nu,1}+\theta_0^{\nu,2},
\quad
f^{\nu}=:f^{\nu,1}+f^{\nu,2},
\end{align}
where $\theta_0^{\nu,1}, \theta_0^{\nu,2}\in L^{1}_x\cap L^{\infty}_x$ and $f^{\nu,1}, f^{\nu,2}\in L^1_T(L^{1}_x\cap L^{\infty}_x)$, with bounds
\begin{align}\label{eq:split_in_cond}
    \sup_{\nu\in (0,1)}\left(\|\theta_0^{\nu,1}\|_{L^r_x\cap L^p_x}+\|f^{\nu,1}\|_{L^1_T (L^r_x\cap L^p_x)}\right)&\leq \varepsilon,
    \\
    \label{eq:split_in_cond.bis}
    \sup_{\nu\in (0,1)}\left(\|\theta_0^{\nu,2}\|_{L^r_x\cap L^{\infty}_x}+\|f^{\nu,2}\|_{L^1_T (L^r_x\cap L^{\infty}_x)}\right)&\leq M_{\eps}<\infty.
\end{align}
In view of this splitting and the pathwise uniqueness for the viscous problem \eqref{eq:general.viscous.SPDE} (with $b^\nu=u^\nu$), the solutions $\theta^{\nu}$ to \eqref{eq:nonlinear_viscous} can be written as $\theta^{\nu}=:\theta^{\nu,1}+\theta^{\nu,2} $ for $\{\theta^{\nu,i}\}_{i=1,2}$ solving
\begin{align*}
\begin{cases} 
\mathd \theta^{\nu,i} + u^\nu \cdot \nabla \theta^{\nu,i} \mathd t + \circ\, \mathrm{d} W \cdot \nabla \theta^{\nu,i} =
  f^{\nu,i} \mathd t + \nu \Delta \theta^{\nu,i} \mathd t, \\
  \theta^{\nu,i} |_{t = 0} = \theta^{\nu,i}_0,
\end{cases}
\end{align*}
on the whole time interval $[0,T]$, so that again by \eqref{viscous_bound_1} we have $\theta^{\nu,i}\in L^{\infty}_{\omega,T}(L^1_x\cap L^{\infty}_x)$ and
\begin{align}\label{eq_viscous_splitting}
        \sup_{\nu\in (0,1)}\|\theta^{\nu,1}\|_{L^{\infty}_{\omega,T}\left(L^r_x\cap L^p_x\right)}\leq \varepsilon,\quad \sup_{\nu\in (0,1)}\|\theta^{\nu,2}\|_{L^{\infty}_{\omega,T}(L^r_x\cap L^{\infty}_x)}\leq M_\eps <\infty.
\end{align}

\subsection{Besov-type spaces}\label{subsec:besov}

Recall the definition of the spaces $\tilde{L}^2_{\omega,T} \tilde B^\beta_{2,\infty}$ as given in \eqref{def:Besov-like.reg}; these are Banach spaces whenever endowed with the norm $\| f\|_{\tilde{L}^2_{\omega,T} \tilde B^\beta_{2,\infty}}:=\| f\|_{L^2_{\omega,T,x}} + \llbracket f\rrbracket_{\tilde{L}^2_{\omega,T} \tilde B^\beta_{2,\infty}}$.

Estimating directly increments $\delta_z f$ can be rather challenging; it turns out convenient to consider spherical averages of increments instead. For any $\varrho>0$, set
\begin{equation}\label{eq:relevant_seminorms}\begin{split}
    g_\beta(\varrho,f):=\mathbb{E} \left[ \int_{[0, T] \times \mathbb{R}^d}
   \int_{\mathbb{S}^{d-1}} \frac{| \delta_{\varrho z} f_s (y) |^2}{\varrho^{2\beta}} \sigma (\dd z) \mathd y \mathd s \right], \quad 
   \seminorm{f}_{\beta,\varrho_0}^2
   := \sup_{\varrho \in (0,\varrho_0)} g_\beta(\varrho,f).
\end{split}\end{equation}
This reduction essentially comes without cost: by \cite[Lemma 2.18]{DrGaPa25}, it holds
\begin{align*}
    \seminorm{f}_{\beta,+\infty}\sim\llbracket f\rrbracket_{\tilde L^2_{\omega,T} \tilde B^\beta_{2,\infty}}.
\end{align*}
Moreover, by triangular inequality one can readily check that
\begin{align*}
    \seminorm{f}_{\beta,+\infty}\leq \seminorm{f}_{\beta,\varrho_0} + \frac{1}{\varrho_0^\beta} \| f\|_{L^2_{\omega,T,x}}.
\end{align*}
As a consequence, in order to obtain regularity estimates, we only need to consider sufficiently small increments, namely study $\seminorm{f}_{\beta,\varrho_0}$ for $\varrho_0$ sufficiently small. This localization turns out very convenient, as $\seminorm{f}_{\beta,\varrho_0}$ provides rather precise control over quantities like $\| \delta_z f\|_{L^2_{\omega,T} H^s_x}$ whenever $|z|\leq \varrho_0$ and $s\in (0,\beta)$, which in turn gives us access to Sobolev embeddings. We refer to \autoref{app:besov} and in particular \autoref{lem:final_besov} for the precise (technical) statements, which will be crucial in \cref{sec:general.criterion,sec:computation}

To prove sharpness of anomalous dissipation in \autoref{thm_anomalous_dissipation}, we need a similar deterministic class of spaces, which will be relevant in \autoref{subsec:no_anomalous}.
For $\gamma\in (0,1)$ and $p\in [1,+\infty]$, we denote by $E^{\gamma,p}_{[s,t]}$ the closure of Schwartz functions in $\tilde L^2_{[s,t]}\tilde B^\gamma_{p,\infty}$, the latter being defined analogously to \cite[Section 2.6.3]{bahouri2011fourier}. In particular, arguing as in \cite[Lemma 2.18]{DrGaPa25}, one has
\begin{align*}
    \seminorm{f}_{\tilde L^2_{[s,t]} \tilde B^\gamma_{p,\infty}}
    \sim 
    \sup_{z\neq 0} \frac{1}{|z|^\gamma}\| \delta_z f\|_{L^2_{[s,t]}L^p_x}, 
    \quad
    \| f\|_{\tilde L^2_T \tilde B^\gamma_{p,\infty}}
    = 
    \| f\|_{L^2_{[s,t]}L^p_x} + \seminorm{f}_{\tilde L^2_{[s,t]} \tilde B^\gamma_{p,\infty}}
\end{align*}
and
\begin{equation}\label{eq:characterization_sharp_endpoint}
    f\in E^{\gamma,p}_{[s,t]}
    \quad\Leftrightarrow \quad 
    f\in \tilde L^2_{[s,t]} \tilde B^\gamma_{p,\infty} 
    \quad \text{and} \quad 
    \lim_{\eps\to 0} \sup_{|z|\leq \eps} \frac{1}{|z|^\gamma}\| \delta_z f\|_{L^2([s,t]; L^p_x)}=0.
\end{equation}
Notice in particular that, for any $q\in [2,\infty)$ we have the embedding
\begin{align*}
    L^2_{[s,t]} B^\gamma_{p,q} \hookrightarrow E^{\gamma,p}_{[s,t]}.
\end{align*}

It is clear from the definitions that $L^2_\omega E^{\gamma,p}_{T} \subset \tilde{L}^2_{\omega,T} \tilde B^\beta_{2,\infty}$. Due to position of quantifiers however, these spaces are often not directly comparable; the use of $E^{\gamma,p}_{T}$ will provide finer pathwise estimates in \autoref{prop:dissipation.implies.irregularity}.

\subsection{Weighted spaces}\label{subsec:weighted_spaces}

Since $\R^2$ is unbounded, we will employ weighted Sobolev spaces in order to perform tightness arguments in the proof of \autoref{thm:regularization}, see \autoref{subsec:weak_existence}.
Let $w(x) := (1 + |x|^2)^{-2}$; for any $s\in \R$, let $\tilde{H}^s_x$ be the closure of compactly supported smooth functions with respect to the weighted norm 
\begin{align*}
    \|\phi\|_{\tilde{H}^s_x}=\|\phi\, w\|_{H^s_x}.
\end{align*}
defined by the weight $w(x) := (1 + |x|^2)^{-2}$.
By \cite[Lemma A.4]{bagnara2025regularization}, for any $s\in\R$ and $\delta>0$ one has
\begin{equation}\label{eq:weighted_spaces_embeddings}
    H^{s+\delta}_x\hookrightarrow \tilde H^s_x\hookrightarrow H^{s-\delta}_{ {\rm loc}}\quad \text{with compact embeddings}.
\end{equation}
Using these facts and classical Aubin-Lions-Simon compactness arguments (cf. \cite{Simon1987}), one can obtain compactness criteria for $\tilde H^{s}_x$-valued paths.
The proof of the next statement is a basic variation on the ones of \cite[Lemma 2.2]{crippa2025zero} and \cite[Lemma 3.5]{bagnara2025no}, and it is therefore omitted.

\begin{lem}\label{lem:compactness_paths_weights}
The following hold:
\begin{itemize}
    \item[i)] For any $s_1,\gamma,\delta_1>0$, we have the compact embedding
    \begin{equation*}
        L^\infty_T L^2_x\cap C^\gamma_T H^{-s_1}_x \hookrightarrow C_T \tilde H^{-\delta_1}_x.
    \end{equation*}
    \item[ii)] If $K$ is compact in $C_T \tilde H^{-\delta_1}_x$, then for any $s_2>0$ and any $R>0$, the set
    \begin{align*}
        \tilde K_R:= K\cap \left\{ f\in L^2_T H^{s_2}_x: \ \| f\|_{L^2_T H^{s_2}_x} \leq R \right\}
    \end{align*}
    is compact in $C_T \tilde H^{-\delta_1}_x\cap L^2_T \tilde H^{s_2-\delta_2}_x$.
\end{itemize}
\end{lem}

\section{A general criterion for anomalous regularization} 
\label{sec:general.criterion}

Consider momentarily a solution $\theta$ to the linear Kraichnan model, namely with $\mathscr{R}\equiv 0$ and $f\equiv 0$ in \eqref{eq:nonlinear_transport}.
In this case, as shown in \cite[Section 5]{DrGaPa25}, due to the specific structure of covariance of $W$, a representation of the dissipation measure $\mathcal{D}[\theta]$ \`{a} la Duchon-Robert produces a coercive term in the energy balance; in turn, this term controls the $\tilde{L}^2_{\omega,T} \tilde{B}^{1-\alpha}_{2,\infty}$ Besov-type regularity of $\theta$. More precisely, \cite[equation (1.11)]{DrGaPa25} yields
\begin{align*}
    \|\theta_0\|_{L^2_x}^2 - \EE[\| \theta_T\|_{L^2_x}^2] \sim \lim_{\eps \to 0} \int_0^T \int_{\mathbb{S}^{d-1}} \frac{\EE[\| \delta_{\varrho z}\theta_t\|_{L^2_x}^2]}{\varrho^{2(1-\alpha)}} \sigma(\dd z) \dd t.
\end{align*}
This suggests the use of Duchon--Robert-type representation formulae as regularization estimates beyond the linear case.
Here we explore this idea more in depth, proving it to be also very effective in more complicated, nonlinear models.

The main result of this section is \autoref{lem:gron_increments}; it can be seen as a general criterion, allowing to obtain anomalous regularization estimates in a perturbative fashion, whenever the commutator $\mathcal{T}_{b,\theta}$ associated to the transport term $b\cdot\nabla \theta$ is of lower order compared to the expected regularization as measured by $\seminorm{\theta}_{\beta,\varrho_0}$, which is localized at sufficiently small scales.

We focus here on general SPDEs of the form \eqref{eq:abstract_transport_SPDE} on $\R^d$ with random drift $b$; in the upcoming \autoref{sec:computation} we will then specialize our results to active scalar SPDEs in $\R^2$ of the form \eqref{eq:nonlinear_transport}.
As before, rather than studying directly \eqref{eq:abstract_transport_SPDE}, let us consider its viscous approximations $\{\theta^\nu\}_{\nu \in (0,1)}$ solving \eqref{eq:general.viscous.SPDE}.
We are interested in uniform-in-$\nu$ bounds on suitable Besov-type norms on $\theta^\nu$. 
We point out that we do not necessarily aim at closing the an estimate best possible regularity exponent $1-\alpha$; indeed, there are scenarios (cf. some of the supercritical cases in \cref{ass:exponent.beta}) in which one can only close the estimate for \emph{some} $\beta<1-\alpha$.

We first need the following deterministic lemma; recall the notation $\delta_z g (x) := g (x + z) - g (x)$.
\begin{lem} \label{lem:trilinear}
Let $\chi:\R^d \to \R$ be a symmetric, compactly supported  probability density with $\nabla \chi \in L^\infty_x$, and denote $\chi^\varepsilon(z) := \varepsilon^{-d} \chi(\varepsilon^{-1} z)$.
  Let $\theta \in L^p_x$ for some $p \in [2,\infty]$ and define $\theta^{\varepsilon} :=  \chi^{\varepsilon} \ast \theta$. 
  Let $b$ be a divergence-free velocity field such that $b \in L^q_x$ with $1/q + 2/p=1$.
  Then for any $\eps>0$ it holds that
  \begin{equation} \label{eq:commutator}
    \langle b \cdot \nabla \theta^{\varepsilon}, \theta \rangle = -
    \frac{1}{4}  \int_{\mathbb{R}^d \times \mathbb{R}^d} \frac{1}{\varepsilon}
    \nabla \chi (z) \cdot \delta_{\varepsilon z} b (x)  |  \delta
_{\varepsilon z} \theta (x) |^2 \mathd x \mathd z. 
  \end{equation}
\end{lem}

\begin{proof}
First, notice that by Young's convolution inequality we have $\nabla \theta^\varepsilon \in L^p_x$ and so the quantity $\langle b \cdot \nabla \theta^{\varepsilon}, \theta \rangle$ is well-defined.
Since $b$ is divergence-free and $\chi$ is symmetric, by integration by parts we have
  \begin{align*}
    2 \langle b \cdot \nabla \theta^{\varepsilon}, \theta \rangle 
    &= 
    \langle b \cdot \nabla \theta^{\varepsilon}, \theta \rangle - \langle (b \cdot \nabla \theta)^{\varepsilon}, \theta \rangle,
  \end{align*}
where $(b \cdot \nabla \theta)^{\varepsilon} := (b \cdot \nabla \theta)\ast \chi^{\varepsilon}$, and the product $\langle (b \cdot \nabla \theta)^{\varepsilon}, \theta \rangle$ is well-defined since $(b \cdot \nabla \theta)^\varepsilon = (\nabla \cdot (b\theta))^\varepsilon = \nabla \cdot (b \theta)^\varepsilon \in L^{p'}_x$, with $p'$ being the H\"older conjugate of $p$. The above is the classical DiPerna--Lions commutator \cite{diperna1989ordinary}:
\begin{align*}
   2 \langle b \cdot \nabla \theta^{\varepsilon}, \theta \rangle
    &= \int_{\mathbb{R}^d \times \mathbb{R}^d} [b (x) - b (y)] \cdot \nabla
    \chi^{\varepsilon} (x - y) \theta (x) \theta (y) \mathd x \mathd y.
\end{align*}
  
  Next, let us manipulate the right-hand side of \eqref{eq:commutator}. 
  By the assumptions on $p$ and $q$, the double integral is well-defined for every $\varepsilon>0$ and we can rigorously expand the square inside.  
  Since $\nabla \chi^{\varepsilon}$ is an odd function, by change of variables we have
  \begin{align*}
   & \int_{\mathbb{R}^d \times \mathbb{R}^d} \nabla \chi^{\varepsilon} (x - y)
    \cdot (b (x) - b (y)) | \theta (y) |^2 \mathd x \mathd y =
    \int_{\mathbb{R}^d \times \mathbb{R}^d} \nabla \chi^{\varepsilon} (x - y)
    \cdot (b (x) - b (y)) | \theta (x) |^2 \mathd x \mathd y.
    \end{align*} 
    Moreover, since $b$ is divergence-free and $\chi$ is compactly supported:
    \begin{align*}
    &\int_{\mathbb{R}^d \times \mathbb{R}^d} \nabla \chi^{\varepsilon} (x - y)
    \cdot (b (x) - b (y)) | \theta (x) |^2 \mathd x \mathd y
    \\& = \int_{\mathbb{R}^d} b (x) | \theta (x) |^2 \cdot \left(
    \int_{\mathbb{R}^d} \nabla \chi^{\varepsilon} (x - y) \mathd y \right)
    \mathd x
     - \int_{\mathbb{R}^d} | \theta (x) |^2 \mathd x \left(
    \int_{\mathbb{R}^d} \nabla \chi^{\varepsilon} (x - y) \cdot b (y) \mathd y
    \right) \mathd x = 0.
  \end{align*}
  Therefore, we conclude that 
\begin{align*}
    \langle b \cdot \nabla \theta^{\varepsilon}, \theta \rangle
    &=
    \frac12 \int_{\mathbb{R}^d \times \mathbb{R}^d} [b (x) - b (y)] \cdot
    \nabla \chi^{\varepsilon} (x - y) \theta (x) \theta (y) \mathd x \mathd
    y
    \\
    &=-\frac14 \int_{\mathbb{R}^d \times \mathbb{R}^d} \frac{1}{\varepsilon}
    \nabla \chi (z) \cdot \delta_{\varepsilon z} b (x)  | \delta
_{\varepsilon z} \theta (x) |^2 \mathd x \mathd z
. \qedhere
\end{align*}
\end{proof}

We can now move to the the main result of this section.
For $\nu \in (0,1)$, suppose we are given a family of solutions $\{\theta^\nu\}_{\nu \in (0,1)}$ to the SPDEs \eqref{eq:general.viscous.SPDE}; here $b^\nu$ and $\theta^\nu$ can be both random and we do not impose any specific relation between them.

We are interested in obtaining uniform-in-$\nu$ bounds for the expected averaged increments of the solutions:
\begin{align} \label{eq:average.increments}
   \sup_{\nu \in (0,1)} \sup_{\varrho \in (0,1)} \mathbb{E} \left[ \int_{[0, T] \times \mathbb{R}^d}
   \int_{\mathbb{S}^{d-1}} \frac{| \delta_{\varrho z} \theta^{\nu}_t (y) |^2}{\varrho^{2\beta}} \sigma (\dd z) \mathd y \mathd t \right]
   < \infty,
\end{align}
where $\beta \in (0,1-\alpha]$ is a parameter determining the Besov-type space regularity of the scalar fields $\{\theta^\nu\}_{\nu \in (0,1)}$, as explained around \eqref{def:Besov-like.reg} and in \autoref{subsec:besov}.
In order to study \eqref{eq:average.increments}, similarly to the quantities defined in \eqref{eq:relevant_seminorms}, let us set
\begin{align}
g^\nu_\beta (\varrho) &:= g_\beta(\varrho,\theta^\nu):=\mathbb{E} \left[ \int_{[0, T] \times \mathbb{R}^d}
   \int_{\mathbb{S}^{d-1}} \frac{| \delta_{\varrho z} \theta^{\nu}_s (y) |^2}{\varrho^{2\beta}} \sigma (\dd z) \mathd y \mathd s \right], \nonumber
   \\
   \seminorm{\theta^\nu}_{\beta,\varrho_0}^2 &:= \sup_{\varrho \in (0,\varrho_0)} g^\nu_\beta (\varrho)
   := \sup_{\varrho \in (0,\varrho_0)} g_\beta(\varrho,\theta^\nu),  \nonumber
   \\
 \mathcal{T}_{b^\nu,\theta^\nu}(t,\varrho) &:= \int_{\mathbb{R}^d} \int_{\{| z |  \leqslant 1\}} z \cdot \delta_{\varrho z} b^\nu_t (y)  | \delta_{\varrho z} \theta^\nu_t (y)
  |^2 \mathd y \mathd z. \label{eq:prototype_longitudinal_function}
\end{align}

\begin{lem}\label{lem:gron_increments}
Let $W$ satisfy \autoref{ass:noise} with constants $c_1,c_2,r_0>0$.
Fix $T>0$ and consider a family of progressive divergence-free velocity fields $\{b^\nu\}_{\nu \in (0,1)} \subset L^\infty_{\omega,T}L^s_x$ for some $s \in [2,\infty]$.
Assume that the families $\{\theta^\nu_0\}_{\nu \in (0,1)} \subset L^1_x \cap L^\infty_x$ and $\{f^\nu\}_{\nu \in (0,1)} \subset L^1_T(L^1_x \cap L^\infty_x)$ are such that
\begin{align} \label{eq:assumptions.theta_0.f}
    \sup_{\nu \in (0,1)}
     \|\theta_0^\nu \|_{L^2_x}^2 =: M_1 < \infty,
     \quad
     \sup_{\nu \in (0,1)}  \|f^\nu\|^2_{L^1_TL^2_x}
     =: M_{2} < \infty;
\end{align} 
let $\{\theta^\nu\}_{\nu\in (0,1)}$ solve \eqref{eq:general.viscous.SPDE} with initial condition $\theta_0^\nu$ and satisfy \eqref{viscous_energy_eq}-\eqref{viscous_bound_1}-\eqref{viscous_bound_2}.\\
Fix $\beta \in (0,1-\alpha]$ and further suppose that there exist constants
\begin{align*}
    \eta \in \left[ 0, \frac{2\beta}{d+2\alpha+2\beta}\right), \quad M_3\in (0,+\infty), \quad \varrho_0 \in (0,r_0\wedge 1)
\end{align*}
such that  
\begin{align} \label{eq:key_bound_regularization}
\mathbb{E} \left[ \int_0^T \mathcal{T}_{b^\nu,\theta^\nu}(s,\varrho) \mathd s \right]
  \leq
  c_1 \varrho^{-1+2\alpha+2\beta} \left(
  M_3  + \eta \seminorm{\theta^\nu}_{\beta,\varrho_0}^2 \right),\quad \forall\nu \in (0,1), \, \forall \varrho \in (0,\varrho_0).
\end{align}
Then there exists a finite constant $K=K(c_1,c_2,d)$ such that, for $M_0 := K(1+T)M_1+KM_2+M_3$, it holds:
\begin{equation}\label{eq:gron_main_estimate}
 \sup_{\nu \in (0,1)} \seminorm{\theta^\nu}_{\beta,\varrho_0}^2  \leq M_0\left( \frac{2\beta}{d+2\alpha+2\beta} - \eta \right)^{-1}  < \infty.
\end{equation}
Moreover, the solutions $\{\theta^\nu\}_{\nu \in (0,1)}$ are uniformly bounded in $\tilde{L}^2_{\omega,T} \tilde{B}^{\beta}_{2,\infty}$ with estimate
\begin{equation}\label{eq:gron_main_estimate2}
    \sup_{\nu \in (0,1)} \seminorm{\theta^\nu}_{\tilde{L}^2_{\omega,T} \tilde{B}^{\beta}_{2,\infty}}^2  
 \lesssim M_0\left[ \left( \frac{2\beta}{d+2\alpha+2\beta} - \eta \right)^{-1} + \varrho_0^{-2\beta}\right].
\end{equation}
\end{lem}

\begin{proof}
The first part of the proof follows some of the computations of \cite[Section 5]{DrGaPa25}. Let $\chi$ be as in \autoref{lem:trilinear} and mollify $\theta^\nu, f^\nu$ by defining
\begin{align*}
    \theta^{\nu,\varepsilon} := \theta^\nu \ast \chi^\varepsilon,
    \quad
    f^{\nu,\varepsilon} := f^\nu \ast \chi^\varepsilon,
\end{align*}
which solve the SPDE
\begin{align*} 
 \mathd \theta^{\nu, \varepsilon} + (b^{\nu} \cdot \nabla
   \theta^{\nu})^\varepsilon \mathd t + \sum_{k \in \N} (\sigma_k \cdot \nabla
   \theta^\nu)^\varepsilon \mathd W^k  =  \frac{1}{2} C (0) : D^2 \theta^{\nu,
   \varepsilon} \mathd t + \nu \Delta \theta^{\nu, \varepsilon}  \mathd t
   +
f^{\nu,\varepsilon} \mathd t.   
\end{align*}

Testing the equation above against $\theta^{\nu}$ and symmetrizing we find
\begin{align*} 
\mathd \left(\theta^{\nu} \theta^{\nu, \varepsilon}\right) 
     &= 
     \mathd M^{\nu,\varepsilon}_t 
     +
     \frac{1}{2} C (0) : D^2 (\theta^{\nu} \theta^{\nu, \varepsilon}) \mathd t 
     - 
     C (0) \nabla \theta^{\nu} \cdot \nabla \theta^{\nu, \varepsilon} \mathd t
     + 
     \sum_{k \in \N} (\sigma_k \cdot \nabla \theta^{\nu}) (\sigma_k \cdot \nabla
     \theta^{\nu})^{\varepsilon}\mathd t 
    \notag \\
     & \quad+ 
     \nu \Delta (\theta^{\nu} \theta^{\nu, \varepsilon})\mathd t 
     - 
     2 \nu  \nabla \theta^{\nu} \cdot \nabla \theta^{\nu, \varepsilon} \mathd t
     - 
     (b^{\nu} \cdot \nabla \theta^{\nu})^{\varepsilon} \theta^{\nu} \mathd t 
     -
     ( b^{\nu} \cdot \nabla \theta^{\nu}) \theta^{\nu, \varepsilon}
     \mathd t
     \notag \\
     & \quad+ 
     f^{\nu,\varepsilon} \theta^\nu \mathd t + f^\nu \theta^{\nu,\varepsilon} \mathd t, 
\end{align*}
where $M^{\nu,\varepsilon}$ is a suitable martingale term. Next, we manipulate the right-hand side of the previous equation in the following way: First, we can rewrite
\begin{align*}
C (0)  \nabla \theta^{\nu} \cdot \nabla \theta^{\nu, \varepsilon} \mathd t
&=
\nabla  \cdot \left( \theta^\nu C (0) \nabla \theta^{\nu, \varepsilon} \right) \mathd t
-
\theta^\nu C (0) :D^2 \theta^{\nu, \varepsilon} \mathd t 
\\
&=
\nabla  \cdot \left( \theta^\nu C (0) \nabla \theta^{\nu, \varepsilon} \right) \mathd t
-
\theta^\nu [(C (0) :D^2 \chi^\varepsilon) \ast\theta^{\nu}] \mathd t;
\end{align*}
Similarly, using the fact that the coefficients $\sigma_k$ are divergence-free:
\begin{align*}
\sum_{k \in \N} (\sigma_k \cdot \nabla \theta^{\nu}) (\sigma_k \cdot \nabla
\theta^{\nu})^{\varepsilon}\mathd t 
&=
\sum_{k \in \N} \nabla \cdot (\sigma_k  \theta^{\nu}) \nabla \cdot(\sigma_k      \theta^{\nu})^{\varepsilon}\mathd t 
\\
&=
\nabla \cdot \left( \sum_{k \in \N}  \sigma_k  \theta^{\nu} \nabla \cdot(\sigma_k      \theta^{\nu})^{\varepsilon}\mathd t  \right)
-
\sum_{k \in \N}  \sigma_k  \theta^{\nu} \cdot \nabla(\nabla \cdot(\sigma_k      \theta^{\nu})^{\varepsilon})\mathd t
\\
&=
\nabla \cdot \left( \sum_{k \in \N}  \sigma_k  \theta^{\nu} \nabla \cdot(\sigma_k      \theta^{\nu})^{\varepsilon}\mathd t  \right)
-
\theta^\nu [(C:D^2\chi^\varepsilon) \ast \theta^\nu] \mathd t;
\end{align*}
Finally, by \autoref{lem:trilinear}, for almost every time $t \in [0,T]$ we have $\mathbb{P}$-almost surely
\begin{align*}
- 
&\int_{\mathbb{R}^d}(b^{\nu} \cdot \nabla \theta^{\nu})^{\varepsilon} \theta^{\nu} \mathd x 
-
\int_{\mathbb{R}^d}( b^{\nu} \cdot \nabla \theta^{\nu}) \theta^{\nu, \varepsilon} \mathd x 
=
2 \langle b^{\nu} \cdot \nabla  \theta^{\nu, \varepsilon}, \theta^{\nu} \rangle 
\\
&=
-\frac{1}{2} \int_{\mathbb{R}^d \times \mathbb{R}^d} \frac{1}{\varepsilon} \nabla \chi(z) \cdot \delta_{\varepsilon z} b^\nu(x) | \delta_{\varepsilon z} \theta^\nu(x)|^2 \mathd x \mathd z.
\end{align*}

Integrating $\mathd \left(\theta^{\nu} \theta^{\nu, \varepsilon}\right)$ over $[0,T] \times \mathbb{R}^d$, taking expectations, and recalling that $Q(z) = C(0)-C(z)$, we get
\begin{equation} \label{eq:energy_balance}\begin{split}
\mathbb{E} [ \langle \theta^{\nu}_T , \theta^{\nu, \varepsilon}_T \rangle]
&+
  2 \nu \mathbb{E} \left[ \int_0^T \langle \nabla \theta^{\nu}_s, \nabla
  \theta^{\nu, \varepsilon}_s \rangle \mathd s\right] 
  -
\mathbb{E} [ \langle \theta^{\nu}_0 , \theta^{\nu, \varepsilon}_0 \rangle]
\\
&=
\mathbb{E} \left[ \int_0^T
\int_{\mathbb{R}^d \times \mathbb{R}^d}
Q (x - y) : D^2\chi^{\varepsilon} (x - y)  
  \theta^{\nu}_s (x) \theta^{\nu}_s (y) \mathd y \mathd x \mathd s \right]
  \\
  & \quad
  -\frac{1}{2}\mathbb{E} \left[
  \int_0^T  \int_{\mathbb{R}^d \times \mathbb{R}^d} \frac{1}{\varepsilon} \nabla \chi(z) \cdot \delta_{\varepsilon z} b^\nu_s(x) | \delta_{\varepsilon z} \theta^\nu_s(x)|^2 \mathd x \mathd z \mathd s \right] 
  \\
  &\quad+ 
  \mathbb{E} \left[ \int_0^T (\langle f^{\nu,\varepsilon}_s ,\theta^\nu_s \rangle + \langle f^\nu_s , \theta^{\nu,\varepsilon}_s \rangle) \mathd s \right]
  \\ 
  & =: I^{\nu, \varepsilon}_1 + I^{\nu, \varepsilon}_2 + I^{\nu, \varepsilon}_3.
\end{split}\end{equation}

As in \cite{DrGaPa25}, the term $I^{\nu,\varepsilon}_1$ can be further manipulated arguing by symmetries, integration by parts, using the fact that $Q$ is divergence-free, and rescaling variables. We get
\begin{align*}
  I^{\nu,\varepsilon}_1 
  &= - \frac{1}{2} \mathbb{E} \left[ \int_0^T \int_{\mathbb{R}^d \times
  \mathbb{R}^d} Q (x - y) : D^2\chi^{\varepsilon} (x - y) | \theta^{\nu}_s (x)
  - \theta^{\nu}_s (y) |^2 \mathd y \mathd x \mathd s \right] 
  \\
  &= 
- \frac{1}{2} \mathbb{E} \left[ \int_0^T \int_{\mathbb{R}^d \times \mathbb{R}^d} 
   \frac{Q (\varepsilon z)}{\varepsilon^2} : D^2 \chi ( z) \, | \delta_{\varepsilon z} \theta^{\nu}_s (x) |^2 \mathd z \mathd x \mathd s \right].
\end{align*}

Following \cite[Corollary 5.5]{DrGaPa25}, we pick the special choice  
\begin{align*} 
\chi (z) &:=   k_d (1 - | z |^2) \mathbf{1}_{\{| z | \leqslant 1\}}, \quad k_d:=\frac{d(d+2)}{2\omega_{d-1}},
\end{align*}
where $\omega_{d-1}$ denotes the surface area of $\mathbb{S}^{d-1}$; it fulfills the assumptions of \autoref{lem:trilinear} and satisfies $D^2\chi \in \mathrm{BV}$.
Using the expansion \eqref{eq:local.expansion.noise} of $Q$ around the origin, for any $\eps\leq r_0$ one finds
\begin{align*}
I^{\nu, \varepsilon}_1 = 
&- 
c_1 k_d \mathbb{E}
   \left[  \int_0^T \int_{\mathbb{R}^d} \int_{\mathbb{S}^{d-1}} 
   \frac{| \delta_{\varepsilon z} \theta^{\nu}_s (y) |^2}{\varepsilon^{2-2 \alpha }} \sigma (\dd z) \mathd y \mathd s \right]
   \\
&+
c_1 k_d(d+2\alpha) \mathbb{E}
   \left[ 
   \int_0^T \int_{\mathbb{R}^d} \int_{\{| z | \leqslant 1\}} | z |^{2 \alpha}  \frac{ | \delta_{\varepsilon z} \theta^{\nu}_s (y) |^2 }{\varepsilon^{2-2 \alpha}} \mathd z \mathd y \mathd s
   \right]
   + Rem_T^{\nu,\varepsilon},    
\end{align*}
where the remainder satisfies $|Rem_T^{\nu,\varepsilon}| \lesssim c_2T\|\theta^\nu\|_{L^\infty_tL^2_x}^2$ uniformly in $T>0$, $\nu \in (0,1)$, and $\varepsilon\in (0,r_0]$.  
Furthermore, our choice of $\chi$ gives
\[ I^{\nu, \varepsilon}_2 =  k_d \mathbb{E} \left[ \int_0^T \int_{\mathbb{R}^d}
   \int_{\{| z | \leq 1\}} \frac{z \cdot \delta_{\varepsilon z} b^{\nu}_s
   (y)}{\varepsilon}  | \delta_{\varepsilon z} \theta^{\nu}_s (y) |^2 \mathd y
   \mathd z \mathd s \right] 
   =
   \frac{k_d}{\varepsilon} \mathbb{E} \left[ \int_0^T \mathcal{T}_{b^\nu,\theta^\nu}(s,\varepsilon) \mathd s \right]. \]

Plugging these expressions for $I^{\nu, \varepsilon}_1$ and $I^{\nu, \varepsilon}_2$ in \eqref{eq:energy_balance} and rearranging, we obtain:
\begin{align*} 
& \mathbb{E}
   \left[  \int_0^T \int_{\mathbb{R}^d} \int_{\mathbb{S}^{d-1}} 
   \frac{| \delta_{\varepsilon z} \theta^{\nu}_s (y) |^2}{\varepsilon^{2-2 \alpha }} \sigma (\dd z) \mathd y \mathd s \right]
   \\
&\quad-
(d+2\alpha) \mathbb{E}
   \left[ 
   \int_0^T \int_{\mathbb{R}^d} \int_{\{| z | \leqslant 1\}} | z |^{2 \alpha}  \frac{ | \delta_{\varepsilon z} \theta^{\nu}_s (y) |^2 }{\varepsilon^{2-2 \alpha}} \mathd z \mathd y \mathd s
   \right]
\\   
&=\frac{1}{c_1 \varepsilon} \mathbb{E} \left[ \int_0^T \mathcal{T}_{b^\nu,\theta^\nu}(s,\varepsilon) \mathd s \right]
  \\
  &\quad
 + \frac{1}{c_1 k_d}\left( -\mathbb{E} [ \langle \theta^{\nu}_T , \theta^{\nu, \varepsilon}_T \rangle]
-
  2 \nu \mathbb{E} \left[ \int_0^T \langle \nabla \theta^{\nu}_s, \nabla
  \theta^{\nu, \varepsilon}_s \rangle \mathd s\right] 
  +
\mathbb{E} [ \langle \theta^{\nu}_0 , \theta^{\nu, \varepsilon}_0 \rangle]
  + 
  Rem^{\nu,\varepsilon}_T + I^{\nu,\varepsilon}_3\right).
\end{align*}
We can estimate the terms in the last line above using \eqref{viscous_energy_eq}-\eqref{viscous_bound_1}-\eqref{viscous_bound_2}, $|Rem_T^{\nu,\varepsilon}| \lesssim c_2 T\|\theta^\nu\|_{L^\infty_T L^2_x}^2$, and the basic bounds
\begin{align} \label{eq:inequality.I3}
    |\langle \theta^{\nu}_T , \theta^{\nu, \varepsilon}_T\rangle|  \lesssim \| \theta^\nu_T\|_{L^2_x}^2,
    \quad
   |\langle \nabla \theta^{\nu}_s, \nabla
  \theta^{\nu, \varepsilon}_s \rangle|
  \lesssim
  \| \theta^\nu_s \|^2_{\dot{H}^1_x},
  \quad
  |I^{\nu,\varepsilon}_3| \lesssim \| \theta_0^\nu\|_{L^2_x}^2 + \|f^\nu\|^2_{L^1_{T}L^2_x};
  \end{align}
taking $\eps=\varrho <r_0$ and rewriting the integral over $|z|\leq 1$ over spherical means, overall one obtains the following: There exist a constant $K=K(d,c_1,c_2)$ such that uniformly in $T>0$, $\nu \in (0,1)$, and $\varrho\in (0, \varrho_0)$ it holds:
\begin{equation}\label{eq:inequality_g}\begin{split}
    g^\nu_\beta (\varrho)& \leq (d + 2\alpha) \varrho^{- d-2\alpha-2\beta} \int_0^\varrho g^\nu_\beta (r) r^{d-1+2\alpha+2\beta} \mathd r
  + \frac{\varrho^{1-2\alpha-2\beta}}{c_1} \mathbb{E} \left[ \int_0^T \mathcal{T}_{b^{\nu},\theta^{\nu}}(s,\varrho)\mathd s \right] 
  \\
  &\quad+ \varrho^{2-2\alpha-2\beta} K \left[ (1+T) \|\theta_0^\nu \|_{L^2_x}^2 +  \|f^\nu\|^2_{L^1_{T}L^2_x} \right].
\end{split}\end{equation}

By definition of $\seminorm{\theta^\nu}_{\beta,\varrho_0}$, for $\varrho < \varrho_0$ we have
    \begin{equation}\label{eq:inequality_g2}
       (d + 2\alpha) \varrho^{- d-2\alpha-2\beta} \int_0^\varrho g^\nu_\beta (r) r^{d-1+2\alpha+2\beta} \mathd r \leq
        \frac{d+2\alpha}{d+2\alpha+2\beta} \seminorm{\theta^\nu}_{\beta,\varrho_0}^2;
    \end{equation}
invoking assumptions \eqref{eq:assumptions.theta_0.f} and \eqref{eq:key_bound_regularization}, from \eqref{eq:inequality_g} we obtain that
\begin{align*}
    g^\nu_\beta(\varrho) \leq \left( \frac{d+2\alpha}{d+2\alpha+2\beta} + \eta \right) \seminorm{\theta^\nu}_{\beta,\varrho_0}^2 + M_0 \quad \forall \varrho<\varrho_0 < 1.
\end{align*}
The bound \eqref{eq:gron_main_estimate} on $\seminorm{\theta^\nu}_{\beta,\varrho_0}^2$ then follows from the assumption $\eta < \frac{2\beta}{d+2\alpha+2\beta}$, after taking the supremum over $\varrho \in (0,\varrho_0)$, $\nu \in (0,1)$ and rearranging terms in \eqref{eq:inequality_g2}.
Finally, by \cite[Lemma 2.18]{DrGaPa25}, estimate \eqref{eq:gron_main_estimate} and the definition of $\seminorm{\cdot}_{\beta,\varrho_0}$ we have
\begin{align*}
 \sup_{\nu \in (0,1)} \seminorm{\theta^\nu}_{\tilde{L}^2_{\omega,T} \tilde{B}^{\beta}_{2,\infty}}^2  
 \lesssim
 \sup_{\nu \in (0,1)} \seminorm{\theta^\nu}_{\beta,+\infty}^2
 \lesssim
 \sup_{\nu \in (0,1)} \seminorm{\theta^\nu}_{\beta,\varrho_0}^2 + \varrho_0^{-2\beta}\sup_{\nu \in (0,1)} \|\theta^\nu\|_{L^2_{\omega,T,x}}^2
\end{align*}
from which \eqref{eq:gron_main_estimate2} follows.
\end{proof}

We conclude this section with a nontrivial application of the lemma above.

\begin{lem}
\label{lem:data-splitting}
Let $d\geq 2$, $p\in (2,3)$, $-1\leq \mathfrak{h}\leq 0$, $0<\alpha<\frac{1-\mathfrak{h}}{2}$ and set
\begin{equation*}
 \gamma:=\frac{1-\mathfrak{h}}{2}-\alpha,\qquad
 \beta:=\gamma(p-2).
\end{equation*}
Let $W$ satisfy \autoref{ass:noise} with constants $c_1,c_2,r_0>0$.
Fix $T>0$ and consider a family of progressive, divergence-free drifts $\{b^\nu\}_{\nu\in(0,1)}\subset L^\infty_{\omega,T}L^s_x$ for some $s \in [2,\infty]$,  satisfying
\begin{equation}
\sup_{\nu\in(0,1)}
\|b^\nu\|_{L^\infty_{\omega,T}\dot W^{-\mathfrak{h},p}_x}<\infty.
\label{eq:drift-sobolev}
\end{equation}
Assume that the families $\{\theta^\nu_0\}_{\nu \in (0,1)} \subset L^1_x \cap L^\infty_x$ and $\{f^\nu\}_{\nu \in (0,1)} \subset L^1_T(L^1_x \cap L^\infty_x)$ are such that \begin{equation}
 M_p:=\sup_{\nu\in(0,1)}
 \left(\|\theta_0^\nu\|_{L^2_x\cap L^p_x}
       +\|f^\nu\|_{L^1_T(L^2_x\cap L^p_x)}\right)<\infty,
 \label{eq:data-bounds}
\end{equation}
and let $\{\theta^\nu\}_{\nu\in (0,1)}$ solve \eqref{eq:general.viscous.SPDE} with initial condition $\theta_0^\nu$ and satisfy \eqref{viscous_energy_eq}-\eqref{viscous_bound_1}-\eqref{viscous_bound_2}. 
Then
\begin{equation}
 \sup_{\nu \in (0,1)} \seminorm{\theta^\nu}_{\tilde{L}^2_{\omega,T} \tilde{B}^{\beta}_{2,\infty}}^2  
 \lesssim
  M_p^2.
 \label{eq:conclusion}
\end{equation}

\end{lem}

\begin{proof}
For every fixed $\nu \in (0,1)$, take $b^\nu$ as in the statement and consider an auxiliary scalar $\Theta^\nu$ solving \eqref{eq:general.viscous.SPDE} with initial condition $\Theta_0^\nu\in L^1_x \cap L^\infty_x$ and forcing $\Phi^\nu\in L^1_T(L^1_x \cap L^\infty_x)$. Denote
\begin{align*}
     \widehat{M}_l^\nu:=\|\Theta_0^\nu\|_{L^l_x}+\|\Phi^\nu\|_{L^1_TL^l_x},
     \qquad 
     \forall l \in [1,+\infty].
\end{align*}
We first obtain a regularization estimate for $\Theta^\nu$ in terms of $\widehat{M}_2^\nu$ and $\widehat{M}_q^\nu$, where $q:=\frac{2p}{p-1}>p$.

Relation \eqref{viscous_bound_1} implies
$\|\Theta^\nu\|_{L^\infty_{\omega,T}L^l_x}\leq \widehat{M}_l^\nu$.
For $-1\leq \mathfrak{h}\leq 0$, \eqref{eq:basic_increments} and \eqref{eq:drift-sobolev} imply
\begin{equation}
 \sup_{\nu\in(0,1)}
 \|\delta_h b^\nu\|_{L^\infty_{\omega,T}L^p_x}
 \lesssim |h|^{-\mathfrak{h}},
 \qquad 
 \forall h\ne0.
 \label{eq:drift-increment}
\end{equation}
Hence, by H\"older's inequality and \eqref{eq:drift-increment}, we have
\begin{align*}
 \mathbb{E}\left[\int_0^T|\mathcal T_{b^\nu,\Theta^\nu}(t,\varrho)|\,\dd t\right]
 &\leq \mathbb{E}\left[\int_0^T\int_{|z|\leq1}|z|
       \|\delta_{\varrho z}b_t^\nu\|_{L^p_x}
       \|\delta_{\varrho z}\Theta_t^\nu\|_{L^q_x}^2\,\dd z\,\dd t\right]
       \notag\\
 &\lesssim T\varrho^{-\mathfrak{h}}(\widehat{M}_q^\nu)^2.
\end{align*}
The choice of $\gamma$ gives
\begin{align*}
    -1+2\alpha+2\gamma=-\mathfrak{h},
 \qquad 0<\gamma\leq1-\alpha.
\end{align*}
Thus, \autoref{lem:gron_increments} applies to the auxiliary scalar with regularity
$\gamma$ and $\eta=0$, leading to 
\begin{equation}
 \seminorm{\Theta^\nu}_{\tilde{L}^2_{\omega,T} \tilde{B}^{\gamma}_{2,\infty}}^2 
 \lesssim (\widehat{M}_2^\nu)^2+(\widehat{M}_q^\nu)^2.
 \label{eq:auxiliary-regularity}
\end{equation}

We now split $\theta_0^\nu$ and $f^{\nu}$. For $K\geq1$ and $v\in L^2_x\cap L^p_x$,
define
\begin{equation}
 v_K^{\leq}
 :=v\one_{\{|v|\leq K\|v\|_{L^p_x}\}},
 \qquad
 v_K^{>}:=v-v_K^{\leq},
 \label{eq:normalized-truncation}
\end{equation}
with both parts zero when $v=0$. Direct integration gives
\begin{align}
 \|v_K^{>}\|_{L^2_x}
 &\leq K^{-(p-2)/2}\|v\|_{L^p_x},
       \label{eq:tail-bound}\\
 \|v_K^{\leq}\|_{L^q_x}
 &\leq K^{1-p/q}\|v\|_{L^p_x},
 \qquad
 \|v_K^{\leq}\|_{L^2_x}\leq\|v\|_{L^2_x}.
       \label{eq:bounded-part-bound}
\end{align}
Apply this splitting to $\theta_0^\nu$ and separately to $f_t^\nu$ for
almost every $t$. 

Let $\theta_K^{\nu,\leq}$ and $\theta_K^{\nu,>}$ solve
the corresponding linear equations, both with drift $b^\nu$ and noise
$W$. Their data retain the approximating $L^1_x\cap L^\infty_x$
regularity and we are in the setting of previous computations. Also, by linearity and uniqueness of solutions, it holds
\begin{equation}
 \theta^\nu=\theta_K^{\nu,\leq}
             +\theta_K^{\nu,>}.
 \label{eq:solution-splitting}
\end{equation}
The bounds \eqref{viscous_bound_1} and \eqref{eq:tail-bound} imply
\begin{equation}
 \sup_{\nu\in (0,1)}\|\theta_K^{\nu,>}\|_{L^\infty_{\omega,T}L^2_x}
 \leq K^{-(p-2)/2} M_p.
 \label{eq:transported-tail}
\end{equation}
On the other hand, applying \eqref{eq:auxiliary-regularity} and \eqref{eq:bounded-part-bound} gives
\begin{equation}
\sup_{\nu\in (0,1)}\seminorm{\theta_K^{\nu,\leq}}_{\tilde{L}^2_{\omega,T} \tilde{B}^{\gamma}_{2,\infty}}
 \lesssim  M_p^2+K^{2(1-p/q)} M_p^2.
 \label{eq:transported-low}
\end{equation}

For $0<\varrho<\varrho_0$, equations \eqref{eq:solution-splitting},\eqref{eq:transported-tail},\eqref{eq:transported-low} now yield
\begin{align*}
\int_0^T \int_{\mathbb{S}^{d-1}} \EE[\| \delta_{\varrho z}\theta_t\|_{L^2_x}^2] \sigma(\dd z) \dd t
 &\lesssim 
   \left(\varrho^{2\gamma} 
   +\varrho^{2\gamma}K^{2(1-p/q)} 
   +K^{2-p}\right) M_p^2.
\end{align*}
Choosing
\begin{equation*}
 K:=\varrho^{-2\gamma/[p(1-2/q)]}
   =\varrho^{-2\gamma }\geq1,
\end{equation*}
then both terms involving $K$ have exactly the same power:
\begin{align*}
    \varrho^{2\gamma}K^{2(1-p/q)}=K^{2-p}
 =\varrho^{2\gamma (p-2)}=\varrho^{2\beta}.
\end{align*}
Also $0<\beta<\gamma$, so $\varrho^{2\gamma}\leq\varrho^{2\beta}$ for
$\varrho<1$. Hence
\begin{align*}
    \int_0^T \int_{\mathbb{S}^{d-1}} \EE[\| \delta_{\varrho z}\theta_t\|_{L^2_x}^2] \sigma(\dd z) \dd t
 &\lesssim  M_p^2\varrho^{2\beta}.
\end{align*}
Taking the supremum over $\varrho$ and $\nu$ proves
\eqref{eq:conclusion}. 
\end{proof}

\section{Anomalous regularization for 2D active scalar equations}
\label{sec:computation}

We prove here the first part of the statement of \autoref{thm:regularization}, concerning uniform estimates for $\{\theta^\nu\}_{\nu \in (0,1)}$ in $\tilde L^2_{\omega,T} \tilde B^\beta_{2,\infty}$.
To do so, the abstract criterion from \autoref{sec:general.criterion} and its consequences, cf. \autoref{lem:data-splitting}, with the choice $b^\nu=\mathscr{R} \theta^\nu$ play a major role.
\begin{prop}\label{prop_cases}
    Let $\alpha,\ p,\ r,\  \mathscr{R}$ satisfy the assumption of \autoref{thm:regularization}, consider a
vanishing viscosity scheme in the sense of \autoref{defn:vanishing_scheme} and set
$b^\nu=u^\nu=\mathscr{R}\theta^\nu$. \begin{enumerate}
\item
If $p\geq3$ or \autoref{ass:strong_regularization} holds, then the assumptions of \autoref{lem:gron_increments} are satisfied
with the exponent $\beta$ from \autoref{ass:exponent.beta}.
\item
If $2<p<3$ and \autoref{ass:exponent.beta} holds,
the hypotheses of \autoref{lem:data-splitting} are satisfied with
\begin{equation}\label{eq:formula_beta_intermediate}\begin{split}
    \mathfrak{h}=\begin{cases}
 -1,&\mathscr{R}=\mathscr{R}_{BS},\\
 0,&\mathscr{R}=\mathscr{R}_{CZ},
 \end{cases}
 \qquad
 \gamma=\frac{1-\mathfrak{h}}{2}-\alpha,
 \qquad \beta=(p-2)\gamma.
\end{split}\end{equation}
\end{enumerate}
\end{prop}

We split verification in the different cases introduced in \autoref{ass:exponent.beta}, see the subsections below.
Before starting, let us recall that whenever needed we can invoke the decomposition $\theta^\nu=\theta^{\nu,1}+\theta^{\nu,2}$ satisfying the estimates \eqref{eq:split_in_cond}-\eqref{eq:split_in_cond.bis}. This allows to treat ``critical'' exponents $p$ by making the contribution of $\theta^{\nu,1}$ small, at the price of allowing a possibly very large, subcritical component $\theta^{\nu,2}$.

\subsection{Case 1: $\mathscr{R}=\mathscr{R}_{BS},\ p \in [2,3),\ \alpha \leq 1-1/p$}\label{sss:regularity_case1}
Let $\eps>0$ to be fixed later. In light of the decomposition $\theta^\nu=\theta^{\nu,1}+\theta^{\nu,2}$, one similarly has $u=u^{\nu,1}+u^{\nu,2}$ with $u^{\nu,i}=\mathscr{R}_{BS} \theta^{\nu,i}$; by estimates \eqref{eq:split_in_cond}-\eqref{eq:split_in_cond.bis} and \eqref{eq:properties_BS}, we deduce the uniform-in-$\nu$ bounds 
\begin{align}\label{estimate_splitting_Euler}
        \sup_{\nu\in (0,1)}\|u^{\nu,1}\|_{L^{\infty}_{\omega,T} (\dot{W}^{1,r}_x \cap \dot{W}^{1,p}_x)}\lesssim \varepsilon,\quad \sup_{\nu\in (0,1)}\|u^{\nu,2}\|_{L^{\infty}_{\omega,T} \left(\dot{W}^{1,r}_x\cap \dot{W}^{1,2p}_x\right)}\lesssim M_\eps, 
\end{align}

Let us fix
\begin{equation*}
    s:=\frac53-\frac2p \in \left[\frac23,1\right), \quad
    \delta := \frac{1}{3(1-\alpha)} \in \left(\frac13,1\right);
\end{equation*}
then by Sobolev embeddings
\begin{align*}
\norm{u^{\nu,1}}_{L^\infty_{\omega,T}\dot{W}^{s, 3}_x} &\lesssim
\norm{u^{\nu,1}}_{L^\infty_{\omega,T}\dot{W}^{1,p}_x}\lesssim \varepsilon, 
\qquad
\norm{u^{\nu,2}}_{L^\infty_{\omega,T}\dot{W}^{1, 3}_x}\lesssim 
 M_\eps,
\\    
\norm{\delta_{r z}\theta^\nu}_{L^2_{\omega,T}L^3_x} 
&\lesssim
\norm{\delta_{r z} \theta^\nu}_{L^2_{\omega,T}H^{1/3}_x}
=
\norm{\delta_{r z}\theta^\nu}_{L^2_{\omega,T}H^{\delta(1-\alpha)}_x}.
\end{align*}

By \autoref{lem:final_besov} from \autoref{app:besov} (with $T_0=0$), for any $\rho\in (0,1)$ it holds that
\begin{align*}
    \mathbb{E} \left[ \int_0^T \int_{\{| z |\leqslant 1\}}  
   \norm{\delta_{\varrho z} \theta^\nu_t}_{H^{\delta(1-\alpha)}_x}^2  \mathd z \mathd t \right]
   \lesssim \rho^{2(1-\delta)(1-\alpha)} \seminorm{\theta^\nu}_{1-\alpha,\varrho_0}^2
   = \rho^{4/3-2\alpha} \seminorm{\theta^\nu}_{1-\alpha,\varrho_0}^2.
\end{align*}
Combining this estimate with H\"{o}lder's inequality and \eqref{eq:basic_increments}, we find
\begin{align*}
    \mathbb{E} \left[ \int_0^T |\mathcal{T}_{u^\nu,\theta^\nu}(t,\varrho)| \mathd t \right] 
    & \lesssim \mathbb{E} \left[ \int_0^T \int_{\{| z |\leqslant 1\}}  \|  \delta_{\rho z} (u^{\nu,1}_t +u^{\nu,2}_t) \|_{L^3_x} \| \delta_{\varrho z} \theta^{\nu}_t \|_{L^3_x}^2 \mathd z \mathd t \right]\\
    &\lesssim
   \varrho^s \mathbb{E} \left[ \int_0^T \int_{\{| z |\leqslant 1\}}  \|  u^{\nu,1}_t \|_{\dot{W}^{s, 3}_x} \| \delta_{\varrho z} \theta^{\nu}_t \|_{L^3_x}^2 \mathd z \mathd t \right]
   \\ 
   &\quad+ \varrho \mathbb{E} \left[ \int_0^T \int_{\{| z | \leqslant 1\}}  \|  u^{\nu,2}_t \|_{\dot{W}^{1, 3}_x} \| \delta_{\varrho z} \theta^{\nu}_t \|_{L^3_x}^2 \mathd z \mathd t \right]
   \\ 
   & \lesssim
   (\varepsilon \varrho^s +M_\eps \varrho) \mathbb{E} \left[ \int_0^T \int_{\{| z |\leqslant 1\}}  
   \norm{\delta_{\varrho z} \theta^\nu_t}_{H^{\delta(1-\alpha)}_x}^2  \mathd z \mathd t \right]
   \\ 
   &\lesssim 
   (\varepsilon \varrho^s+M_\eps \varrho) \varrho^{4/3-2\alpha} \seminorm{\theta^\nu}_{1-\alpha,\varrho_0}^2\\ 
   & \lesssim 
   (\varepsilon+M_\eps \varrho_0^{4/3-2\alpha})\varrho\seminorm{\theta^\nu}_{1-\alpha,\varrho_0}^2.
\end{align*}
In the last passage we used that $4/3-2\alpha+s\geq 1$ by construction.
The last inequality implies the validity of \eqref{eq:key_bound_regularization} for $\beta=1-\alpha$ and $M_3=0$, up to choosing first $\eps>0$ and consequently $\rho_0$ small enough so that
\begin{align*}
    \eps + M_\eps \rho_0^{4/3-2\alpha} \leq c_1 \eta < c_1\, \frac{1-\alpha}{2},
\end{align*}
which is allowed since $4/3-2\alpha>0$ by assumption.
\subsection{Case 2: $\mathscr{R}=\mathscr{R}_{BS},\ p\in [3,\infty],\ \alpha\in (0,1)$}\label{sss:regularity_case3} 
Since $L^r_x \cap L^p_x\subset L^3_x$, by estimates \eqref{viscous_uniform_bound} and \eqref{eq:properties_BS} we have
\begin{align*}
    \sup_{\nu\in (0,1)}\norm{u^\nu}_{L^\infty_{\omega,T}\dot{W}^{1, 3}_x}\lesssim \sup_{\nu\in (0,1)}\left(\norm{\theta_0^{\nu}}_{L^r_x\cap L^p_x}+\| f^{\nu} \|_{L^1_T \left(L^r_x\cap L^p_x\right)}\right) < \infty.
\end{align*}
As before, by H\"{o}lder's inequality and \eqref{eq:basic_increments}, we have 
\begin{align*}
    \mathbb{E} \left[ \int_0^T |\mathcal{T}_{u^\nu,\theta^\nu}(t,\varrho)| \mathd t \right] 
    & \lesssim
   \varrho\, \mathbb{E} \left[ \int_0^T \| u^{\nu}_t \|_{\dot{W}^{1, 3}_x} \| \theta^{\nu}_t \|_{L^3_x}^2
   \mathd t \right] \lesssim \varrho,
\end{align*}
which verifies \eqref{eq:key_bound_regularization} for $\beta=1-\alpha$ and $\eta=0$ .
\subsection{Case 3: $\mathscr{R}=\mathscr{R}_{BS},\ p\in (2,3)$ and $1-1/p<\alpha<1$}\label{sss:regularity_case2}
By equations \eqref{eq:basic_increments},\eqref{eq:properties_BS},\eqref{viscous_bound_1} it holds
\begin{align*}
    \|\delta_h u^\nu\|_{L^\infty_{\omega,T}L^p_x}
 \le |h|\|\nabla u^\nu\|_{L^\infty_{\omega,T}L^p_x}
 \lesssim |h| M_p.
\end{align*}
The $L^r_x\cap L^p_x$ bounds also give a uniform $L^\infty_x$ bound
for $u^\nu$, and imply \eqref{eq:data-bounds} by interpolation.
Therefore the assumptions of \autoref{lem:data-splitting} are satisfied with $\mathfrak{h}=-1$ and
$\gamma=1-\alpha$, leading to the claim.
\subsection{Case 4: $\mathscr{R}=\mathscr{R}_{CZ}$, $p\in (2,6],\ \alpha\leq \frac{1}{2}-\frac{1}{p}$ }\label{sss:regularity_case4}

Let $\eps>0$ to be chosen later. Similarly to Case 1, by the decomposition $\theta^\nu=\theta^{\nu,1}+\theta^{\nu,2}$, estimates \eqref{eq:split_in_cond}-\eqref{eq:split_in_cond.bis} and the properties of Caldéron-Zygmund kernels (cf. \eqref{eq:properties_CZ}),  \eqref{eq:properties_BS}, we have $u^\nu=u^{\nu,1}+u^{\nu,2}$ with  
\begin{align}\label{estimate_splitting_CZ}
        \sup_{\nu\in (0,1)}\|u^{\nu,1}\|_{L^{\infty}_{\omega,T}  L^p_x}\lesssim \varepsilon,\quad \sup_{\nu\in (0,1)}\|u^{\nu,2}\|_{L^{\infty}_{\omega,T}  L^{2p}_x}\lesssim M_\eps
\end{align}
where we additionally used the inclusion $L^{2p}_x\subset L^r_x\cap L^\infty_x$.
Let
\begin{align*}
    s_1:=\frac{1}{p}\in \left( 0, \frac{1}{2}-\alpha\right],\quad s_2:=\frac{1}{2p}\in \left(0,\frac12-\alpha\right).
\end{align*}
By H\"older's inequality, Sobolev embeddings, estimate \eqref{estimate_splitting_CZ} and \autoref{lem:final_besov} we have 
 \begin{align*}
    \mathbb{E} \left[ \int_0^T |\mathcal{T}_{u^\nu,\theta^\nu}(t,\varrho)| \mathd t\right] 
    & \lesssim
 \mathbb{E} \left[ \int_0^T \int_{\{| z | \leqslant 1\}} \| \delta_{\varrho z} u^{\nu,1}_t \|_{L^{p}_x}\| \delta_{\varrho z} \theta^{\nu}_t \|_{H^{s_1}_x}^{2}
   \mathd z \mathd t \right]\\ & \quad + \mathbb{E} \left[ \int_0^T \int_{\{| z | \leqslant 1\}} \| \delta_{\varrho z} u^{\nu,2}_t \|_{L^{2p}_x}\| \delta_{\varrho z} \theta^{\nu}_t \|_{H^{s_2}_x}^{2}
   \mathd z \mathd t \right]
   \\ 
& \lesssim  \mathbb{E}\left[\int_0^T \int_{\{| z |\leqslant 1\}} 
\left(\varepsilon\|\delta_{\varrho z} \theta^{\nu}_t \|_{H^{s_1}_x}^{2}+M_\eps \| \delta_{\varrho z} \theta^{\nu}_t \|_{H^{s_2}_x}^{2}\right)
   \mathd z \mathd t\right]
   \\
   & \lesssim \left(\varepsilon \varrho^{2\left(1-\alpha-\frac{1}{p}\right)}+ M_\eps \varrho^{2\left(1-\alpha-\frac{1}{2p}\right)}\right) \seminorm{\theta^\nu}_{1-\alpha,\varrho_0}^2 
   \\
   &\lesssim \left(\varepsilon+M_\eps \varrho_0^{\frac{1}{p}}\right)\varrho \seminorm{\theta^\nu}_{1-\alpha,\varrho_0}^2.
\end{align*}   
The last inequality implies \eqref{eq:key_bound_regularization} for $\beta=1-\alpha$ and $M_3=0$, up to choosing (similarly to Case 1) $\varepsilon>0$ and $\varrho_0>0$ sufficiently small so that the condition on $\eta$ is satisfied.
\subsection{Case 5: $\mathscr{R}=\mathscr{R}_{CZ}$, $p\in (2,3),\ \frac{1}{2}-\frac{1}{p}<\alpha<\frac{1}{2}$} By the estimate \eqref{eq:properties_CZ}, it holds
\begin{align*}
    \|\delta_h u^\nu\|_{L^\infty_{\omega,T}L^p_x}
 \leq 2\|u^\nu\|_{L^\infty_{\omega,T}L^p_x}
 \lesssim M_p.
\end{align*}
Therefore the assumptions of \autoref{lem:data-splitting} are satisfied with $\mathfrak{h}=0$ and
$\gamma=1/2-\alpha$ leading to the claim.
\subsection{Case 6: $\mathscr{R}=\mathscr{R}_{CZ}$, $p\in [3,6],\ \frac{1}{2}-\frac{1}{p}<\alpha<\frac{1}{2}$ or $p>6,\ \alpha \in (0,1/2)$ }\label{sss:regularity_case5}
Let us fix $\varepsilon>0$ and decompose $\theta^\nu=\theta^{\nu,1}+\theta^{\nu,2}$ so that \eqref{estimate_splitting_CZ} applies. Let $s_1,s_2\in [0,1)$ be defined by
\begin{align*}
    \frac{p-1}{2p}=\frac{s_1}{2}+\frac{1-s_1}{p},\quad \frac{2p-1}{4p}=\frac{s_2}{2}+\frac{1-s_2}{p},
\end{align*}
and set $\beta:=(1/2-\alpha)(p-2)\wedge (1-\alpha)$; note that by construction
\begin{equation}\label{eq:case5_parameters}
        s_1 =\frac{p-3}{p-2}<s_2=\frac{2p-5}{2p-4}, \quad 2(s_2-s_1)=\frac{1}{p-2}, \quad -1+2\alpha+2\beta \leq2\beta s_1.
\end{equation}
By H\"older's inequality, interpolation, estimates \eqref{viscous_uniform_bound} and \eqref{estimate_splitting_CZ}, and \autoref{lem:final_besov}, we have
 \begin{align*}
    \mathbb{E} \left[ \int_0^T |\mathcal{T}_{u^\nu,\theta^\nu}(t,\varrho)| \mathd t \right] 
    & \lesssim
 \mathbb{E} \left[ \int_0^T \int_{\{| z | \leqslant 1\}} \| \delta_{\varrho z} u^{\nu,1}_t \|_{L^{p}_x}\| \delta_{\varrho z} \theta^{\nu}_t \|_{L^{p}_x}^{2(1-s_1)}\| \delta_{\varrho z} \theta^{\nu}_t \|_{L^{2}_x}^{2s_1}
   \mathd z \mathd t \right]
   \\ 
   & \quad + \mathbb{E} \left[ \int_0^T \int_{\{| z |\leqslant 1\}} \| \delta_{\varrho z} u^{\nu,2}_t \|_{L^{2p}_x}\| \delta_{\varrho z} \theta^{\nu}_t \|_{L^{p}_x}^{2(1-s_2)}\| \delta_{\varrho z} \theta^{\nu}_t \|_{L^{2}_x}^{2s_2}
   \mathd z \mathd t \right]
   \\ 
   & \lesssim  \mathbb{E}\left[\int_0^T \int_{\{| z |\leqslant 1\}} \left(\varepsilon \| \delta_{\varrho z} \theta^{\nu}_t \|_{L^{2}_x}^{2s_1}+M_\eps\| \delta_{\varrho z} \theta^{\nu}_t \|_{L^{2}_x}^{2s_2}\right)
   \mathd z \mathd t\right]\\
   & \lesssim \left(\varepsilon \varrho^{2s_1 \beta}+M_\eps \varrho^{2 s_2 \beta}\right)\left(\seminorm{\theta^\nu}_{\beta,\varrho_0}^{2s_1} + \seminorm{\theta^\nu}_{\beta,\varrho_0}^{2 s_2}\right)\\
   & \lesssim \left(\varepsilon \varrho^{2s_1 \beta}+M_\eps \varrho_0^{2 (s_2 -s_1)\beta}\right) \rho^{2\beta s_1} \left(1+\seminorm{\theta^\nu}_{\beta,\varrho_0}^2\right)\\
   &\lesssim \left(\varepsilon+M_\eps \varrho_0^{\frac{\beta}{p-2}}\right)\varrho^{-1+2\alpha+2\beta}\left(1+\seminorm{\theta^\nu}_{\beta,\varrho_0}^2\right)
\end{align*}  
where in the last passage we used \eqref{eq:case5_parameters} and $\varrho<1$.
The above estimate shows \eqref{eq:key_bound_regularization}, for suitable $M_3$ and $\eta$, by choosing $\varepsilon>0$ and $\varrho_0$ sufficiently small.
By its definition, one readily checks that $\beta=1-\alpha$ if $p> 6$ and  $\alpha\leq \frac{1}{2}\left(1-\frac{1}{p-3}\right)$, which matches the definition \eqref{eq:definition.regularity_CZ}.

\section{Anomalous integrability} \label{sec:integrability}
This section is devoted to investigating anomalous integrability of solutions to \eqref{eq:nonlinear_transport}, meant as a space integrability gain of the solution $\theta$, as opposed to a simple propagation of space integrability from the initial condition $\theta_0$, which is formally expected in incompressible transport type equations like \eqref{eq:nonlinear_transport} and \eqref{eq:abstract_transport_SPDE}.
\cref{lem:regularization-conditional-bound} constitutes the elementary step described in \autoref{subsec:intro_integrability} for the space integrability gain, which is then completed by the iteration argument in \autoref{prop:abstract_anomalous_integrability}.
Similarly to what was done in \cref{sec:general.criterion,sec:computation}, we first provide in \cref{subsec:integrability_strategy} abstract results that apply in any dimension $d\geq 2$ and for viscous approximations of general SPDEs of the form \eqref{eq:abstract_transport_SPDE}, with no constitutive relation between $b^\nu$ and $\theta^\nu$; then in \cref{subsec:integrability_active_scalars} we specialize to the case of 2D active scalars \eqref{eq:nonlinear_transport} with drift $u^\nu=\mathscr{R} \theta^\nu$.

\subsection{General strategy}\label{subsec:integrability_strategy}

We start with the following observation, which will simplify the arguments by allowing us not to track the forcing term $f$.
Suppose $\theta_0 \in L^r_x \cap L^p_x$, $f \in L^1_T(L^r_x \cap L^{p_\ast}_x)$ with $p_\ast \geq p$ and set up a vanishing viscosity scheme for \eqref{eq:abstract_transport_SPDE} that additionally satisfies the 
uniform bound \eqref{eq:bound.forcing.anomalousintegrability}.
Then we can decompose the solutions $\{\theta^\nu\}_{\nu \in (0,1)}$ as $\theta^\nu = \theta^{\nu,1}+\theta^{\nu,2}$, where $\theta^{\nu,1}$ solves \eqref{eq:general.viscous.SPDE} with initial condition $\theta^{\nu,1}_0=\theta^\nu_0$ and forcing $f^{\nu,1}=0$, while $\theta^{\nu,2}$ solves the same SPDE with initial condition $\theta^{\nu,2}_0=0$ and forcing $f^{\nu,2}=f^\nu$. For the latter, by \eqref{viscous_bound_variant} we immediately have the $\PP$-a.s. bound
\begin{align*}
    \| \theta^{\nu,2} \|_{L^\infty_T L^{p_\ast}_x} \lesssim \| f^\nu \|_{L^1_T L^{p_\ast}_x} \leq M_f < \infty;
\end{align*}
therefore without loss of generality we only need to consider anomalous integrability in the case of zero external forcings $f^\nu \equiv 0$, namely $\theta^\nu$ solving
\begin{align*}
    \mathd \theta^\nu + b^\nu \cdot \nabla \theta^\nu \mathd t + \mathrm{d} W \cdot \nabla \theta^\nu =
  \frac12 C(0) : D^2 \theta^\nu \mathd t + \nu \Delta\theta^\nu \mathd t.
\end{align*}
Moreover, by propagation of $L^p_x$ moments (cf. again \eqref{viscous_bound_variant}) and interpolation, for SPDEs on $\R^d$ it is sufficient to prove that
\begin{equation}\label{eq:anomalous_integrability_general}
    \sup_{\nu \in (0,1)}   \mathbb{E} \left[\sup_{s \in [ t ,T]}  \| {\theta}^{\nu}_s \|_{L^\infty_x}^p  \right]
   \lesssim 
   t^{- \frac{d}{2(1-\alpha)}} M_\theta^p\quad\forall t\in (0,1 \wedge T].
\end{equation}
This is consistent with the estimate discussed in \autoref{rmk:integrability_intro} for $d=2$; \autoref{prop:abstract_anomalous_integrability} extends this a bit further, allowing more general parameters $\beta\in (0,1-\alpha]$.

We are now ready to present our main strategy of proof.
We already know that, if the results from \autoref{sec:general.criterion} apply, then solutions gain some Sobolev regularity $H^s_x$ at Lebesgue a.e. positive time $T_0>0$, which by embeddings imply integrability $L^q_x$ with $1/q=1/2-s/d$. In turn, this integrability is then preserved for $t\geq T_0$ by the incompressibility of the drift $b^\nu$ and the noise $W$. We want to understand whether, restarting the SPDE at time $T_0$ and using Markovianity of the system, we can iterate this argument to gain even further integrability.

In order to do so, we fix an auxiliary parameter $q \in (2,\infty)$ and investigate whether we can apply the anomalous regularization result of \autoref{sec:general.criterion} to the stochastic process
\begin{align*}
    \Theta^\nu :=|\theta^\nu|^{q/2}.
\end{align*}
At a technical level, to rigorously apply It\^o Formula, we regularize the function $z \mapsto |z|^{q/2}$ by taking a sequence $\{F_m\}_{m \in \N} \subset C^2(\R)$ of nonnegative, convex functions such that $F_m(z) \uparrow |z|^{q/2}$ monotonically as $m \to \infty$; 
for technical reasons that will be clear in a moment, we also assume $F''_m$ bounded for every fixed $m$.\footnote{
One can construct $F_m$ in the following way. Let $\psi_m \in C^\infty(\R)$ be a smooth cutoff function taking values in $[0,1]$ and such that $\psi_m(z) = 0$ for $|z|<1/m$ or $|z|>2m$ and $\psi_m(z) = 1$ for $2/m < |z| < m$. Define $F_m$ by imposing $F''_m(z) := \psi_m(z) F''(z)$ and $F_m(0)=F'_m(0)=0$.
}

Let $T_0 \in [0,T)$ denote the (deterministic) time that will serve the purpose of base point from which to restart the system, and so the iteration.
Following the same computation as in \cite[Proposition 2.20]{DrGaPa25}, one can verify that $\Theta^{\nu,m} := F_m(\theta^\nu)$ solves the SPDE
\begin{align} \label{eq:Theta_nu,m}
   \mathd \Theta^{\nu,m} 
    &+ 
    b^{\nu} \cdot \nabla \Theta^{\nu,m} \mathd t 
    +
 \mathd W \cdot \nabla
   \Theta^{\nu,m} 
   =
   \frac{1}{2} C (0) : D^2 \Theta^{\nu,m}\mathd t + \nu \Delta \Theta^{\nu,m} \mathd t + \Phi^{\nu,m} \mathd t,  
\end{align}
on the time interval $[0,T]$, where we defined
\begin{align*}
    \Phi^{\nu,m} := - 
    \nu F''_m( \theta^\nu) |\nabla \theta^\nu|^2.  
\end{align*}

We regard it as an SPDE on $[T_0,T]$ with initial condition $\Theta^{\nu,m}_{T_0}$; notice that \eqref{eq:Theta_nu,m} has the same structure of the general viscous SPDE \eqref{eq:general.viscous.SPDE} of \autoref{sec:general.criterion},
with the only differences that the initial condition $\Theta^{\nu,m}_{T_0}$ and the external forcing $\Phi^{\nu,m}$ are random.
Nonetheless, as shown in \autoref{cor:Theta.gronwall} below, a conditional analogue of \autoref{lem:gron_increments} holds true, providing anomalous regularization for $\Theta^{\nu,m}$.
Hereafter we adopt the short-hand notation $\mathbb{E}_{T_0}:= \mathbb{E}[\,\cdot \mid \mathcal{F}_{T_0}]$ and define
\begin{align*}
    G^{\nu,m}_{\beta,T_0} (\varrho) &:= G_{\beta,T_0}(\varrho,\Theta^{\nu,m}) := \mathbb{E}_{T_0} \left[ \int_{[T_0, T] \times \mathbb{R}^d}
   \int_{\mathbb{S}^{d-1}} \frac{| \delta_{\varrho z} \Theta^{\nu,m}_s (y) |^2}{\varrho^{2\beta}} \sigma (\dd z) \mathd y \mathd s \right],
   \\
   \seminorm{\Theta^{\nu,m}}_{\beta,\varrho_0 \mid \mathcal{F}_{T_0}}^2 &:=
   \sup_{\varrho \in (0,\varrho_0)}G^{\nu,m}_{\beta,T_0}(\varrho):=\sup_{\varrho \in (0,\varrho_0)} G_{\beta,T_0}(\varrho,\Theta^{\nu,m}).
\end{align*}

We refer to \autoref{app:besov} for more details on such conditional seminorms.

\begin{lem} \label{cor:Theta.gronwall}
Let $W$ satisfy \autoref{ass:noise} with constants $c_1,c_2,r_0>0$.
Fix $T>0$, $T_0 \in [0,T)$, and a family of progressive divergence-free drifts $\{b^\nu\}_{\nu \in (0,1)} \subset L^\infty_{\omega,[T_0,T]}L^s_x$ for some $s\in [2,\infty]$.
Assume that $\{\theta^\nu_{T_0}\}_{\nu \in (0,1)} \subset L^1_x \cap L^\infty_x$ $\PP$-a.s. and let $\Theta^{\nu,m}$ solve \eqref{eq:Theta_nu,m} on the time interval $[T_0,T]$.

Fix $\beta \in (0,1-\alpha]$ and suppose that there exist parameters
\begin{align*}
    \eta < \frac{2\beta}{d+2\alpha+2\beta},\quad \varrho_0 \in (0,r_0\wedge 1)
\end{align*}
such that 
\begin{align*} \label{eq:key_bound_regularization.Theta}
\mathbb{E}_{T_0} \left[ \int_{T_0}^T \mathcal{T}_{b^\nu,\Theta^{\nu,m}}(s,\varrho) \mathd s  \right]
  \leq
  c_1 \eta \varrho^{-1+2\alpha+2\beta} \seminorm{\Theta^{\nu,m} }_{\beta,\varrho_0 \mid \mathcal{F}_{T_0}}^2,
  \quad \forall\nu \in (0,1),\ \varrho \in (0,\varrho_0), \ m \in \N.
\end{align*}
Then $\PP$-a.s.
\begin{align*}
 \seminorm{\Theta^{\nu,m}}_{\beta,\varrho_0 \mid \mathcal{F}_{T_0}}^2  \leq K(1+T) \left( \frac{2\beta}{d+2\alpha+2\beta} - \eta \right)^{-1} \| \theta^\nu_{T_0}\|_{L^q_x}^q,
\end{align*}
where $K$ is the same constant as in \autoref{lem:gron_increments}.
\end{lem}

\begin{proof}
The proof follows by performing the same computations as in the proof of \autoref{lem:gron_increments} leading to \eqref{eq:inequality_g}, with $(\Theta^{\nu,m},\Phi^{\nu,m})$ in place of $(\theta^{\nu},f^{\nu})$, using the conditional expectation $\mathbb{E}_{T_0}=\mathbb{E}[\,\cdot \mid \mathcal{F}_{T_0}]$ instead of the expectation $\mathbb{E}$.
Here we only sketch the main differences. 

First of all, the regularity assumption \eqref{eq:assumptions.theta_0.f} on the forcing $f^\nu$ is not necessarily satisfied by $\Phi^{\nu,m}$ in the present setting, and was used to obtain the bound \eqref{eq:inequality.I3} on the term $I^{\nu,\varepsilon}_3$. 
To circumvent this, we notice that since $F''_m$ is bounded,  by \eqref{viscous_bound_2} it holds $\Phi^{\nu,m} \in L^\infty_\omega L^1_T L^1_x$ for every $\nu \in (0,1)$ and $m \in \N$.
Moreover, since $F_m$ is convex, $\Phi^{\nu,m} \leq 0$ for almost every $(\omega,t,x)$.
Therefore, $\PP$-a.s. we have 
\begin{align}
    I^{\nu,\varepsilon}_3
    &=
    \mathbb{E}_{T_0} \left[ \int_{T_0}^T (\langle \Phi^{\nu,m,\varepsilon}_s ,\Theta^{\nu,m}_s \rangle + \langle \Phi^{\nu,m}_s , \Theta^{\nu,m,\varepsilon}_s \rangle) \mathd s  \right]
    \leq
    0,
\end{align}
where we used the facts that $\Theta^{\nu,m} \geq 0$ by construction and that convolution with $\chi^\varepsilon$ preserves positivity.
Similarly, we have the $\PP$-a.s. inequality $\int_{T_0}^T \langle\Theta^{\nu,m}_s, \Phi^{\nu,m}_s \rangle \mathd s \leq 0$, from which we deduce the following energy estimate for $\Theta^{\nu,m}$:
\begin{align*} 
    \sup_{t \geq T_0} \left( \| \Theta^{\nu,m}_t \|^2_{L^2_x}
    +
    2\nu \int_{T_0}^t \| \nabla \Theta^{\nu,m}_s \|^2_{L^2_x} \mathd s \right)
    \leq
    \| \Theta^{\nu,m}_{T_0} \|^2_{L^2_x}
    \leq
    \| \theta^\nu_{T_0}\|_{L^q_x}^q,
    \quad
    \PP\mbox{-almost surely}.
\end{align*}
Using the above inequalities and mutatis mutandis in \eqref{eq:inequality.I3}, the analogue of \eqref{eq:inequality_g} becomes:
\begin{align*}
    G^{\nu,m}_{\beta,T_0} (\varrho)& \leq (d + 2\alpha) \varrho^{- d-2\alpha-2\beta} \int_0^\varrho G^{\nu,m}_{\beta,T_0} (\rho) \rho^{d-1+2\alpha+2\beta} \mathd \rho
  \\
  &\quad+ 
  \frac{\varrho^{1-2\alpha-2\beta}}{c_1}
  \mathbb{E}_{T_0} \left[ \int_{T_0}^T \mathcal{T}_{b^\nu,\Theta^{\nu,m}}(s,\varrho) \mathd s  \right] 
  \\
  &\quad+ 
  \varrho^{2-2\alpha-2\beta} K (1+T) \| \theta^\nu_{T_0}\|_{L^q_x}^q.
\end{align*}
From here the claim follows exactly as in the proof of \autoref{lem:gron_increments}.
\end{proof}

Given the anomalous regularization for $\Theta^{\nu,m}$, we can obtain anomalous integrability of $\theta^\nu$ with the iterative procedure described above. The iteration step is described by the following lemma.

\begin{lem} \label{lem:regularization-conditional-bound}
Suppose the assumptions of \autoref{cor:Theta.gronwall} are satisfied with parameters $\beta,\eta,\varrho_0$ independent of $q\in [2,\infty)$, and define
\begin{equation}\label{eq:defn_sigma_beta}
    \sigma := \frac{2d}{d-\beta}.
\end{equation}
Then for any $t_\ast \in (0,T-T_0)$, $t_\ast \leq 1$ it holds
\begin{align} \label{eq:regularization-conditional-bound}
\mathbb{E}_{T_0} \left[  \sup_{t \in [T_0+t_\ast,T]}   \| {\theta}^{\nu}_t \|^q_{L^{\sigma q/2}_x} \right]
&\lesssim
t_\ast^{-1/2} M_\ast \| \theta^\nu_{T_0} \|^{q}_{L^q_x},
\quad
\PP\mbox{-almost surely},
\end{align}
where the implicit constant is independent of $\nu,q,t_\ast$, and
\begin{align*}
 M_\ast :=  \varrho_0^{-\beta} + K(1+T) \left( \frac{2\beta}{d+2\alpha+2\beta} - \eta \right)^{-1}. 
\end{align*}
\end{lem}

\begin{proof}
Under our assumptions, we have a control on $\seminorm{\Theta^{\nu,m} }_{\beta,\varrho_0 \mid \mathcal{F}_{T_0}}^2$ as in the statement of \autoref{cor:Theta.gronwall}.
By the assumption $f^\nu \equiv 0$, we have the $\PP$-a.s. inequality
\begin{equation}\label{eq:regularization-conditional-bound-proof1}
    \| \Theta^{\nu,m} \|_{L^\infty_{[T_0,T_0+t_\ast]} L^2_x} \leq \| \theta^\nu \|^{q/2}_{L^\infty_{[T_0,T_0+t_\ast]} L^q_x} \leq \| \theta^\nu_{T_0} \|^{q/2}_{L^q_x}.
\end{equation}
Applying in sequence Sobolev embeddings, \cref{lem:besov.higher.time},  \autoref{cor:Theta.gronwall}, estimate \eqref{eq:regularization-conditional-bound-proof1}, the assumption $t_\ast \leq1$, and Young's inequality, we obtain
\begin{align*}
 \int_{T_0}^{T_0+t_\ast} \EE_{T_0}\left[\| \Theta^{\nu,m}_t\|_{L^\sigma_x}^2 \right] \mathd t 
 &\lesssim
 \int_{T_0}^{T_0+t_\ast} \EE_{T_0} \left[\| \Theta^{\nu,m}_t\|_{H^{\beta/2}_x}^2  \right] \mathd t
 \\
&\lesssim
t_\ast^{1/2} \| \theta^\nu_{T_0} \|^{q/2}_{L^q_x} \left(  \frac{\| \theta^\nu_{T_0} \|^{q/2}_{L^q_x}}{\varrho_0^{\beta}} + \seminorm{\Theta^{\nu,m} }_{\beta,\varrho_0 \mid \mathcal{F}_{T_0}} \right)
\\
&\lesssim
t_\ast^{1/2} \| \theta^\nu_{T_0} \|^{q}_{L^q_x} \left(  \varrho_0^{-\beta} + K(1+T) \left( \frac{2\beta}{d+2\alpha+2\beta} - \eta \right)^{-1}  \right)
\\
& = t_\ast^{1/2} M_\ast \| \theta^\nu_{T_0} \|^{q}_{L^q_x},
\end{align*} 
where the implicit constant is independent of the parameters $\nu,m,q,t_\ast$. 

Since $\Theta^{\nu,m} \uparrow \Theta^{\nu}$ for $m \to \infty$, by conditional monotone convergence \cite[Theorem 10.1.7]{Dudley}  we have for every fixed $\nu \in (0,1)$:
\begin{align*} 
\int_{T_0}^{T_0+t_\ast} \EE_{T_0} \left[\| \theta^{\nu}_t\|_{L^{\sigma q/2}_x}^q \right] \mathd t 
=
\int_{T_0}^{T_0+t_\ast} \EE_{T_0} \left[\| \Theta^{\nu}_t\|_{L^\sigma_x}^2  \right] \mathd t   
    \lesssim
    t_\ast^{1/2} M_\ast \| \theta^\nu_{T_0} \|^{q}_{L^q_x}
    \quad
    \PP\mbox{-a.s.}
\end{align*}
As a consequence, using that conditional expection commutes with time integral and the map $t \mapsto \| \theta^\nu_t \|_{L^{\sigma q/2}_x}$ is $\PP$-a.s. non-increasing (recall \eqref{viscous_bound_variant} with $f^\nu \equiv 0$), we deduce the following: for every fixed $t_\ast \in (0,T-T_0)$ and $\nu \in (0,1)$ it holds:
\begin{align} \label{eq:integrability.theta_tau}
 \mathbb{E}_{T_0}\left[ \|\theta^\nu_{T_0+t_\ast}\|_{L^{\sigma q/2}_x}^q  \right] \lesssim t_\ast^{-1/2} M_\ast \| \theta^\nu_{T_0} \|^{q}_{L^q_x}  < \infty,
  \quad \PP\mbox{-a.s.}
\end{align}
Using the monotonicity of conditional expectation and again the fact that $t \mapsto \| \theta^\nu_t \|_{L^{\sigma q/2}_x}$ is $\PP$-a.s. non-increasing, we obtain
\begin{align*}
    \mathbb{E}_{T_0} \left[  \sup_{t \in [T_0+t_\ast,T]}  \| {\theta}^{\nu}_t \|^q_{L^{\sigma q/2}_x} \right]
    \leq \mathbb{E}_{T_0} \left[  \| {\theta}^{\nu}_{T_0+t_\ast} \|^q_{L^{\sigma q/2}_x} \right]
\end{align*}
which combined with \eqref{eq:integrability.theta_tau} yields the desired \eqref{eq:regularization-conditional-bound}.
\end{proof}

We are now finally ready to complete the iteration argument.
\begin{prop}\label{prop:abstract_anomalous_integrability}
    Under the assumptions of \autoref{lem:regularization-conditional-bound}, for any $p\in [2,\infty)$ and $t\in (0,1\wedge T]$ it holds
    \begin{align*}
        \mathbb{E} \left[\sup_{s \in [ t ,T]}  \| {\theta}^{\nu}_s \|_{L^{\infty}_x}^p  \right]
        \lesssim t ^{- \frac{d}{2\beta}} \| {\theta}^{\nu}_{0} \|_{L^p_x}^p\quad\forall \nu \in (0,1),
    \end{align*}
where the hidden constant in the above estimate depends on $M^\nu_\ast$ and the hidden constant in \eqref{eq:regularization-conditional-bound}, but does not depend on $p$, $\nu$, $\theta^\nu_0$, $T$.
\end{prop}

\begin{proof}
    We can iteratively apply \autoref{lem:regularization-conditional-bound} starting from parameters $q=p$ and $T_0=0$. More precisely, for fixed $t \in (0,1 \wedge T]$, let us set
    \begin{align*}
        q_n := p \left(\frac{\sigma}{2}\right)^n, \quad \tau_n := \tau_{n-1} + 2^{-n} t,\quad \forall n\geq 1
    \end{align*}
    for $\sigma$ defined by \eqref{eq:defn_sigma_beta}; note that $\sigma>2$ and so $q_n\geq p$.
    By estimate \eqref{eq:regularization-conditional-bound} applied with $q = q_{n}$, $T_0=\tau_n$ and $t_\ast = 2^{-n} t$, there exists a constant $L=L(K,T,d,\beta,c_1,c_2,\varrho_0,\eta)$ 
    such that
    \begin{align*}
        \mathbb{E}_{\tau_n} \left[\sup_{s \in [ \tau_{n+1},T]}  \| {\theta}^{\nu}_s \|_{L^{q_{n+1}}_x}^{q_n}   \right]
        \leq L\, 2^{\frac{n}{2}}\, t^{-\frac{1}{2}} \| \theta^\nu_{\tau_n}\|_{L^{q_n}_x}^{q_n}.
    \end{align*}
    Combined with conditional Jensen's inequality and the tower property, this yields
    \begin{align*}
        \mathbb{E} \left[\sup_{s \in [ \tau_{n+1},T]}  \| {\theta}^{\nu}_s \|_{L^{q_{n+1}}_x}^p   \right]
        \leq 
        L^{\frac{p}{q_n}}
        2^{\frac{np}{2 q_n}} 
        t^{-\frac{p}{2 q_n}}   \mathbb{E} \left[  \| {\theta}^{\nu}_{\tau_n} \|_{L^{q_{n}}_x}^p  \right].
    \end{align*}
    and iterating
    \begin{align*}
        \mathbb{E} \left[\sup_{s \in [ \tau_{n+1},T]}  \| {\theta}^{\nu}_s \|_{L^{q_{n+1}}_x}^p \right]  
        \leq 
        L^{\sum_{k=0}^n \frac{p}{q_k}}\,  2^{\sum_{k=0}^n \frac{kp}{2 q_k}}\, t^{-\frac12\sum_{k=0}^n \frac{p}{q_k}}  \| {\theta}^{\nu}_{0} \|_{L^p_x}^p.
    \end{align*}
Since for every $n$ we have
\begin{align*}
    \tau_{n+1} \leq t,
    \quad
    \sum_{k=0}^n \frac{p}{q_k}
    =
    \sum_{k=0}^n \frac{1}{(\sigma/2)^k}
    \leq
    \frac{\sigma}{\sigma-2}
    = \frac{d}{\beta},
    \quad
    \sum_{k=0}^n \frac{kp}{q_k}
    =
    \sum_{k=0}^n \frac{k}{(\sigma/2)^k}
    \leq
    \frac{2\sigma}{(\sigma-2)^2}
    =\frac{d}{\beta}\left(\frac{d}{\beta}-1\right),
\end{align*}
from the line above, monotone convergence and \cite[Lemma A.2]{GalLuo25} we deduce that
\begin{align*}
 \mathbb{E} \left[\sup_{s \in [ t ,T]}  \| {\theta}^{\nu}_s \|_{L^{\infty}_x}^p  \right]
   \leq L^{\frac{d}{\beta}}\, 2^{\frac{d}{\beta}\left(\frac{d}{\beta}-1\right)} t ^{- \frac{d}{2\beta} } \| {\theta}^{\nu}_{0} \|_{L^p_x}^p
\end{align*}
which yields the desired conclusion.
\end{proof}

We conclude this section presenting sufficient conditions on the drifts $b^\nu$ under which the criteria from \autoref{sec:general.criterion} and this one apply, yielding both anomalous regularization and anomalous integrability for linear transport SPDEs on $\R^d$ with random drifts. For simplicity, below we restrict to the subcritical case, namely enforce strict inequalities on $\alpha$; the critical case would require $\{b^\nu\}_{\nu \in (0,1)}$ to satisfy appropriate small-large decompositions in the style of \eqref{estimate_splitting_Euler}-\eqref{estimate_splitting_CZ}. Below, for $\gamma\in (0,1]$, $W^{\gamma,\infty}_x$ denotes the homogeneous Sobolev space, with associated seminorm $\llbracket \cdot \rrbracket_{\dot W^{\gamma,\infty}}=\llbracket\cdot\rrbracket_{C^\gamma_x}$.

\begin{thm}\label{thm:anomalous_linear}
    Let $W$ satisfy \autoref{ass:noise}, $\{\theta_0^\nu\}_{\nu \in (0,1)}\subset L^1_x\cap L^\infty_x$ and $\{f^\nu\}_{\nu \in (0,1)}\subset L^1_T (L^1_x\cap L^\infty_x)$; let $\{b^\nu\}_{\nu \in (0,1)}$ be a family of progressive divergence-free drifts, $\{b^\nu\}_{\nu \in (0,1)} \subset L^\infty_{\omega,[T_0,T]}L^s_x$ for some $s\in [2,\infty]$ and let $\{\theta^\nu\}_{\nu \in (0,1)}$ be the unique solutions to \eqref{eq:abstract_transport_SPDE}. Further assume that:
    \begin{itemize}
        \item[a)] either $\alpha\in [1/2,1)$ and there exists $\gamma\in (0,1]$ such that
        \begin{equation}\label{eq:anomalous_linear_cond1}
            \alpha<\frac{1+\gamma}{2}, \qquad
            \sup_{\nu \in (0,1)} \|b^\nu\|_{L^\infty_{\omega, T} \dot W^{\gamma,\infty}_x} <+\infty; \qquad \text{or}
        \end{equation}
        \item[b)] $\alpha\in (0,1/2)$ and there exists $p\in (d,\infty]$ such that
        \begin{equation}\label{eq:anomalous_linear_cond2}
            \alpha<\frac{1}{2}-\frac{d}{2p}, \qquad
            \sup_{\nu \in (0,1)} \|b^\nu\|_{L^\infty_{\omega, T} L^p_x} <+\infty.
        \end{equation}
    \end{itemize}
    Then conditions \eqref{eq:key_bound_regularization} and \eqref{eq:key_bound_regularization.Theta} are satisfied for $\beta=1-\alpha$.
    In particular, for every $\nu \in (0,1)$, the solution $\theta^\nu$ to \eqref{eq:general.viscous.SPDE} satisfies
    \begin{align*}
        & \llbracket \theta^\nu \rrbracket_{\tilde{L}^2_{\omega,T} \tilde B^\beta_{2,\infty}} \lesssim \| \theta^\nu_0\|_{L^2_x}^2 + \| f^\nu\|_{L^1_T L^2_x}^2,\\
        & \mathbb{E} \left[\sup_{s \in [ t ,T]}  \| {\theta}^{\nu}_s \|_{L^{\infty}_x}^2  \right]
        \lesssim (1\wedge t) ^{- \frac{d}{2(1-\alpha)}} (\| {\theta}^{\nu}_{0} \|_{L^p_x}^2 + \| f^\nu\|_{L^1_T L^\infty_x}^2)\quad\forall t\in (0,T]
    \end{align*}
    where the hidden constant does not depend on $\nu \in (0,1)$.
\end{thm}

\begin{proof}
    We only verify \eqref{eq:key_bound_regularization}, since \eqref{eq:key_bound_regularization.Theta} is identical up to replacing $\theta^\nu$ with $\Theta^{\nu,m}$.

    First assume that $\alpha\in [1/2,1)$ and \eqref{eq:anomalous_linear_cond1} holds. Fix any $\varrho_0<1\wedge r_0$, then for $\varrho \leq \varrho_0$ we have
    \begin{align*}
        \left|\mathbb{E} \left[ \int_0^T \mathcal{T}_{b^\nu,\theta^\nu}(s,\varrho) \mathd s \right]\right|
        & \leq 2 \varrho^\gamma \| b^\nu\|_{L^\infty_{\omega,T} \dot W^{\gamma,\infty}_x} \int_{\{|z|\leq 1\}} \| \delta_{\varrho z} \theta^\nu\|_{L^2_{\omega,T,x}}^2 \mathd z \\
        & \lesssim \varrho^{\gamma+2(1-\alpha)} \| b^\nu\|_{L^\infty_{\omega,T} \dot W^{\gamma,\infty}_x} \seminorm{\theta^\nu}_{1-\alpha,\varrho_0}^2.
    \end{align*}
    By assumption $\gamma+2(1-\alpha)>1$, so we can find $\varrho_0$ small enough such that \eqref{eq:key_bound_regularization} is satisfied.

    Consider now the case $\alpha\in (0,1/2)$ and assume that \eqref{eq:anomalous_linear_cond2} holds. Define parameters $q\in [2,\infty]$, $s\in [0,1)$, $\theta\in [0,1)$ by
    \begin{align*}
        \frac{1}{q}=\frac{1}{2}-\frac{1}{2p}, \quad s:= \frac{d}{2p},\quad \theta:=\frac{s}{1-\alpha}.
    \end{align*}
    Then by H\"older's inequality, Sobolev embeddings and \autoref{lem:final_besov} we obtain
    \begin{align*}
        \left|\mathbb{E} \left[ \int_0^T \mathcal{T}_{b^\nu,\theta^\nu}(s,\varrho) \mathd s \right]\right|
        & \leq 2 \| b^\nu\|_{L^\infty_{\omega,T}L^p_x} \mathbb{E}\left[ \int_{\{|z|\leq 1\}} \int_0^T \| \delta_{\varrho z} \theta^\nu\|_{L^q_x}^2 \mathd t \mathd z \right]\\
        & \lesssim \| b^\nu\|_{L^\infty_{\omega,T}L^p_x} \mathbb{E}\left[ \int_{\{|z|\leq 1\}} \int_0^T \| \delta_{\varrho z} \theta^\nu\|_{H^s_x}^2 \mathd t \mathd z \right]\\
        & \lesssim \| b^\nu\|_{L^\infty_{\omega,T}L^p_x} |\varrho|^{2(1-\theta) (1-\alpha)} \llbracket \theta^\nu\rrbracket_{1-\alpha,\varrho_0}^2
    \end{align*}
    for all $\varrho\leq \varrho_0$. By construction we have
    \begin{align*}
        2(1-\theta)(1-\alpha)>1
        \quad \Leftrightarrow \quad
        2(1-\alpha-s)>1
        \quad \Leftrightarrow \quad
        \alpha <\frac{1}{2}-s=\frac{1}{2}-\frac{d}{2p};
    \end{align*}
    therefore by assumption \eqref{eq:key_bound_regularization.Theta}, up to choosing $\varrho_0$ sufficiently small, \eqref{eq:key_bound_regularization} is satisfied.
\end{proof}

\subsection{Application to 2D active scalars}\label{subsec:integrability_active_scalars}
Before entering the case-by-case estimates, we fix the common setup. In view of the previous reduction, we work with zero external forcing $f^\nu\equiv 0$, and set $d=2$, $\beta:=1-\alpha$.
From now on we assume that $\theta^\nu$ solves \eqref{eq:nonlinear_viscous}, namely the drift is given by $u^\nu=\mathscr{R}\theta^\nu$.
Let $q\geq2$, fix an arbitrary base time $T_0\in[0,T)$ and let $\Theta^{\nu,m}=F_m(\theta^\nu)$ be the approximation of $|\theta^\nu|^{q/2}$ introduced in \autoref{subsec:integrability_strategy}.
In each of the regimes below, we only need to verify assumption \eqref{eq:key_bound_regularization.Theta} from \autoref{cor:Theta.gronwall} for this choice of parameters, namely
\begin{equation}\label{eq:anomalous_integrability_goal}
    \mathbb{E}_{T_0}\left[
\int_{T_0}^T
\mathcal{T}_{u^\nu,\Theta^{\nu,m}}(s,\varrho)\,\mathd s
\right]
\leq
c_1 \eta\,\varrho\,
\seminorm{\Theta^{\nu,m}}_{1-\alpha,\varrho_0\mid\mathcal{F}_{T_0}}^2
\quad
\forall \varrho\in(0,\varrho_0),
\end{equation}
with $\eta>0$ sufficiently small and with $\eta,\varrho_0$ independent of $\nu,m,q$. Let us postpone the verification of \eqref{eq:anomalous_integrability_goal} in the ranges of parameters $\alpha,p$ described by \autoref{thm:integrability} to the subsections below; assuming \eqref{eq:anomalous_integrability_goal}, we can complete the
\begin{proof}[Proof of \autoref{thm:integrability}]
    By the reduction presented at the beginning of \autoref{subsec:integrability_strategy}, we can reduce ourselves to the case $f\equiv 0$.    
    Since \eqref{eq:anomalous_integrability_goal} holds, \autoref{lem:regularization-conditional-bound} applies and so does \autoref{prop:abstract_anomalous_integrability}; taking $T_0=0$ and using \eqref{eq:bound.forcing.anomalousintegrability} one then obtains estimate \eqref{eq:anomalous_integrability_general}. Finally, interpolating with the bound \eqref{viscous_bound_variant} yields the proof of estimate \eqref{eq:intro_anomalous_integrability}.
\end{proof}

We proceed with verifying \eqref{eq:anomalous_integrability_goal}; up to relabelling $\eta$, henceforth without loss of generality we may set $c_1=1$.

\subsubsection{Case 1: $\mathscr{R}=\mathscr{R}_{BS},\ p=2,\ \alpha \leq 1/2$} \label{ssec:integrability1}

Compared to Case a) from \autoref{thm:anomalous_linear}, this regime is borderline, since $u^\nu\in L^\infty_{\omega,T} H^1_x$, but $H^1_x$ does not embed in $L^\infty_x$ and we want to allow $\alpha=1/2$.
Similarly to \autoref{sss:regularity_case1}, set
\begin{align*}
    s:=\frac{2}{3},\quad \delta:=\frac{1}{3(1-\alpha)};
\end{align*}
as therein, for any $\eps>0$, we can employ the splitting $u^\nu=u^{\nu,1}+u^{\nu,2}$ satisfying
\begin{align*}
    \sup_{\nu \in (0,1)} \| u^{\nu,1}\|_{L^\infty_{\omega,T} \dot W^{2/3,3}_x} \lesssim \eps,\quad 
    \sup_{\nu \in (0,1)} \| u^{\nu,2}\|_{L^\infty_{\omega,T} \dot W^{1,3}_x} \lesssim M_\eps.
\end{align*}
Going through similar computations as in \autoref{sss:regularity_case1} and applying \autoref{lem:final_besov}, one thus finds
\begin{align*}
    \mathbb{E}_{T_0} \left[ \int_{T_0}^T \mathcal{T}_{u^\nu,\Theta^{\nu,m}}(t,\varrho) \mathd t \right]
    & \lesssim
    (\varepsilon \varrho^{2/3} +M_\eps \varrho) \mathbb{E}_{T_0} \left[ \int_{T_0}^T \int_{\{| z |\leqslant 1\}} \norm{\delta_{\varrho z} \Theta^{\nu,m}_t}_{H^{\delta(1-\alpha)}_x}^2  \mathd z \mathd t \right]
   \\ 
   &\lesssim (\varepsilon \varrho^{2/3}+M_\eps \varrho) \varrho^{4/3-2\alpha} \seminorm{\Theta^{\nu,m}}_{1-\alpha,\varrho_0\mid \mathcal{F}_{T_0}}^2\\ 
   & \lesssim 
   (\varepsilon+M_\eps \varrho_0^{4/3-2\alpha})\varrho\seminorm {\Theta^{\nu,m}}_{1-\alpha,\varrho_0\mid \mathcal{F}_{T_0}}^2
\end{align*}
where we used that $2-2\alpha\geq 1$ since $\alpha\leq 1/2$.
Hence, by first fixing $\eps>0$ and then choosing $\varrho_0>0$ small enough, condition \eqref{eq:anomalous_integrability_goal} is satisfied with $\eta := C(\varepsilon+M_\eps \varrho_0^{4/3-2\alpha})$, which we can make arbitrarily small.

\subsubsection{Case 2: $\mathscr{R}=\mathscr{R}_{BS},\ p \in (2,\infty),\ \alpha \leq 1-\frac{1}{p}$} \label{ssec:integrability2}

This is effectively a subcase of Case~a) from \autoref{thm:anomalous_linear}, in light of the Sobolev embedding $W^{1,p}_x \hookrightarrow C^\gamma_x$ for $\gamma:=1-2/p$; the only difference is that we can cover the critical equality $\alpha = 1-1/p$ thanks to splitting arguments.
Similarly to \eqref{estimate_splitting_Euler}, we can employ a splitting $u^\nu=u^{\nu,1}+u^{\nu,2}$ with
\begin{align*}
    \sup_{\nu \in (0,1)} \| u^{\nu,1}\|_{L^\infty_{\omega,T} C^\gamma_x} 
    &\lesssim 
    \sup_{\nu \in (0,1)} \| u^{\nu,1}\|_{L^\infty_{\omega,T} W^{1,p}_x} 
    \lesssim 
    \eps,
    \\
    \sup_{\nu \in (0,1)} \| u^{\nu,1}\|_{L^\infty_{\omega,T} C^{\gamma+1/p}_x}
    &\lesssim 
    \sup_{\nu \in (0,1)} \| u^{\nu,1}\|_{L^\infty_{\omega,T} W^{1,2p}_x} 
    \lesssim 
    M_\eps.
\end{align*}
Therefore by the definition of $\seminorm {\Theta^{\nu,m}}_{1-\alpha,\varrho_0\mid \mathcal{F}_{T_0}}$ one finds
\begin{align*}
    \mathbb{E}_{T_0} \left[ \int_{T_0}^T \mathcal{T}_{u^\nu,\Theta^{\nu,m}}(t,\varrho) \mathd t  \right] 
    &\lesssim
    (\eps \rho^\gamma + M_\eps \rho^{\gamma+1/p}) \mathbb{E}_{T_0} \left[ \int_{T_0}^T \int_{\{| z | \leqslant 1\}} \| \delta_{\varrho z} \Theta^{\nu,m}_t \|_{L^2_x}^2 \mathd z \mathd t  \right]
   \\ 
    &\lesssim  (\eps + M_\eps \varrho_0^{1/p}) \rho^{\gamma + 2(1-\alpha)} \seminorm{\Theta^{\nu,m} }_{1-\alpha,\varrho_0 \mid \mathcal{F}_{T_0}}^2.
\end{align*}
By assumption $\gamma+2(1-\alpha)\geq 1$, therefore up to choosing $\eps>0$ and $\rho_0>0$ appropriately the reabsorb constants, condition condition \eqref{eq:anomalous_integrability_goal} is satisfied.

\subsubsection{Case 3: $\mathscr{R}=\mathscr{R}_{CZ}$, $p\in (2,\infty),\ \alpha\leq \frac{1}{2}-\frac{1}{p}$ }\label{ssec:integrability3}

This case is analogous to Case b) from \cref{thm:anomalous_linear}, up to using the splitting \eqref{estimate_splitting_CZ} to handle the critical case $\alpha=1/2-1/p$, similarly to Case 2 above. We leave the details to the reader.

\section{Existence and uniqueness results for inviscid SPDEs}\label{sec:existence_uniqueness}
This section is divided into two parts. In \autoref{subsec:weak_existence}, we prove the second part of \autoref{thm:regularization}, establishing weak existence of solutions to \eqref{eq:nonlinear_transport}.
More precisely, we show tightness of the family $\{\theta^\nu\}_{\nu \in (0,1)}$ and characterize any limit point as $\nu\to 0^+$ as a weak solution to \eqref{eq:nonlinear_transport}, in the sense of \autoref{martingale_sol}. In \autoref{subsec:pathwise_uniqueness} we show pathwise uniqueness for \eqref{eq:nonlinear_transport} in the setting of \autoref{ass:strong_regularization}, thereby completing the proof of \autoref{thm:intro_wellposedness}. To simplify the notation, from now on we denote by $\Bar{M}$ the square root of the right-hand side appearing in \eqref{viscous_uniform_bound}, including the implicit constant.

\subsection{Weak existence}\label{subsec:weak_existence}
We present here the proof of the second part of \autoref{thm:regularization}.
Namely, we show tightness of the family $\{\theta^\nu\}_{\nu \in (0,1)}$ and characterize any limit point as $\nu\to 0^+$ as a weak solution to \eqref{eq:nonlinear_transport}, in the sense of \autoref{martingale_sol}.

Throughout this section, we will always consider the following setting.
We fix parameters $(r,p,\alpha,\beta)$ satisfying \autoref{ass:exponent.beta}; recall in particular that $\beta$ is defined as in \eqref{eq:definition.regularity}-\eqref{eq:definition.regularity_CZ}.
The initial condition and external forcing satisfy $\theta_0\in L^r_x\cap L^p_x$ and $f\in L^1_T (L^r_x\cap L^p_x)$, respectively.
We consider a given vanishing viscosity scheme $\{(\theta^\nu_0,f^\nu,\theta^\nu)\}_{\nu\in (0,1)}$ in the sense of \autoref{defn:vanishing_scheme}. 
The noise $W$ is assumed to satisfy \autoref{ass:noise} for some given constants $c_1,c_2,r_0>0$, which covers the Kraichnan model as a special subcase.

As a consequence of \autoref{prop_cases} and \autoref{lem:gron_increments}, together with \autoref{lem:besov.higher.time}, we obtain the following result.
\begin{cor}\label{cor:space_compactness}
    Under the same assumptions of \autoref{thm:regularization}, we have for every $\delta\in (0,\beta)$
    \begin{align*}
        \sup_{\nu\in(0,1)}\llbracket {\theta}^{\nu}\rrbracket_{\Tilde{L}^2_{\omega,T}\Tilde{B}^{\beta}_{2,\infty}}
        \lesssim 1,
        \quad
        \sup_{\nu\in(0,1)}\mathbb{E}\left[\|{\theta}^{\nu}\|^2_{L^2_TH^{\beta-\delta}_x}\right]&\lesssim 1.
    \end{align*}
\end{cor}
Moreover, the time increments of $\theta^{\nu}$ can also be controlled uniformly in $\nu$, as stated in the following lemma.
\begin{lem}\label{time_compactness}
For every $\sigma>2$, $\gamma\in (0,1/2)$ and $n\in \mathbb{N}$, it holds that
\begin{align*}
    \sup_{\nu\in (0,1)}\mathbb{E}\left[\left\|\theta^\nu-\int_0^\cdot f^\nu_s \dd s\right\|_{C^{\gamma}_TH^{-\sigma}_x}^n\right]&\lesssim 1.
\end{align*}
\end{lem}

\begin{proof}
To control the time increments of $\theta^\nu - \int_0^\cdot f^\nu_s \dd s$ we can use similar computations as those in \cite[Section A.2]{crippa2025zero} used to obtain \cite[equation (A.8)]{crippa2025zero}, with the right-hand side therein replaced by $\Bar{M}^{2n}$.
The main difference with respect to \cite{crippa2025zero} is the presence of the nonlinearity, that we bound as follows. 

   When $\mathscr{R}=\mathscr{R}_{CZ}$, since $u^\nu$ is divergence-free, by \eqref{eq:properties_CZ} and Sobolev embeddings we have the $\PP$-a.s. estimate
   \begin{align*}
        \| \nabla\cdot( u^\nu_t \theta^\nu_t)\|_{H^{-\sigma}_x}
        \lesssim \| u^\nu_t\theta^\nu_t\|_{L^1_x}
        \lesssim \| \theta^\nu_t\|_{L^2_x}^2 \lesssim \bar M^2,\quad\forall t\in [0,T];
    \end{align*}
    similarly, when $\mathscr{R}=\mathscr{R}_{BS}$, by H\"older's inequality and \eqref{eq:properties_BS_integrability} we find
    \begin{align*}
        \| \nabla\cdot( u^\nu_t \theta^\nu_t)\|_{H^{-\sigma}_x}
        \lesssim \| u^\nu_t\theta^\nu_t\|_{L^r_x}
        \lesssim \| u^\nu_t\|_{L^{\frac{2r}{2-r}}_x} \|\theta^\nu_t\|_{L^2_x}
        \lesssim \| \theta^\nu_t\|_{L^r_x\cap L^2_x}^2
        \lesssim \bar M^2,\quad\forall t\in [0,T].
    \end{align*}
    
    In both cases it follows that
    \begin{align*}
        \sup_{\nu\in (0,1)} \left\| \left\| \int_0^\cdot u^\nu_s\cdot \nabla \theta^\nu_s \mathd s \right\|_{W^{1,\infty}_T H^{-\sigma}_x} \right\|_{L^\infty_\omega}
        \lesssim \Bar{M}^2.
    \end{align*}
    The other terms can be bounded as in \cite[Section A.2]{crippa2025zero}; we omit the straightforward details.
\end{proof}

Fix $\delta\in (0,\beta/2)$ small; let $\mathcal{B}_{r,p}$ denote the closed ball of radius $2\Bar{M}$ in $L^{r}_x\cap L^{p}_x$, endowed with the weak topology. Let us define the Polish space
\begin{align}\label{polish_space_tightness}
    \mathcal{E}:=
    C_T \mathcal{B}_{r,p} \cap L^2_T \tilde{H}^{\beta-2\delta}_x,
\end{align}
where we recall that the weighted Sobolev spaces $\tilde H^s_x$ were introduced in \autoref{subsec:weighted_spaces}.

Let us also consider the class of processes $\mathcal{U}_p$ from \cite[Definition 1.2]{BaGa25}.
Given $p\in (1,+\infty)$, we say that a process $\varphi$ with paths in $L^\infty_{\omega,T}L^p_x$ belongs to $\mathcal{U}_p$ if for every $\eps>0$ there is $q\in (p,+\infty)$ such that $\varphi$ can be written as
\begin{equation}\label{eq:defn_Up}
    \varphi=\varphi^{<}+\varphi^{>},\quad
    \mbox{with}\quad\|\varphi^{<}\|_{L^\infty_{\omega,T}L^p_x}\leq \eps,\quad \|\varphi^{>}\|_{L^\infty_{\omega,T}L^q_x}<+\infty.
\end{equation}

With these notations in mind, the second part of \autoref{thm:regularization} is a consequence of the following proposition.

\begin{prop}\label{prop:convergence_law}
The laws of $\{\theta^\nu\}_{\nu \in (0,1)}$ are tight on $\mathcal{E}$. Moreover, if $\mu$ is any accumulation point of these laws as $\nu\downarrow 0$, then there exists a weak solution 
\begin{align*}
\theta \in L^{\infty}_{\omega,T}(L^r_x\cap L^p_x)
\cap
\Tilde{L}^2_{\omega,T}\Tilde{B}^{\beta}_{2,\infty} \cap \mathcal{U}_p
\end{align*}
of \eqref{eq:nonlinear_transport} in the sense of \autoref{martingale_sol}, with law $\mu$ as $\mathcal{E}$-valued random variable and almost surely with weakly continuous paths in $L^2_x$.
\end{prop}

\begin{proof}
We first check tightness in $C_T \mathcal{B}_{r,p}$ and then upgrade it to $\mathcal{E}$.
First, note that by \autoref{defn:vanishing_scheme} $f^\nu\to f$ in $L^1_T(L^r_x\cap L^p_x)$, therefore $\int_0^\cdot f^\nu_s \dd s\to \int_0^\cdot f_s \dd s$ in $C_T (L^r_x\cap L^p_x)$, which continuously embeds in $C_T \mathcal{B}_{r,p}$; since this family converges to a limit, it is automatically tight.
Next we focus on $\tilde\theta^\nu_t:=\theta^\nu_t-\int_0^t f^\nu_s \dd s$; in this case, tightness easily follows from \cite[Lemma 2.2 (1)]{crippa2025zero}, in light of the uniform estimates given by \eqref{viscous_uniform_bound} and \autoref{time_compactness}.
Since tightness behaves well under summation, we conclude that $\{\theta^\nu\}_\nu=\{\tilde\theta^\nu+\int_0^\cdot f^\nu\}_\nu$ is tight in $C_T \mathcal{B}_{r,p}$ as well.

Now fix $\eps>0$, and consider any compact set $K_\eps$ in $C_T \mathcal{B}_{r,p}$ such that $\PP(\theta^\nu\notin K_\eps)\leq \eps/2$.
Set $K_{\eps,R}:=K\cap \{f: \| f\|_{L^2_T H^{\beta-\delta}_x} \leq R\}$; noting that the embedding $\mathcal{B}_{r,p}\hookrightarrow \tilde H_s^{-\delta}$ is compact by \eqref{eq:weighted_spaces_embeddings}, it follows from \autoref{lem:compactness_paths_weights}-ii) that $K_{\eps,R}$ is compact in $\mathcal{E}$.
By \autoref{cor:space_compactness} and Markov's inequality, we see that
\begin{align*}
    \PP(\theta^\nu\notin K_{\eps,R}) \leq \PP(\theta^\nu\notin K_{\eps}) + \PP(\|\theta^\nu\|_{L^2_T H^{\beta-\delta}_x} > R) \leq \frac{\eps}{2} + \frac{C}{R^2}
\end{align*}
for some constant $C>0$; since the estimate is uniform in $\nu \in (0,1)$, upon taking $R$ sufficiently large, this shows tightness of the laws of $\{\theta^\nu\}_{\nu \in (0,1)}$ in $\mathcal{E}$.

Next, observe that we can rewrite $W=\mathcal{C}^{1/2}\mathcal{W}$, where $\mathcal{W}$ is a cylindrical Brownian motion on the space of square-integrable divergence-free vector fields, and $\mathcal{C}^{1/2}$ is defined as in \cite{DrGaPa25}. $\mathcal{W}$ can be identified with a family of independent real Brownian motions $\{W^{k}\}_{k\in \N}$, namely as a process taking values in $C_T \R^{\N}$. Notice that the law of $\{W^{k}\}_{k\in \N}$ is independent of $n$.

We now apply Prokhorov's theorem together Skorokhod's representation theorem. The following arguments are standard and are detailed, for instance, in \cite[Chapter 2]{flandoli2023stochastic}.
Up to extracting a non-relabeled subsequence, there exist an auxiliary probability space $(\tilde{\Omega},\tilde{\mathcal{F}},\tilde{\mathbb{P}})$ and processes $(\tilde{\theta}^{\nu_n},\{\tilde{W}^{\nu_n,k}\}_{k\in \N})$, $(\tilde{\theta},\{\tilde{W}^{k}\}_{k\in \N})$ defined on it such that the processes $(\tilde{\theta}^{\nu_n},\{\tilde{W}^{\nu_n,k}\}_{k\in \N})$ and $({\theta}^{\nu_n},\{{W}^{k}\}_{k\in \N})$ have the same law on $\mathcal{E}\times C_T\R^{\N}$, and the convergences
\begin{align} \label{convergence_theta}
    \tilde{\theta}^{\nu_n}\rightarrow \tilde{\theta} \quad\, \textit{in } \mathcal{E},\qquad
    \{\tilde{W}^{\nu_n,k}\}_{k\in \N}\rightarrow \{\tilde{W}^{k}\}_{k\in \N} \quad \textit{in } C_T\R^{\N},
\end{align}
hold $\tilde{\mathbb{P}}$-almost surely. Denoting by $(\tilde{\mathcal{F}}^n_t)_{t\geq 0}$ the complete right-continuous filtration generated by $(\tilde{\theta}^{\nu_n},\tilde{W}^{\nu_n})$, the process $\tilde{\theta}^{\nu_n}$ is the unique probabilistically strong, analytically weak solution of \eqref{eq:nonlinear_viscous} on the filtered probability space $(\tilde{\Omega},\tilde{\mathcal{F}},(\tilde{\mathcal{F}}^n_t)_{t\geq 0},\tilde{\mathbb{P}})$ with initial datum $\theta^{\nu_n}_0$, forcing term $f^{\nu_n}$, and noise $\tilde{W}^{\nu_n}=\mathcal{C}^{1/2}\tilde{\mathcal{W}}^{\nu_n}$ given by \cite[Lemma 4.7]{bagnara2025regularization}. 
Therefore, \autoref{prop_cases} also holds on the auxiliary probability space $(\tilde{\Omega},\tilde{\mathcal{F}},\tilde{\mathbb{P}})$. In particular, by \autoref{cor:space_compactness}
\begin{align}\label{uniform_bound_Besov}
    \sup_{n\geq 1}\, \llbracket \tilde{\theta}^{\nu_n}\rrbracket_{\Tilde{L}^2_{\omega,T}\Tilde{B}^{\beta}_{2,\infty}}
    \lesssim 1.
\end{align}

Moreover, \eqref{convergence_theta} together with \eqref{viscous_uniform_bound} and \eqref{eq:weighted_spaces_embeddings} implies that
\begin{equation}\label{eq:tightness_proof}
    \tilde{\theta}^{\nu_n}\rightharpoonup \tilde{\theta} \quad\text{in } L^2(\tilde{\Omega}\times [0,T]\times \R^2), 
    \qquad
    \tilde{\theta}^{\nu_n}\to \tilde{\theta} \quad\text{in } L^2(\tilde{\Omega}\times[0,T]; L^2_{{\rm loc}}(\R^d)).
\end{equation}

By \eqref{uniform_bound_Besov} and lower semicontinuity of the $\Tilde{L}^2_{\omega,T}\Tilde{B}^{\beta}_{2,\infty}$ seminorm (cf. \cite[Remark 2.17]{DrGaPa25}), we deduce that $\tilde\theta\in \tilde L^2_{\omega,T} \tilde B^\beta_{2,\infty}$ with
\begin{align*}
    \llbracket \tilde{\theta}\rrbracket_{\Tilde{L}^2_{\omega,T}\Tilde{B}^{\beta}_{2,\infty}}
    \leq \liminf_{n\to\infty} \llbracket \tilde{\theta}^{\nu_n}\rrbracket_{\Tilde{L}^2_{\omega,T}\Tilde{B}^{\beta}_{2,\infty}}
    \lesssim 1.
\end{align*}
Following verbatim the proof of \cite[Lemma 3.5]{flandoli2021scaling}, one can show that, for every $q\in [r,p]$, the process $\tilde{\theta}$ has $\tilde{\mathbb{P}}$-a.s. weakly continuous trajectories in $L^q_x$, and satisfies $\tilde{\mathbb{P}}$-almost surely
\begin{align}\label{bound_inviscid}
    \| \tilde{\theta} \|_{L^\infty_{T} \left(L^r_x\cap L^p_x\right)}
    & \leq \norm{\theta_0}_{L^r_x\cap L^p_x}+\| f \|_{L^1_T \left(L^r_x\cap L^p_x\right)} < \infty,\\
     \label{energy_inviscid} \| \tilde{\theta}_t \|_{L^2_x}^2
    & \leq  \norm{\theta_0}_{L^2_x}^2+\int_0^t \langle f_s,\tilde{\theta}_s\rangle \mathd s ,   \qquad \mbox{for every }t\in [0,T].
\end{align}
Moreover, for any $\eps>0$, a splitting of the viscous approximations like \eqref{eq_viscous_splitting} holds in the auxiliary probability space, with estimates uniform in $\nu\in (0,1)$. By lower semicontinuity of $L^\infty_{\omega,T} L^q_x$-norms, it follows that $\tilde\theta$ satisfies a similar decomposition, with each component satisfying the same estimates; therefore $\tilde{\theta}\in \mathcal{U}_p$.

Denoting by $(\tilde{\mathcal{F}}_t)_{t\geq 0}$ the complete right-continuous filtration generated by $\tilde{\theta}$ and $\tilde{W}$, we are left to show that the tuple $(\tilde{\Omega},\tilde{\mathcal{F}},(\tilde{\mathcal{F}}_t)_{t\geq 0},\tilde{\mathbb{P}},\tilde{\theta},\tilde{W})$ is a weak solution to \eqref{eq:nonlinear_transport} in the sense of \autoref{martingale_sol}. 

The convergence of the linear terms easily follows by \eqref{viscous_uniform_bound}, \eqref{convergence_theta}, \eqref{bound_inviscid}, dominated convergence and \cite[Lemma 4.3]{bagnara2025no} for deterministic and stochastic integrals respectively.
So we only need to focus on the convergence of the nonlinear terms.
In the Euler case $\mathscr{R}=\mathscr{R}_{BS}$, since $r\in (1,2)$, by \eqref{eq:properties_BS_integrability}, \eqref{eq:tightness_proof}, and weak-strong convergence we deduce that
\begin{equation*}
    \tilde{u}^{\nu_n}\rightharpoonup \tilde{u} \quad\text{in } L^2(\tilde{\Omega}\times [0,T]; L^q(\R^2)), \qquad
    \tilde{u}^{\nu_n} \theta^{\nu_n} \rightharpoonup \tilde u \tilde{\theta} \quad\text{in } L^2(\tilde{\Omega}\times[0,T]; L^1_{{\rm loc}}(\R^d)),
\end{equation*}
for $q\in (2,\infty)$ given by $1/q=1/r-1/2$. Then, using integration by parts, it follows that for any $\varphi\in C^\infty_c(\R^2)$ and any $t\in [0,T]$
\begin{align*}
    \int_0^t \langle u^{\nu_n}_s\cdot \nabla \theta^{\nu_n}_s, \varphi\rangle \mathd s 
    \to 
    \int_0^t \langle u_s\cdot \nabla \theta_s, \varphi\rangle \mathd s.
\end{align*}
The case of Calderón--Zygmund operators is similar, up to taking $q=2$ and applying \eqref{eq:properties_CZ} instead of \eqref{eq:properties_BS_integrability}.   
\end{proof}

We can now complete the:
\begin{proof}[Proof of \autoref{thm:regularization}]
    The first part of the statement follows from the application of \autoref{lem:gron_increments}, whose assumptions were verified in each case in \autoref{prop_cases}. The second part is covered by \autoref{prop:convergence_law}.
\end{proof}

\begin{proof}[Proof of \autoref{cor:existence}]
    The first part of the statement is a consequence of \autoref{thm:regularization}; in particular, the law of the weak solution $\theta$ is the limit as $n\to\infty$ of the laws of $\theta^{\nu_n}$, for a suitable sequence $\nu_n\downarrow 0$, in the topology of $\mathcal{E}$.
    In particular, since we have the freedom to choose the vanishing viscosity scheme as we like, similarly to \autoref{rmk:bound_viscoud_data} we can construct it so that
    \eqref{bound_viscoud_data} holds, as well as
    \begin{equation}\label{eq:yet_another_unif_bound}
        \sup_{\nu\in (0,1)} \| f^\nu\|_{L^1_T L^{p_\ast}_x}\leq \| f\|_{L^1_T L^{p_\ast}_x}.
    \end{equation}
    By lower semicontinuity of the relevant norms, for any choice of $p_\ast\in [p,\infty]$ and $t\in [0,T)$, the map
    \begin{align*}
        F:\mathcal{E}\to [0,+\infty],\quad F(x):=\sup_{s\in [t,T]} \|x_s\|_{L^{p_\ast}_x}^p 
    \end{align*}
    is lower semicontinuous. Therefore, under the assumptions of \autoref{thm:integrability}, estimate \eqref{eq:corollary_anomalous_integrability} follows from the uniform-in-$\nu$ estimate \eqref{eq:intro_anomalous_integrability}, estimates \eqref{bound_viscoud_data}-\eqref{eq:yet_another_unif_bound} and the Portmanteau theorem.
\end{proof}

\begin{rmk}
    As clear from the proofs, the conclusions of \autoref{thm:regularization} and \autoref{cor:existence} hold for any choice of noise $W$ satisfying \autoref{ass:noise}.
\end{rmk}

\begin{rmk}\label{rem:right_continuity_weak_sol_t=0}
The weak solutions $\tilde \theta$ constructed in \autoref{prop:convergence_law} always satisfy the property that
\begin{align*}
    \tilde \PP\left(\lim_{t\to 0^+} \| \tilde\theta_t-\theta_0\|_{L^2_x}^2=0\right)=1.
\end{align*}
Indeed, for $\tilde\PP$-a.e. $\tilde\omega$, by combining the weak continuity of $t\mapsto \tilde\theta_t$, the lower semicontinuity of $L^2_{x}$-norm and the pathwise estimate \eqref{energy_inviscid}, one has
\begin{align*}
    \|\theta_0\|_{L^2_x}^2
    \leq \liminf_{t\to 0^+}\|\tilde \theta_{t}\|_{L^2_x}^2
    \leq \limsup_{t\to 0^+}\|\tilde \theta_{t}\|_{L^2_x}^2
    \leq \limsup_{t\to 0^+}\left( \norm{\theta_0}_{L^2_x}^2+\int_0^t \langle f_s,\tilde{\theta}_s\rangle \mathd s\right)
    = \|\theta_0\|_{L^2_x}^2
\end{align*}
which again by properties of weak convergence in $L^2_x$ implies the claim.
\end{rmk}

\subsection{Pathwise uniqueness and strong existence}\label{subsec:pathwise_uniqueness}

The goal of this section is to complete the proof of \autoref{thm:intro_wellposedness}.
We start by proving some novel uniqueness results for \eqref{eq:nonlinear_transport} under \autoref{ass:strong_regularization}, see \autoref{thm_pathwise_uniq_order_0} and \autoref{thm_pathwise_uniq_euler}.
To this end, we need the following condition; for consistency with \autoref{Appendix:approx_kernels}, we formulate it on $\R^d$, even though in this section we only need it for $d=2$.

\begin{ass}\label{ass:noise2}
    The noise covariance $C$ satisfies \eqref{eq:properties.C}. Moreover, there exist $\alpha\in (0,1)$ and constants $c_1'>0$, $c_2'\geq 0$ such that
    \begin{equation}\label{eq:ass_noise2}
        C(x)=c_1' C_\alpha(x) + R(x) \quad \text{and} \quad \int_{\R^d} (1+|\xi|^2) |\hat R(\xi) |\mathd \xi \leq c_2',
    \end{equation}
    where $C_\alpha$ is given by \eqref{eq:isotropic.covariance}.
\end{ass}

\begin{rmk}
    \autoref{ass:noise2} should be regarded as a ``Fourier analogue'' of \autoref{ass:noise}.
    In fact, it is a stronger requirement, as condition \eqref{eq:ass_noise2} implies \eqref{eq:local.expansion.noise}.
    Indeed, by \eqref{eq:ass_noise2} and properties of Fourier transform, it holds $R\in C^2_x$ with $\| R\|_{C^2_x}\lesssim c_2'$, therefore \autoref{ex:ass_noise} applies.
\end{rmk}

\begin{prop}\label{thm_pathwise_uniq_order_0}
Let $W$ satisfy \autoref{ass:noise2} and let $(\alpha,p,r,\mathscr{R})$ satisfy \autoref{ass:strong_regularization} with $\mathscr{R}=\mathscr{R}_{CZ}$.
Let $\theta_0\in L^2_x\cap L^p_x$ and $f\in L^1_T (L^2_x\cap L^{p}_x)$ be fixed.
Then, in the strictly subcritical case $\alpha<1/2-1/p$, pathwise uniqueness in $L^\infty_{\omega,T}(L^2_x\cap L^p_x)$ holds for weak solutions of \eqref{eq:nonlinear_transport}.
Similarly, in the critical case $\alpha=1/2-1/p$, pathwise uniqueness in $L^\infty_{\omega,T}(L^{2}_x\cap L^p_x)\cap \mathcal{U}_p$ holds for weak solutions of \eqref{eq:nonlinear_transport}, with $\mathcal{U}_p$ being defined by \eqref{eq:defn_Up}.
\end{prop}

\begin{proof}
The argument is inspired by the previous works \cite{CoMa23,GaGrMa24,crippa2025zero}, based on looking at the evolution of appropriate negative Sobolev norms.
Let us first present the argument in the subcritical regime.

Let $\theta^1,\theta^2 \in L^\infty_{\omega,T}(L^r_x\cap L^p_x)$ be two weak solutions on the same tuple $\left(\Omega,\mathcal{F},\mathcal{F}_t,\mathbb{P},W\right)$, and denote $\rho_t:={\theta}^1_t-{\theta}^2_t$ their difference and let $\hat{\rho}_t$ be its Fourier transform.  
Arguing analogously to \cite[Section A.1]{crippa2025zero}, calling $a_t(\xi):=\mathbb{E}\left[|\hat{\rho}_t(\xi)|^2\right]$ and $e_{\xi}:=(2\pi)^{-1} e^{-i\xi\cdot x}$, for any $n\geq 1$ it holds that
\begin{align}\label{Ito_uniqueness}
    \int_{\R^2}  a_t(\xi)\langle \xi\rangle^{2(\alpha-1)}\varphi_n(\xi) \mathd\xi
    &=
    \int_0^t \int_{\R^2}a_s(\xi) F^n_{1-\alpha}(\xi) \mathd\xi \mathd s\notag\notag\\ 
    & -2\int_0^t\int_{\R^2}{\mathbb{E}}\left[\mathfrak{Re} \left(\overline{\hat{\rho}_s(\xi)} \langle u^1_s \cdot \nabla \theta^1_s-u^2_s \cdot \nabla \theta^2_s , e_\xi \rangle \right)\right] \langle \xi\rangle^{2(\alpha-1)}\varphi_n(\xi) \mathd\xi \mathd s,
\end{align}
where $\varphi_n$ are suitable frequency cut-off functions and the ``flux'' functions $F^n_{1-\alpha}$ are defined in \autoref{Appendix:approx_kernels}. 
By dominated convergence, we have 
\begin{align}\label{conve_1_uniq}
     \lim_{n\to\infty} \int_{\R^d} a_t(\xi)\langle \xi\rangle^{2(\alpha-1)}\varphi_n(\xi) \mathd \xi = \mathbb{E}\left[\|\rho_t\|_{H^{\alpha-1}_x}^2\right].
\end{align}
By assumption $\rho\in L^\infty_{\omega,T} L^2_x$, thus $a_s\in L^\infty_{T}L^1_{\xi}$; in view of estimate \eqref{eq:pointwise_bound_approx} from \autoref{lem:flux_functions} (with $s=1-\alpha$), $F^n_{1-\alpha}$ are uniformly bounded. Therefore by dominated convergence and \eqref{eq:good_bound}, we find
\begin{align}\label{conve_2_uniq}
     \lim_{n\to\infty} \int_0^t \int_{\R^2}a_s(\xi) F^n_{1-\alpha}(\xi) \mathd\xi \mathd s
     &= \int_0^t \int_{\R^2}a_s(\xi) F_{1-\alpha}(\xi) \mathd\xi \mathd s\notag\\ 
     &\leq -K_1\int_0^t \mathbb{E}\left[\|\rho_s\|_{L^2_x}^2 \right] \mathd s+K_2\int_0^t\mathbb{E}\left[ \|\rho_s\|_{H^{\alpha-1}_x}^2\right] \mathd s
\end{align}
for some $K_1,K_2>0.$ We are left to estimate the nonlinear term. In the case $\alpha<1/2-1/p$, calling $\gamma:=1-2\alpha-2/p> 0$, by H\"older's inequality and the Sobolev embedding 
\begin{align*}
L^{\frac{2p}{p+2}}_x\hookrightarrow H^{-\frac{2}{p}}_x=H^{2\alpha-1+\gamma}_x,
\end{align*}
it holds 
\begin{align}\label{conve_3_uniq}
\bigg|\int_0^t\int_{\R^2}\mathbb{E}& \left[\mathfrak{Re} \left(\overline{\hat{\rho}_s(\xi)} \langle u^1_s \cdot \nabla \theta^1_s-u^2_s \cdot \nabla \theta^2_s , e_\xi \rangle \right)\right] \langle \xi\rangle^{2(\alpha-1)}\varphi_n(\xi) \mathd\xi \mathd s\bigg|\notag\\
& \leq \int_0^t\int_{\R^2}\mathbb{E}\left[|\hat{\rho}_s(\xi)| |\langle u^1_s \rho_s+(\mathscr{R}_{CZ}\rho_s) \theta^2_s, e_\xi \rangle|\right] |\xi|\langle \xi\rangle^{2(\alpha-1)}\varphi_n(\xi) \mathd\xi \mathd s\notag\\
& \leq \int_0^t \mathbb{E} \int_{\R^2} |\hat{\rho}_s(\xi)|  \big(|\widehat{u^1_s \rho_s}(\xi)|+|\widehat{(\mathscr{R}_{CZ}\rho_s) \theta^2_s}(\xi)| \big)\langle \xi\rangle^{2\alpha-1} \, \mathd\xi \mathd s \notag\\
& \lesssim \int_0^t \mathbb{E}\left[\|\rho_s\|_{H^{-\gamma}_x}^2\right]^{1/2} \mathbb{E}\left[\|u^1_s \rho_s\|_{H^{2\alpha-1+\gamma}_x}^2+\|\mathscr{R}_{CZ}\rho_s \theta^2_s\|_{H^{2\alpha-1+\gamma}_x}^2\right]^{1/2} \mathd s\notag\\
& \lesssim 
\int_0^t \mathbb{E}\left[\|\rho_s\|_{H^{-\gamma}_x}^2\right]^{1/2} 
\mathbb{E}\left[\|u^1_s \rho_s\|_{L^{\frac{2p}{p+2}}_x}^2+\|\mathscr{R}_{CZ}\rho_s \theta^2_s\|_{L^{\frac{2p}{p+2}}_x}^2\right]^{1/2} \mathd s\notag\\
& \lesssim 
\big( \| \theta^1\|_{L^\infty_{\omega,T} L^p_x} + \| \theta^2\|_{L^\infty_{\omega,T} L^p_x}\big)\int_0^t \mathbb{E}\left[\|\rho_s\|_{H^{-\gamma}_x}^2\right]^{1/2}\mathbb{E}\left[\|\rho_s\|_{L^2_x}^2\right]^{1/2} \mathd s,
\end{align}
uniformly in $n$. Letting $n\rightarrow +\infty$ in \eqref{Ito_uniqueness}, combining \eqref{conve_1_uniq}, \eqref{conve_2_uniq}, \eqref{conve_3_uniq} and Young's inequality, we obtain
\begin{align*}
\mathbb{E}\left[\|\rho_t\|_{H^{\alpha-1}_x}^2\right]+\frac{K_1}{2}\int_0^t \mathbb{E}\left[\|\rho_s\|_{L^2_x}^2\right] \mathd s \lesssim \int_0^t \mathbb{E}\left[\|\rho_s\|_{H^{\alpha-1}_x}^2\right] \mathd s.
\end{align*}
The latter implies the claim by Gr\"onwall's inequality.

In the critical case $\alpha=1/2-1/p$, it holds $\gamma=0$, so \eqref{conve_3_uniq} together with Young's inequality does not allow to conclude anymore.
In this case, we follow the argument from \cite{bagnara2025regularization}. 

For $\eps>0$ to be fixed later, consider a decomposition of $\theta^1,\theta^2$ such that
\begin{align*}
    \|\theta^{1,<}\|_{L^{\infty}_{\omega,T}L^p_x}\leq \eps,\quad \|\theta^{2,<}\|_{L^{\infty}_{\omega,T}L^p_x}\leq \eps,\quad  \|\theta^{1,>}\|_{L^{\infty}_{\omega,T}(L^p_x\cap L^{q}_x)}\leq M_{1,\eps},\quad \|\theta^{2,>}\|_{L^{\infty}_{\omega,T}(L^p_x\cap L^{q}_x)}\leq M_{2,\eps}
\end{align*}
for some $q\in (p,+\infty).$ Then arguing as above, it holds
\begin{align}\label{conve_3_uniq_2}
\int_0^t\int_{\R^2}\mathbb{E} & \left[\mathfrak{Re} \left(\overline{\hat{\rho}_s(\xi)} \langle u^1_s \cdot \nabla \theta^1_s-u^2_s \cdot \nabla \theta^2_s , e_\xi \rangle \right)\right] \langle \xi\rangle^{2(\alpha-1)}\varphi_n(\xi) \mathd\xi \mathd s\notag\\
& \leq C_p \int_0^t \mathbb{E}\left[\|\rho_s\|_{L^{2}_x}^2\right]^{1/2} \mathbb{E}\left[\|\mathscr{R}_{CZ}\left(\theta^{1,<}_s+\theta^{1,>}_s\right) \rho_s\|_{L^{\frac{2p}{p+2}}_x}^2
+
\|\mathscr{R}_{CZ}\rho_s \left(\theta^{2,<}_s+\theta^{2,>}_s\right)\|_{L^{\frac{2p}{p+2}}_x}^2\right]^{1/2} \mathd s\notag\\
& \leq \eps C_p\int_0^t \mathbb{E}\left[\|\rho_s\|_{L^{2}_x}^2\right]\mathd s+C_p(M_{1,\varepsilon}+M_{2,\varepsilon})\int_0^t\mathbb{E}\left[\|\rho_s\|_{L^{2}_x}^2\right]^{1/2}\mathbb{E}\left[\|\rho_s\|_{H^{-\gamma_{\eps}}_x}^2\right]^{1/2}\mathd s,
\end{align}
for some $\gamma_{\eps}\in (0,1-\alpha)$ and some constant $C_p>0$ possibly changing line by line, but independent of the choice of $\eps.$ Thus, we can first fix $\eps$ small enough such that $\eps C_p\leq K_1/4$, then \eqref{conve_3_uniq_2} allows to conclude by Young's inequality as in the case $\alpha<1/2-1/p$.
\end{proof}

\begin{rmk}\label{rmk:pathwise_uniqueness}
    In the strictly subcritical regime $\alpha<1/2 - 1/p$, the previous \autoref{thm_pathwise_uniq_order_0} provides pathwise uniqueness in the class $\theta\in L^\infty_{\omega,T}(L^r_x \cap L^p_x)$, without any additional $\tilde L^2_{\omega,T} \tilde B^\beta_{2,\infty}$-regularity of solutions required. In particular, this class is strictly larger than the one to which solutions obtained by vanishing viscosity scheme (\cref{thm:regularization}) belong. 
    A similar remark applies in the critical regime $\alpha = 1/2-1/p$, up to including the space $\mathcal{U}_p$.
\end{rmk}

\begin{rmk}
As shown in \cite[Theorem 1.1]{castro2025unstable}, non-uniqueness of solutions in $L^\infty_{T}(L^r_x \cap L^p_x)$ holds for the deterministic SQG equation with forcing.
The presence of rough transport noise genuinely restores wellposedness for this SPDE.
Similar considerations apply for IPM, where in the absence of noise uniqueness fails for $L^\infty_{T,x}$-solutions, as shown by convex integration techniques \cite{cordoba2011lack}.
\end{rmk}

Our next result is a technical variation of \cite[Theorem 1.1]{BaGa25} that shows pathwise uniqueness in the Euler case $\mathscr{R} = \mathscr{R}_{BS}$. With respect to the aforementioned result, we replace the assumption $\theta_0 \in L^p_x \cap \dot{H}^{-1}_x$ therein with $\theta_0 \in L^r_x \cap L^p_x$, which is more suited for our purposes. Moreover, as in \autoref{thm_pathwise_uniq_order_0}, the result holds under the more general \autoref{ass:noise2} on the noise $W$. 

\begin{prop}\label{thm_pathwise_uniq_euler}
Let $W$ satisfy \autoref{ass:noise2} and let $(\alpha,p,r,\mathscr{R})$ satisfy \autoref{ass:strong_regularization} with $\mathscr{R}=\mathscr{R}_{BS}$.
Then for any $\theta_0\in L^r_x\cap L^p_x$ and deterministic forcing $f\in L^1_T (L^r_x\cap L^{p}_x)$, pathwise uniqueness in $L^\infty_{\omega,T}(L^r_x\cap L^p_x)$ holds for weak solutions of \eqref{eq:nonlinear_transport}.
\end{prop}

\begin{proof}
Let $\theta^1,\theta^2$ be two solutions belonging to $L^\infty_{\omega,T}(L^r_x\cap L^p_x)$ and defined on the same probability space. Thanks to \autoref{prop:convergence_law}, we have weak existence of another weak solution in $L^\infty_{\omega,T}(L^r_x\cap L^p_x)\cap \mathcal{U}_p$; therefore, using the Yamada--Watanabe argument from \cite[Lemma 4.2]{bagnara2025regularization}, without loss of generality we may assume $\theta^2\in L^\infty_{\omega,T}(L^r_x\cap L^p_x)\cap \mathcal{U}_p$ whenever needed in the proof of pathwise uniqueness.

Set $u^i := \mathscr{R}_{BS} \theta^i$, $\rho:=\theta^1-\theta^2$, $u:=u^1-u^2$; since the deterministic forcing $f$ cancels out in the equation for $\rho$, we may ignore it.
The argument is similar to that of \cite[Theorem 1.1]{BaGa25}, based on a Gr\"onwall-type argument for the evolution of $\mathbb{E}[\| \rho_t \|^2_{\dot{H}^{-1}_x}]$, exploiting the specific cancellations from the Euler nonlinearity.
The key point, compared to \cite{BaGa25}, is that even though $\theta^i$ do not need to belong to $L^2_{\omega,T}\dot{H}^{-1}_x$, their difference $\rho$ still does. Let us preliminarily show this fact.

For this purpose, fix an auxiliary parameter $\ell\in(r,2)$ and set
\begin{align*}
q:=\frac{2\ell}{2-\ell}>\frac{2r}{2-r} \geq 2,
\qquad
\gamma:=\frac{2}{q}\in(0,1).
\end{align*}
Recall that $\rho\in L^\infty_{\omega,t}L^2_x$ since $r<2\leq p$ by assumption.
Moreover, by \eqref{eq:properties_BS_integrability} and our choice of the parameter $q$ we have for each $i\in\{1,2\}$
\begin{equation}\label{bound_velocity}
\|u^i\|_{L^\infty_{\omega,t}L^q_x}
\lesssim
\|\theta^i\|_{L^\infty_{\omega,t}(L^r_x\cap L^p_x)}.
\end{equation}
As in
\cite[Section~2.3]{BaGa25}, for $\delta\in(0,1)$, define the following approximation of the Green function $G$ associated to $(-\Delta)^{-1}$:
\begin{align*}
G^\delta(x):=\int_\delta^{1/\delta}p_s(x)\,\dd s,
\qquad
m_\delta(\xi):=\int_\delta^{1/\delta}
e^{-s|\xi|^2}\,\dd s,
\end{align*}
where $p_s$ denotes the heat kernel on $\mathbb R^2$,
so that $m_\delta$ is the Fourier multiplier of convolution
with $G^\delta$. 
Set
\begin{align*}
e^\delta_t:=
\langle G^\delta*\rho_t,\rho_t\rangle,
\qquad
E^\delta_t:=\mathbb E[e^\delta_t].
\end{align*}
Since $0\leq |\xi|^2m_\delta(\xi)\leq1$, Plancherel's identity
gives, for every $z\in L^2_x$,
\begin{equation}\label{eq:estimate_reg}
\|\nabla G^\delta*z\|_{L^2_x}^2
\leq \langle G^\delta*z,z\rangle,
\qquad
\|D^2G^\delta*z\|_{L^2_x}
\leq \|z\|_{L^2_x}.
\end{equation}
Combining these estimates with the
Gagliardo--Nirenberg inequality, we obtain
\begin{align}\label{eq:gagliardo_niremberg}
\|\nabla G^\delta*\rho_s\|_{L^{2q/(q-2)}_x}
\lesssim_q
\|\nabla G^\delta*\rho_s\|_{L^2_x}^{1-\gamma}
\|D^2G^\delta*\rho_s\|_{L^2_x}^{\gamma}\leq
(e^\delta_s)^{(1-\gamma)/2}
\|\rho_s\|_{L^2_x}^{\gamma}.
\end{align}

Arguing as in \cite[Section~4]{BaGa25}, we apply Itô's formula to $e_\delta(t)$ and take expectations.
In this step we observe that $\rho$ has zero initial condition and forcing (as they were the same for $\theta^1$ and $\theta^2$), and
the contribution coming from the term with $C(0) : D^2 \rho$ is non positive;
overall, we get
\begin{align}\label{eq:regularized_ito_euler}
E^\delta_t
&\lesssim
\int_0^t
\mathbb E\bigl[
\langle\nabla G^\delta*\rho_s,J_s\rangle
\bigr]\,\dd s
 +
\sum_{k\in\mathbb N}\int_0^t
\mathbb E\bigl[
\langle
G^\delta*\operatorname{div}(\sigma_k\rho_s),
\operatorname{div}(\sigma_k\rho_s)
\rangle
\bigr]\,\dd s
\end{align}
where $J:=u^1\theta^1-u^2\theta^2$.
Let us control the right-hand side of \eqref{eq:regularized_ito_euler} to apply Gr\"onwall. First, by Hölder's inequality,  \eqref{eq:gagliardo_niremberg}, \eqref{bound_velocity}, and the uniform $L^2_x$ bounds on the
solutions, we have
\begin{align*}
|\langle\nabla G^\delta*\rho_s,J_s\rangle|
&\leq
\|\nabla G^\delta*\rho_s\|_{L^{2q/(q-2)}_x}
\sum_{i=1}^2
\|u^i_s\|_{L^q_x}\|\theta^i_s\|_{L^2_x}\\
&\lesssim
e_\delta(s)^{(1-\gamma)/2}
\|\rho_s\|_{L^2_x}^{\gamma}
\sum_{i=1}^2
\|u^i_s\|_{L^q_x}\|\theta^i_s\|_{L^2_x}
\lesssim 
1+e_\delta(s),
\end{align*}
where all implicit constants are independent of $\delta$.
Secondly, the bound $|\xi|^2m_\delta(\xi)\leq1$ gives
\begin{align*}
\sum_{k\in\mathbb N}
\langle
G^\delta*\operatorname{div}(\sigma_k\rho_s),
\operatorname{div}(\sigma_k\rho_s)
\rangle\leq
\sum_{k\in\mathbb N}\|\sigma_k\rho_s\|_{L^2_x}^2
=
\operatorname{Tr}C(0)\,\|\rho_s\|_{L^2_x}^2
\lesssim 1.
\end{align*}
Combining these estimates with
\eqref{eq:regularized_ito_euler}, we obtain
\begin{align*}
E^\delta_t
\lesssim \int_0^t (1+E^\delta_s)\,\dd s,
\end{align*}
and Grönwall's lemma yields $\sup_{\delta\in(0,1), t\in[0,T]}
E^\delta_t < \infty$.
Since $m_\delta(\xi)\uparrow |\xi|^{-2}$ as $\delta\rightarrow 0$
for every $k\neq0$, monotone convergence then implies the desired $\rho \in L^\infty_T L^2_\omega\dot H^{-1}_x$;
as a consequence $u\in L^\infty_T L^2_\omega  H^1_x$. 

With this additional information on $\rho$, we can now follow the lines of the proof of \cite[Theorem 1.1]{BaGa25} to get the equivalent of their Equation (4.4), and then conclude by Gr\"onwall's lemma; let us only briefly sketch one possible argument.
Notice that, due to the structure of the Euler nonlinearity, we have $\rho_t = \int_0^ t f_s \dd s + \nabla\cdot M_t$, where
\begin{align*}
    f_t:= -\nabla^\perp\cdot [\nabla\cdot (u\otimes u^1 + u^2\otimes u)],\qquad 
    M_t := \int_0^t \rho_s\, \dd W_s.
\end{align*}
Using that $u\in L^2_{\omega,T} H^1_x$ by the above and $u^1,u^2\in L^\infty_{\omega,T} L^q_x$ with $q=2r/(2-r)$ by \eqref{eq:properties_BS_integrability}, it's easy to see that $u\otimes u^1 + u^2\otimes u\in L^2_{\omega,T,x}$, thus $f\in L^2_{\omega,T} \dot H^{-2}_x$. Similarly, $\nabla\cdot M$ is a continuous $\dot H^{-1}_x$-valued martingale by \cite[Lemma 2.2]{DrGaPa25}.
With these facts in mind, as well as $\rho\in L^\infty_{\omega,t} L^2_x$, one can rigorously pass to the limit as $\delta\to 0^+$ in the It\^o formula for $E^\delta$ and take expectation to find
\begin{align*}
    \EE[\| \rho_t\|_{\dot H^{-1}_x}^2] = -2\int_0^t \EE[ \langle u_s, (u_s\cdot\nabla) u^2_s) \rangle] + \mathbb{E}\left[ \int_0^t \langle \mathrm{Tr} [QD^2G] \ast \rho_s ,\rho_s \rangle \, \dd s \right]
\end{align*}
where we exploited the cancellation $\langle u_s, (u^1_s\cdot\nabla) u_s\rangle = 0$.
The rest of the proof proceeds identically to \cite{BaGa25}, up to replacing the application of Corollary 2.5 therein with \autoref{cor:flux_-1}.
\end{proof}

We are finally ready to complete the

\begin{proof}[Proof of \autoref{thm:intro_wellposedness}]
    By \cref{thm_pathwise_uniq_euler} and \cref{thm_pathwise_uniq_order_0}, pathwise uniqueness for \eqref{eq:nonlinear_transport} holds in the class $\mathcal{E}$ for $\mathscr{R}=\mathscr{R}_{BS}$ or strictly subcritical regimes when $\mathscr{R}=\mathscr{R}_{CZ}$, and in $\mathcal{E}\cap \mathcal{U}_p$ in critical regimes for $\mathscr{R}=\mathscr{R}_{CZ}$.
    On the other hand, by \autoref{prop:convergence_law}, any limit points of the laws of $\{\theta^{\nu}\}_{\nu \in (0,1)}$ must be the law of a solution to \eqref{eq:nonlinear_transport} with values in $\mathcal{E}\cap \mathcal{U}_p$.
    Combining these facts with the classical Gyöngy--Krylov criterion \cite[Lemma 1.1]{gyongy1996existence} we can deduce that the whole family $\{\theta^\nu\}_{\nu \in (0,1)}$ converges, without the need to extract a subsequence, and that $\theta^\nu\to \theta$ in $\mathcal{E}$ in probability, without the need to pass to another probability space via Skorokhod's theorem.
    It follows that $\theta$ is adapted to the (augmented) filtration generated by $W$, since it is a limit in probability of adapted processes $\{\theta^\nu\}_{\nu \in (0,1)}$.
    Finally, Markovianity of $\theta$ follows by standard arguments, since we have established global strong existence and pathwise uniqueness for an SPDE with autonomous coefficients, whose driving noise $W$ has independent stationary increments.
\end{proof}

\section{Enstrophy dissipation and sharpness of regularization for $2$D Euler}\label{sec:enstrophy_dissipation}

The first part of this section is devoted to the proof of property \eqref{eq:an_diss_euler} from \autoref{thm_anomalous_dissipation}.
Then in \autoref{subsec:no_anomalous} we provide a general criterion to establish conservation of energy for regular enough solutions to \eqref{eq:nonlinear_transport}; this implies sharpness of anomalous regularization by a contradiction argument, completing the proof of \autoref{thm_anomalous_dissipation}.

\subsection{Girsanov transform and anomalous dissipation}\label{subsec:girsanov}

Throughout this specific section, we will assume the noise $W$ to be \emph{exactly} the Kraichnan model, namely with covariance function $C_\alpha$ as defined in \eqref{eq:isotropic.covariance}.
In this case, the Cameron--Martin space associated to $\dd W$ is given by $\mathcal{H}=L^2_T \mathcal{H}_0$, where $\mathcal{H}_0:=\mathcal{C}_\alpha^{1/2}(L^2_x)$; the latter can be explicitly identified, see for instance in \cite[Lemma 3.1]{bagnara2024anomalous}:
\begin{equation}\label{eq:identification_cameron}
    \mathcal{H}_0=\{f\in H^{1+\alpha}_x:\, {\rm div} f = 0\}, \quad \| f\|_{\mathcal{H}_0}\sim \| f\|_{H^{1+\alpha}_x}.
\end{equation}

We restrict our attention to the stochastic Euler equations, $\mathscr{R}=\mathscr{R}_{BS}$, and assume that condition \eqref{eq:parameters_anomalous_dissipation} holds.
In this case, for any initial condition $\theta_0$, strong existence and pathwise uniqueness of solutions are provided by \cite[Theorem 1.1]{BaGa25}; going through similar estimates as \cite[Proposition 4.1, equation (4.4)]{GalLuo25}, it's easy to see that  $\theta\in L^{2}_\omega C_T \dot H^{-1}_x$. 
By \autoref{thm:regularization}, for any $\delta>0$, we also have $\theta\in L^{2}_{\omega,T}H^{1-\alpha-\delta}_x$; by the properties of the Biot--Savart kernel \eqref{eq:BS_Sobolev_spaces} it then follows that
\begin{align}\label{regularity_u_girsanov}
    u=\mathscr{R}_{BS}\theta\in L^{2}_{\omega,T}H^{2-\alpha-\delta}_x\qquad \text{and }\qquad \operatorname{div}u=0.
\end{align}
Let now $\rho$ be the unique probabilistically strong solution to the linear problem
\begin{equation}\label{eq:linear_kraichnan}
    \begin{cases}
\mathd \rho + \circ\, \mathrm{d} {W} \cdot \nabla \rho =
  0, 
  \\
  \rho |_{t = 0} = \theta_0 \in L^r_x \cap L^p_x\cap \dot H^{-1}_x
  \end{cases}
\end{equation}
given e.g. by \cite[Proposition 1.1]{crippa2025zero}.
Combining this result with \cite[equation (4.4)]{GalLuo25}, one can see that $ \rho\in L^{2}_{\omega} C_T\dot H^{-1}_x\cap L^{2}_{\omega,T}H^{1-\alpha-\delta}_x$, and therefore one again
\begin{align}\label{regularity_girsanov_2}
    \hat{u}:=\mathscr{R}_{BS}\rho\in L^{2}_{\omega,T}H^{2-\alpha-\delta}_x\qquad \text{and }\qquad \operatorname{div}\hat{u}=0.
\end{align}
Overall, since $\alpha\in (0,1/2)$ and $\delta>0$ can be taken arbitrarily small, by \eqref{eq:identification_cameron}, \eqref{regularity_u_girsanov} and \eqref{regularity_girsanov_2} we deduce that
\begin{equation}\label{eq:relative_entropy_requirement}
    u,\hat{u} \in L^2_\omega \mathcal{H}=L^2_{\omega,T}\mathcal{H}_0.
\end{equation}
In what follows, we want to work on a Polish path space. 
To this end, let $\mathcal{B}_{2}$ denote the closed ball of radius $\Bar{M}$ in $L^{2}_x$ endowed with the weak topology;\footnote{We recall that $\Bar{M}$ is the square root of the right-hand side in \eqref{viscous_uniform_bound}.}
consider $C([0,T];\mathcal{B}_{2})=C_T \mathcal{B}_{2}$, seen as a measurable space equipped with its Borel $\sigma$-algebra.
By contruction, $\theta,\rho$ belong $\PP$-a.s. to $C_T \mathcal{B}_{2}$.

The difference between the SPDEs \eqref{eq:linear_kraichnan} and \eqref{eq:nonlinear_transport} is that the first is transported by $\dd W_t$, while the second one by $\dd W_t + u_t \dd t$, where the shift $u=\mathscr{R}_{BS}\theta$ belongs to $\mathcal{H}$ $\PP$-a.s. due to \eqref{eq:relative_entropy_requirement}.
As a consequence, we can apply stopping time arguments and Girsanov transform to deduce relative entropy bounds for the laws of $\theta$ and $\rho$; see the contributions \cite{LiSh01,Lehec2013,Ferrario2012} for several instances of this argument.
We recall that, given two probability measures $\mu_1,\mu_2$ on a measurable space $(A,\mathcal{A})$, the relative entropy of $\mu_1$ with respect to $\mu_2$ is defined by
    \begin{align*}
        H(\mu_1|\mu_2)=\begin{cases}
        \int_A \log(\frac{d\mu_1}{d\mu_2})d\mu_1 \quad &\mbox{if }\mu_1\ll \mu_2,\\
            +\infty\quad &\mbox{otherwise.} 
        \end{cases}
    \end{align*}
In light of \eqref{eq:relative_entropy_requirement}, \cite[Propositions 4.4 and 4.6]{GalLuo25} apply and yield the following result.

\begin{prop}\label{prop:girsanov}
    Let $\mathscr{R}=\mathscr{R}_{BS}$ and assume that \eqref{eq:parameters_anomalous_dissipation} holds; let the covariance of $W$ be given by $C_\alpha$ as in \eqref{eq:isotropic.covariance}.
    Let $\mu^{\theta}:=\mathcal{L}(\theta)$ and $\mu^{\rho}:=\mathcal{L}(\rho)$ be the laws of $\theta$ and $\rho$, solving respectively \eqref{eq:nonlinear_transport} and \eqref{eq:linear_kraichnan}, on $C([0,T];\mathcal{B}_{2})$.
    Then $\mu^{\theta}$ and $\mu^{\rho}$ are equivalent measures. Moreover, their relative entropies satisfy
    \begin{equation}\label{eq:relative_entropy_estimates}\begin{split}
        H(\mu^{\theta}|\mu^{\rho})&\leq \frac{1}{2}\mathbb{E}\left[\int_0^T\|u_s\|_{\mathcal{H}_0}^2 \mathd s\right]<+\infty,\\
        H(\mu^{\rho}|\mu^{\theta})&\leq \frac{1}{2}\mathbb{E}\left[\int_0^T\|\hat{u}_s\|_{\mathcal{H}_0}^2 \mathd s\right]<+\infty.
    \end{split}\end{equation}
\end{prop}

As a consequence, we are ready to present the

\begin{proof}[Proof of \eqref{eq:an_diss_euler} from \autoref{thm_anomalous_dissipation}]
Let $a$ denote generic elements of $C_T \mathcal{B}_2$.
Fix $t\in[0,T]$ and let $\{s_n\}_{n\in\N}\subset [0,T]$ be such that $s_n\rightarrow t$ as $n\rightarrow+\infty$. Define the random variables
\begin{align*}
    X_n(a):=\|a(s_n)-a(t)\|_{L^2_x}^2.
\end{align*}
By \cite[Theorem 1.1]{DrGaPa25}, $X_n\rightarrow 0$ in measure with respect to $\mu^{\rho}$.
Since $\mu^\rho$ and $\mu^\theta$ are equivalent by \autoref{prop:girsanov}, it follows that $X_n\rightarrow 0$ in measure with respect to $\mu^{\theta}$ as well; namely,
\begin{align}\label{convergence_prob_theta}
    \|\theta_{s_n}-\theta_t\|_{L^2_x}^2\rightarrow 0\quad \mbox{in probability.}
\end{align}
On the other hand,
\begin{align*}
    \sup_{n\geq 1}\|\theta_{s_n}-\theta_t\|_{L^2_x}
    \leq 2 \sup_{t\in [0,T]} \|\theta_t\|_{L^2_x}
    \leq 2 \| \theta_0\|_{L^2_x}\quad \mathbb{P}\mbox{-a.s.}
\end{align*}
Therefore, \eqref{convergence_prob_theta} and dominated convergence imply the first claim in \eqref{eq:an_diss_euler}.

We now prove the second claim. 
Fix any $s<t$; by \cite[Proposition 2.6]{DrGaPa25}, $\rho$ is a Markov process, which combined with Remark 2.7 therein implies that
\begin{align*}
    \PP(\| \rho_t\|_{L^2_x} \leq \| \rho_s\|_{L^2_x})=1.
\end{align*}
On the other hand, by \cite[Theorem 1.1]{DrGaPa25}, $\EE[\rho_t\|_{L^2_x}]<\EE[\rho_s\|_{L^2_x}]$, so that necessarily $\PP(\| \rho_t\|_{L^2_x} < \| \rho_s\|_{L^2_x})>0$.
Since $\mu^\theta$ and $\mu^\rho$ are equivalent, this necessarily implies that
\begin{align*}
    \PP(\| \theta_t\|_{L^2_x} \leq \| \theta_s\|_{L^2_x})=1, \quad \PP(\| \theta_t\|_{L^2_x} < \| \theta_s\|_{L^2_x})>0
\end{align*}
and so we conclude that $\EE[\| \theta_t\|_{L^2_x}]<\EE[\| \theta_s\|_{L^2_x}]$.
\end{proof} 

\begin{rmk}
Consider the following family of viscous approximations of \eqref{eq:nonlinear_transport} with varying noise intensity (cf. \cite[p. 859]{le2002integration} and \cite[eq, (2.11)]{DrGaPa25} for some motivations)
\begin{align*}
    \begin{cases}
\mathd \bar\theta^\nu + \bar u^\nu \cdot \nabla \bar\theta^\nu \mathd t + \sqrt{1-\nu}\circ\, \mathrm{d} W \cdot \nabla \bar\theta^\nu =
   \frac{\nu}{2}\Delta \bar\theta^\nu \mathd t, \\
  \bar u^\nu := \mathscr{R}_{BS} \bar\theta^\nu,\\
  \bar\theta^\nu |_{t = 0} = \theta^\nu_0.
\end{cases}
\end{align*}
By readapting the Girsanov-type arguments presented above and exploiting the strong convergence in $L^2(\Omega\times \R^2)$ for the corresponding linear problems (see \cite[Lemma 2.8]{DrGaPa25}), one can easily verify that the following continuous-in-time anomalous dissipation property holds: for every $0\leq s<t\leq T$, 
\begin{equation}\label{eq:anomalous.girsanov}
    \lim_{\nu\rightarrow 0} \nu\int_s^t\mathbb{E}\left[\|\nabla\bar\theta^{\nu}_r\|_{L^2_x}^2\right]\mathd r>0.
\end{equation}
One can moreover redapt the arguments from \cite[Proposition 2.3]{DrGaPa25} to obtain equivalent characterizations of \eqref{eq:anomalous.girsanov}, e.g. at the level of the dissipation measure, and deduce a variant of \eqref{eq:anomalous.girsanov} with $\limsup_{\nu\to 0}$ inside the expectation rather than outside.
\end{rmk}

\subsection{Anomalous dissipation implies irregularity}\label{subsec:no_anomalous}
We are left to show the optimality of the anomalous regularization obtained in \autoref{thm:regularization}. This is a direct consequence of the strict dissipation of enstrophy \eqref{eq:an_diss_euler} via \autoref{prop:dissipation.implies.irregularity} below.
The argument is classical: if solutions enjoyed higher regularity, then by commutator estimates mean enstrophy should be preserved, yielding a contradiction.

Even though \autoref{thm_anomalous_dissipation} only applies to the Euler equations and $\alpha<1/2$, we give the next statement in higher generality.
For $\gamma\in (0,1)$ and $p\in [1,+\infty]$, recall the deterministic function spaces $E^{\gamma,p}_{[s,t]}$ defined in \autoref{subsec:besov}.
The next conditional statement only requires regularity of the covariance associated to the noise $W$, which is much weaker than \autoref{ass:noise} (since \eqref{eq:local.expansion.noise} also encodes some form of non-degeneracy and local isotropy).

\begin{lem}\label{prop:dissipation.implies.irregularity}
Let $\alpha\in (0,1)$ and let $W$ be a noise whose covariance $C$ satisfies \eqref{eq:properties.C} and belongs to $C^{2\alpha}_x$. Let $\theta$ be a weak solution to the SPDE \eqref{eq:nonlinear_transport} with $f\equiv 0$ and $\theta\in L^\infty_{\omega,T} (L^r_x\cap L^p_x)$.
Suppose that $\mathscr{R},\alpha,r,p,\beta$ satisfy \autoref{ass:exponent.beta} with $\beta=1-\alpha$, namely that either of the following holds: 
\begin{itemize}
    \item $\mathscr{R} = \mathscr{R}_{BS}$, $r\in (1,2)$, $p\in [2,\infty)$ and $(\alpha,p)$ are such that $\beta=1-\alpha$ in \eqref{eq:definition.regularity};
    \item $\mathscr{R}=\mathscr{R}_{CZ}$, $r\in (1,2]$, $p\in [2,\infty)$ and $(\alpha,p)$ are such that $\beta=1-\alpha$ in \eqref{eq:definition.regularity_CZ}.
\end{itemize}
Let $0 \leq t_0 < t_1 \leq T$ and further assume that 
\begin{equation}\label{eq:condition_no_anomalous1}
    \PP\left( \theta\in E^{1-\alpha,2}_{[t_0,t_1]}\right)=1.
\end{equation}
Then 
\begin{equation}\label{eq:conclusion_no_anomalous1}
    \PP\left( \|\theta_{t}\|_{L^2_x}=\|\theta_{t_0}\|_{L^2_x} \text{ for every } t\in [t_0,t_1]\right)=1.
\end{equation}
\end{lem}

\begin{proof}
     Let $\chi\in C^\infty_c(\R^2)$ be a symmetric probability density supported in the annulus $B_1\setminus B_{1/2}$, and let $\{\chi^\eps\}_{\eps\in (0,1)}$ denote the associated family of standard mollifiers.

    Arguing as in \cite[Proposition 3.10]{DrGaPa25}, we can apply It\^o's formula to the process
    $\langle\theta^\eps_t,\theta_t\rangle$ and use that $Q$ is divergence-free to obtain that $\PP$-a.s.
    \begin{equation}\label{eq:approx_energy_balance}
        \langle\theta^\eps_{t},\theta_{t}\rangle - \langle\theta^\eps_{t_0},\theta_{t_0}\rangle
        = 2 \int_{t_0}^t \langle u_{s}\cdot\nabla \theta^\eps_{s}, \theta_{s}\rangle \dd {s}
        + 2 M^\eps_{t} + \int_{t_0}^t \mathcal{A}^\eps_{s} \dd {s},\quad\forall t\in [t_0,t_1];
    \end{equation}
    in the above, the martingale $M^\eps$ and the process $\mathcal{A}^\eps$ are given by
    \begin{equation}\label{eq:no_anomalous_terms}
        M^\eps_{t} =\sum_{k\in\N} \int_{t_0}^{t} \langle \sigma_k\cdot\nabla\theta^\eps_{s}, \theta_{s}\rangle \dd W^k_{s}, \quad
        \mathcal{A}^\eps_{s} = -\frac{1}{2\eps^2} \int_{\R^2} Q(\eps z):D^2 \chi(z) \|\delta_{\eps z}\theta_{s}\|_{L^2_x}^2 \dd z,
    \end{equation}
    where we used the series representation \eqref{eq:noise_series_expansion}.
    Set $I^\eps_s:= \langle u_{s}\cdot\nabla \theta^\eps_{s}, \theta_{s}\rangle$; in view of \eqref{eq:approx_energy_balance} and properties of mollifiers, in order to deduce \eqref{eq:conclusion_no_anomalous1} it suffices to show that
    \begin{equation}\label{eq:no_anomalous_goal}
        \lim_{\eps\to 0} \int_{t_0}^{t_1} |I^\eps_s| \mathd s=0, \quad
        \lim_{\eps\to 0} \int_{t_0}^{t_1} |\mathcal{A}^\eps_{s}| \dd {s}=0,\quad
        \lim_{\eps\to 0} \sup_{t\in [t_0,t_1]} |M_t^\eps|=0,
    \end{equation}
    where all limits above must be understood in probability.

    For notational convenience, throughout this proof let us adopt the notation
    \begin{align*}
        \trinorm{\theta}_\eps:= \sup_{0<|z|\leq \eps} \frac{1}{|z|^{1-\alpha}} \|\delta_z \theta\|_{L^2([t_0,t_1]\times \R^d)}
    \end{align*}
    so that assumption \eqref{eq:condition_no_anomalous1} amounts to the $\PP$-a.s. convergence $\trinorm{\theta}_\eps\to 0$ as $\eps\to 0^+$ (see \eqref{eq:characterization_sharp_endpoint}).

    Concerning $\mathcal{A}^\eps$, by \eqref{eq:no_anomalous_terms} and \eqref{eq:local.expansion.noise} we have
    \begin{align*}
        \limsup_{\eps\to 0^+} \int_{t_0}^{t_1} |\mathcal{A}^{\eps}_{s}| \mathd{s}
        \lesssim \limsup_{\eps\to 0^+} \frac{1}{\eps^{2-2\alpha}} \int_{t_0}^{t_1} \int_{B_1}\|\delta_{\eps z}\theta_s\|_{L^2_x}^2 \mathd z \mathd s
        \lesssim \trinorm{\theta}_\eps^2
        \rightarrow 0.
    \end{align*}

    Since $M^\eps$ are continuous martingales with $M^\eps_{t_0}=0$, denoting by $[M^\eps]$ their quadratic variations, in order to check \eqref{eq:no_anomalous_goal} for this term it suffices to show that $[M^\eps]_{t_1}\to 0$ in probability as $\eps\to 0^+$.

    A direct computation yields
    \begin{align*}
        [M^\eps]_{t_1}
        & =\int_{t_0}^{t_1} \sum_{k\in\N} |\langle \sigma_k\cdot\nabla \theta^\eps_{s}, \theta_{s}\rangle|^2 \dd {s}\\
        & = \int_{t_0}^{t_1} \frac{1}{16\eps^2} \sum_{k\in\N} \left|\int_{\R^2\times \R^2} \nabla\chi(z)\cdot \delta_{\eps z} \sigma_k(x) |\delta_{\eps z}\theta_{s}(x)|^2 \dd x\dd z\right|^2 \dd {s}
        =: \frac{1}{16\eps^2}\int_{t_0}^{t_1} J^\eps_{s} \dd {s},
    \end{align*}
    where we applied \autoref{lem:trilinear}. 
    By Minkowski's integral inequality,
    \begin{align*}
        \sum_{k\in\N} \| f_k\|_{L^1(\dd x,\dd z)}^2
        & =\Big\| \| f_k\|_{L^1(\dd x,\dd z)} \Big\|_{\ell^2(k)}^2\\
        & \leq \Big\| \|f_k(x,z)\|_{\ell^2(k)} \Big\|^2_{L^1(\dd x,\dd z)}
        = \Bigg[ \int_{\R^2\times \R^2} \bigg( \sum_{k\in\N} |f_k(x,z)|^2 \bigg)^{1/2} \dd x \dd z \Bigg]^2.
    \end{align*}
    Applying this to $J^\eps_{s}$ yields
    \begin{align*}
        J^\eps_{s}
        & \leq \Bigg[ \int_{\R^2\times \R^2} |\nabla\chi(z)| |\delta_{\eps z}\theta_{s}(x)|^2  \bigg( \sum_{k\in\N} |\delta_{\eps z}\sigma_k(x)|^2 \bigg)^{1/2} \dd x \dd z \Bigg]^2\\
        & = 2 \Bigg[ \int_{\R^2\times \R^2} |\nabla\chi(z)| |\delta_{\eps z}\theta_{s}(x)|^2  \big( {\rm Tr} Q(\eps z)\big)^{1/2} \dd x \dd z \Bigg]^2\\
        & = 2 \Bigg[ \int_{\R^2} |\nabla\chi(z)| \|\delta_{\eps z}\theta_{s}\|_{L^2_x}^2  \big( {\rm Tr} Q(\eps z)\big)^{1/2} \dd z \Bigg]^2\\
        & \lesssim \int_{\R^2} |\nabla\chi(z)| \|\delta_{\eps z}\theta_{s}\|_{L^2_x}^4  \, {\rm Tr} Q(\eps z) \dd z
    \end{align*}
    wehre in the second passage we used the identity \eqref{eq:covariance_series_expansion} and the definition of $Q$, while in the last step we applied Jensen's inequality, which is justified since $\nabla\chi\in L^1_x$.
    Using the assumption \eqref{eq:local.expansion.noise}, we infer that for all $\eps>0$ sufficiently small
    \begin{align*}
        [M^\eps]_{t_1}
        \lesssim \frac{1}{\eps^2} \int_{t_0}^{t_1} \int_{B_1} \|\delta_{\eps z}\theta_{s}\|_{L^2_x}^4  \, {\rm Tr} Q(\eps z) \dd z
        \lesssim \frac{\eps^{2(1-\alpha)} \eps^{2\alpha} }{\eps^2}\| \theta\|_{L^\infty_T L^2_x}^2 \trinorm{\theta}_\eps^2
        \lesssim \trinorm{\theta}_\eps^2\to 0
    \end{align*}
    which concludes the verification of \eqref{eq:no_anomalous_goal} for this term as well.

    It remains to treat $I^\eps$. Note that this is the usual DiPerna--Lions commutator term, which by \autoref{lem:trilinear} satisfies
    \begin{align*}
        \int_{t_0}^{t_1} |I^\eps_s| \mathd s
        & \lesssim \frac{1}{\eps} \int_{t_0}^{t_1} \int_{\R^2\times \R^2}  |\nabla\chi(z)| |\delta_{\eps z} u_s(x)| |\delta_{\eps z}\theta_s(x)|^2 \mathd x \mathd z \mathd s\\
        & \lesssim \frac{1}{\eps} \int_{t_0}^{t_1} \int_{B_1} \int_{\R^2} |\delta_{\eps z} u_s(x)| |\delta_{\eps z}\theta_s(x)|^2 \mathd x \mathd z \mathd s;
    \end{align*}
    the last term above has almost the same structure as $\mathcal{T}_{u^\nu,\theta^\nu}(s,\eps)$ previously estimated in \autoref{sec:computation}.
    In particular, obtaining uniform estimates for this term is very similar to verifying \eqref{eq:key_bound_regularization} with $\beta=1-\alpha$ and/or verifying \eqref{eq:anomalous_integrability_goal}. A couple observations are in order:
    \begin{itemize}
        \item[i)] Compared to the setting of the previous sections, we do not need to enforce any decomposition for either $u$ or $\theta$ in order to handle critical parameters: as soon as we obtain an estimate which contains a positive power of $\trinorm{\theta}_\eps$, this term $\PP$-a.s. converges to $0$ by assumption.
        \item[ii)] Applying \autoref{lem:final_besov} to deterministic functions, for any $\gamma\in (0,1)$, we similarly obtain the estimate
        \begin{align*}
            \int_{t_0}^{t_1} \| \delta_{\eps z} \theta_s\|_{H^{\gamma(1-\alpha)}_x}^2 \mathd s \lesssim \eps^{2(1-\gamma)(1-\alpha)} \trinorm{\theta}_\eps^2, \quad\forall\eps\in (0,1).
        \end{align*}
        \item[iii)] Many estimates performed in \autoref{sec:computation} are purely analytic, based on Sobolev embeddings and H\"older inequalities, and only use expectation in the end to make relevant norms like $\seminorm{\theta}_{\beta,\varrho_0}$ appear.
    \end{itemize}
    With these facts in mind, the verification that $\int_{t_0}^{t_1} |I^\eps_s| \mathd s\to 0$ in probability now follows in the associated ranges of parameters by going through the same computations as in Cases 1, 4 and 5 from \autoref{sec:computation}.
    The only partial exception is the case $\mathscr{R}=\mathscr{R}_{BS}$ and $p\geq 3$; here we can use the more direct estimate
    \begin{align*}
        \int_{t_0}^{t_2} |I^\eps_s| \mathd s
        & \leq \frac{1}{\eps} \int_{B_1} \| \delta_{\eps z} u\|_{L^3_{[t_0,t_1]} L^3_x} \| \delta_{\eps z} \theta\|_{L^3_{[t_0,t_1]} L^3_x}^2 \mathd z\\
        & \lesssim \| \nabla u\|_{L^3_{[t_0,t_1]} L^3_x} \int_{B_1} \| \delta_{\eps z} \theta\|_{L^3_{[t_0,t_1]} L^3_x}^2 \mathd z
        \lesssim \sup_{|z|\leq 1} \| \delta_{\eps z} \theta\|_{L^3_{[t_0,t_1]} L^3_x}^2
    \end{align*}
    where the last quantity goes to $0$ $\PP$-a.s. as $\eps\to 0^+$ due to continuity of translations in $L^3_T L^3_x$.    
\end{proof}

\begin{rmk}
    Going through similar computations, one can obtain the following variant of \autoref{prop:dissipation.implies.irregularity}: if condition \eqref{eq:condition_no_anomalous1} is replaced by
    \begin{align*}
        \lim_{|h|\to 0} \frac{1}{|h|^{2(1-\alpha)}} \int_{t_0}^{t_1} \EE[\| \delta_h \theta_s\|_{L^2_x}^2] \dd s=0,
    \end{align*}
    then instead of \eqref{eq:conclusion_no_anomalous1} one can conclude that $\EE[\|\theta_{t}\|_{L^2}^2]=\EE[\|\theta_{t_0}\|_{L^2}^2]$ for every $t\in [t_0,t_1]$.
\end{rmk}

\begin{proof}[Proof of \autoref{thm_anomalous_dissipation}]
    The proof of property \eqref{eq:an_diss_euler} was presented in \autoref{subsec:girsanov}.
    Concerning the last claim, if $\theta$ belonged to $L^2_\omega L^2_{[t_0,t_1]} B^{1-\alpha}_{2,q}$, then in particular by the facts recalled in \autoref{subsec:besov} one would have
    \begin{align*}
        \PP\big(\theta \in E^{1-\alpha,2}_{[t_0,t_1]}\big)
        \geq \PP\big(\theta \in L^2_{[t_0,t_1]} B^{1-\alpha}_{2,q}\big)=1
    \end{align*}
    which combined with \autoref{prop:dissipation.implies.irregularity} would contradict \eqref{eq:an_diss_euler}.
\end{proof}

\appendix
\section{Technical lemmas}\label{Appendix:approx_kernels}

\subsection{Flux functions}

As noticed in \cite{GaGrMa24}, one way to deduce anomalous regularization estimates is to keep track of the evolution of negative Sobolev norms $\dot H^{-s}_x$ under the linear SPDE, which requires to study the derive coercive bounds on the associated flux functions $\dot{F}_s$. Once these are available, they can be used to infer pathwise uniqueness for nonlinear equations. Here we combine the estimates from \cite{GaGrMa24} with the argument from \cite{crippa2025zero} to extend similar bounds to the fluxes $F_s$ associated to inhomogeneous Sobolev norms $H^{-s}_x$. This result is crucially applied in \autoref{thm_pathwise_uniq_order_0} in the special case $d=2$, $s=1-\alpha$; for the sake of broader applications, we state it in higher generality.

Let $\varphi\in C^\infty_c(\R_+;\R_+)$ be a decreasing function such that $\varphi\equiv 1$ on $[0,1]$, $\varphi\equiv 0$ on $[2,+\infty)$ and set $\varphi_n(r):=\varphi(r/n)$.

\begin{lem}\label{lem:flux_functions}
    Let $\alpha\in (0,1)$ and let $C$ satisfy \autoref{ass:noise2} with constants $c_1'$, $c_2'$. Let $s,\alpha\in (0,1)$ such that $s+\alpha\leq 1$ and consider
\begin{align*}
    & \dot F^n_s(\xi):= (2\pi)^{-d/2} \int_{\R^d}  \xi \cdot \hat{C}(\xi-\eta)\xi\, \Big( \frac{\varphi_n(|\eta|)}{|\eta|^{2s}}-\frac{\varphi_n(|\xi|)}{|\xi|^{2s}}\Big) \mathd \eta,\\
    & F^n_s(\xi):= (2\pi)^{-d/2} \int_{\R^d} \xi \cdot \hat{C}(\xi-\eta)\xi\, \Big( \frac{\varphi_n(|\eta|)}{\langle\eta\rangle^{2s}}-\frac{\varphi_n(|\xi|)}{\langle\xi\rangle^{2s}}\Big) \mathd \eta.
    \end{align*}
    Define similarly $F_s(\xi)$, $\dot F_s(\xi)$ by replacing $\varphi_n$ with the constant $1$ in the previous lines.
    Then there exists a constant $K=K(d,s,\alpha,c_1',c_2')$ such that
    \begin{equation}\label{eq:pointwise_bound_approx}
        |F_s(\xi)| + |\dot F_s(\xi)| + \sup_n |\dot F^n_s(\xi)| +  \sup_n | F^n_s(\xi)|\leq K \langle \xi\rangle^{2(1-\alpha-s)} \quad \forall \xi\in\R^d.
    \end{equation}
    Moreover, there exist constants $K_1,K_2>0$, depending on the same parameters as $K$, such that
    \begin{equation}\label{eq:good_bound}
        F_s(\xi) \leq -K_1 \langle \xi\rangle^{2(1-\alpha-s)} + K_2 \langle \xi\rangle^{-2s}\quad \forall \xi\in\R^d.
    \end{equation}
\end{lem}

\begin{proof}
    By \autoref{ass:noise2}, one has $\hat C=c_1' \hat C_\alpha + \hat R$ with $\hat C_\alpha$ given by \eqref{eq:isotropic.covariance}. We divide the proof in two steps.

    \textit{Step 1.} First assume $c_2'=0$, namely $R\equiv 0$. By homogeneity we can further assume $c_1'=0$ and $C=C_\alpha$; $\dot F_s^n$ is then given by
    \begin{equation*}
        \dot F^n_s(\xi):= (2\pi)^{-d/2} \int_{\R^d} \frac{1}{\langle \xi-\eta\rangle^{d+2\alpha}} |P^\perp_{\xi-\eta}\xi|^2 \Big( \frac{\varphi_n(|\eta|)}{|\eta|^{2s}}-\frac{\varphi_n(|\xi|)}{|\xi|^{2s}}\Big) \mathd \eta,
    \end{equation*}
    similarly for $F_s^n$.

    For $|\xi|\leq 1$, estimate \eqref{eq:pointwise_bound_approx} follows from the bound $|P^\perp_{\xi-\eta}\xi|^2\leq |\xi|^2 \leq 1$, that fact that $|\eta|^{-2s}, \langle \eta\rangle^{-2s}\in L^1_x+L^\infty_x$ and Young's convolution inequalities, see the proof of \cite[Proposition 3.2]{GaGrMa24}.

    Assume now $|\xi|\geq 1$. Following \cite[Proposition 3.2]{GaGrMa24}, for given $\xi$, one can define
    \begin{equation}\label{eq:domains}\begin{split}
         & D_1:=\{\eta\in \R^d: |\eta-\xi|\wedge |\eta|>|\xi|/2\},\\
         & D_2:=\{\eta\in \R^d: |\eta-\xi|> |\xi|/2, \ |\eta|\leq |\xi|/2\},\\
         & D_3:=\{\eta\in \R^d: |\eta-\xi|\leq |\xi|/2\}.
    \end{split}\end{equation}
    Accordingly we can write $\dot F^n_s=\sum_{i=1}^3 \dot F^n_{s,i}$, where $\dot F^n_{s,i}$ is defined as $\dot F^n_s$ but with the integral taken over $D_i$.   
    The term $\dot F^n_3$ can be estimated identically as in \cite{GaGrMa24} and is thus omitted (see also Step 2 below).
    Noting that $|\eta-\xi|\sim |\eta|$ on $D_1$, we can estimate $\dot F^n_{s,1}$ by
    \begin{align*}
        |\dot F^n_{s,1}(\xi)|
        \lesssim |\xi|^2 \int_{|\eta|>|\xi|/2}  \frac{1}{|\eta|^{d+2\alpha+2s}} \mathd \eta + |\xi|^{2-2s} \int_{|\eta|>|\xi|/2}  \frac{1}{|\eta|^{d+2\alpha}} \mathd \eta
        \lesssim |\xi|^{2-2\alpha-2s}.
    \end{align*}
    On $D_2$ it holds $|\xi|\sim |\eta-\xi|$, therefore
    \begin{align*}
        |\dot F^n_{s,2}(\xi)|
        & \leq \int_{D_2} \frac{1}{\langle \xi-\eta\rangle^{d+2\alpha}} \Big( \frac{
        |\xi|^2}{|\eta|^{2s}}+|\xi|^{2-2s}\Big) \mathd \eta\\
        & \lesssim |\xi|^{2-2\alpha-d} \int_{|\eta|\leq |\xi|/2 } \frac{
        1}{|\eta|^{2s}} \mathd \eta + |\xi|^{2-2s} \int_{|\xi-\eta|\geq |\xi|/2 } \frac{1}{\langle \xi-\eta\rangle^{d+2\alpha}} \mathd \eta
        \lesssim |\xi|^{2-2\alpha-2s}.
    \end{align*}
    Combining the above estimates gives the $n$-uniform bound for $|\dot F^n_s(\xi)|$.

    Concerning $|F^n_s(\xi)|$, whenever $|\xi|\geq R=R(\alpha,s)\gg 1$, by \cite[Lemma A.2]{crippa2025zero} it holds
    \begin{equation}\label{eq:estimate_CrLuPa}
        |\dot F_s(\xi)-F_s(\xi)| + \sup_{n} |\dot F^n_s(\xi)-F^n_s(\xi)| \leq C' \langle \xi\rangle^{-2s (\frac{d+2\alpha}{d+2})},
    \end{equation}
    which combined with the bound for $\sup_n |\dot F^n_s(\xi)|$ allows to conclude. Instead for $|\xi|\leq R$, we have the trivial estimate
    \begin{equation*}
        |F^n_s(\xi)| \lesssim |\xi|^2 \int_{\R^d} \frac{1}{\langle \xi-\eta\rangle^{d+2\alpha}} \mathd \eta \lesssim |\xi|^2 \leq R^{2(\alpha+s)} |\xi|^{2(1-\alpha-s)}.
    \end{equation*}
    Having established uniform-in-$n$ bounds for $\dot F^n_s$, $F^n_s$, the pointwise estimates for $\dot F_s$, $F_s$ follow by taking $n\to\infty$, concluding the proof of \eqref{eq:pointwise_bound_approx}.

    It remains to show \eqref{eq:good_bound}.
    By \cite[Proposition 4.2]{GaGrMa24}, there exist constants $\tilde K_1$, $\tilde K_2$ such that
    \begin{equation}\label{eq:estimate_GaGrMa}
        \dot F_s(\xi) \leq - \tilde K_1 |\xi|^{2(1-\alpha-s)}+\tilde K_2 |\xi|^{-2s} \quad \forall \xi\in \R^d\setminus\{0\}.
    \end{equation}
    Together with \eqref{eq:estimate_CrLuPa} and Young's inequality, for $|\xi|\geq R$ this implies the existence of constants $C''$, $\tilde K_3>0$ such that
    \begin{align*}
        F_s(\xi)
        &\leq - \tilde K_1 |\xi|^{2(1-\alpha-s)} + \tilde K_2 |\xi|^{-2s} +  C' \langle \xi\rangle^{-2s (\frac{d+2\alpha}{d+2})}\\
        & \leq -\frac{\tilde K_1}{2} |\xi|^{2(1-\alpha-s)} + (\tilde K_2+C'') |\xi|^{-2s}
        \leq -\frac{\tilde K_1}{4} \langle \xi\rangle^{2(1-\alpha-s)} + \tilde K_3 \langle \xi\rangle^{-2s}
    \end{align*}
    where we used that $|\xi|\geq R\gg 1$ and $\frac{d+2\alpha}{d+2}<1$.
    On the other hand, for $|\xi|\leq R$, by virtue of \eqref{eq:pointwise_bound_approx} it holds
    \begin{align*}
        |F_s(\xi)|\leq C\langle \xi\rangle^{2(1-\alpha-s)} \pm \frac{\tilde K_1}{4} \langle \xi\rangle^{2(1-\alpha-s)} \leq - \frac{\tilde K_1}{4} \langle \xi\rangle^{2(1-\alpha-s)} + \Big(C+\frac{\tilde K_1} {4}\Big) \langle R\rangle^{2(1-\alpha)} \langle \xi\rangle ^{-2s}.
    \end{align*}
    Up to relabelling constants, this yields \eqref{eq:good_bound}.
    
    \textit{Step 2.} Consider now $c_2'\neq 0$. To conclude, it suffices to show that the contributions associated to $\hat R$ can be absorbed in the r.h.s. of \eqref{eq:pointwise_bound_approx}. Namely, letting $H_s$, $\dot{H}_s$, $\dot H^n_s$, $H^n_s$ denote the analogues of $F_s$, $\dot F_s$, $\dot F^n_s$, $F^n_s$ with $\hat C$ replaced by $\hat R$, we want to show that
\begin{equation}\label{eq:pointwise_bound_step2}
        |H_s(\xi)| + |\dot H_s(\xi)| + \sup_n |\dot H^n_s(\xi)| +  \sup_n | H^n_s(\xi)|\leq c_2' \langle \xi\rangle^{-2s} \quad \forall \xi\in\R^d.
    \end{equation}
    The proof of \eqref{eq:pointwise_bound_step2} holds for any $s\in (0,1]$, independently of $\alpha$.
    We show \eqref{eq:pointwise_bound_step2} only for $\dot H^n_s$ and $|\xi|\geq 1$; the other cases are similar and left to the reader.
    Let $D_i$ be defined as in \eqref{eq:domains} and consider $\dot H^n_{s,i}$ as before. On $D_1$, since $|\eta-\xi|\sim |\eta| > |\xi|/2$, we have
    \begin{align*}
        |\dot H^n_{s,1}(\xi)|
        \lesssim \int_{\R^d} |\xi|^2 |\hat R(\xi-\eta)| \left(\frac{1}{|\eta|^{2s}}+\frac{1}{|\xi|^{2s}}\right) \dd\eta
        \lesssim \frac{1}{|\xi|^{2s}} \int_{\R^d} |\xi-\eta|^2 |\hat R(\xi-\eta)| \dd \eta
        \lesssim c_2' \langle \xi\rangle^{-2s}.
    \end{align*}
    On $D_2$, thanks to $|\eta-\xi|\sim |\xi|\geq 2|\eta|$ and the cancellation property \eqref{eq:cancellation_fourier_covariance}, it holds
    \begin{align*}
        |\dot H^n_{s,2}(\xi)|
        & \lesssim \int_{\R^d} |\eta|^{2-2s} |\hat R(\xi-\eta)| \dd \eta
        \lesssim \int_{\R^d} |\xi|^{2-2s} |\hat R(\xi-\eta)| \dd \eta\\
        & \lesssim \langle \xi\rangle^{-2s} \int_{\R^d} |\xi-\eta|^2 |\hat R(\xi-\eta)| \dd \eta
        \leq c_2' \langle \xi\rangle^{-2s} 
    \end{align*}
    where we used that $s\leq 1$.
    Finally, on $D_3$ we need to go through the same argument as \cite[Proposition 3.2, Step 2]{GaGrMa24}. Letting $\psi_n$ be defined as therein, using the fact that
    \begin{align*}
        (\eta-\xi)\mapsto [\xi\cdot \hat R(\xi-\eta) \xi]\,\nabla\psi_n(\xi)\cdot(\eta-\xi)
    \end{align*}
    is an odd function (for fixed $\xi$) due to the evenness of $\hat R$, one ends up finding once again the estimate
    \begin{align*}
        |\dot H^n_{s,3}|
        \lesssim \int_{\R^d} |\xi\cdot \hat R(\xi-\eta)\xi| \, |\xi-\eta|^2 \, \frac{1}{|\xi|^{2+2s}} \dd \eta
        \lesssim c_2'\langle \xi\rangle^{-2s}.
    \end{align*}
    Overall this concludes the verification of \eqref{eq:pointwise_bound_step2} and the proof.
\end{proof}

As mentioned above, estimate \eqref{eq:pointwise_bound_step2} holds for any $s\in (0,1]$. Combined with \cite[equation (2.6) and Corollary 2.5]{BaGa25}, this readily yields the following variant of \autoref{lem:flux_functions}.

\begin{cor}\label{cor:flux_-1}
    Let $d=2$, $\alpha\in (0,1)$ and let $C$ satisfy \autoref{ass:noise2} with constants $c_1'$, $c_2'$. Set $Q(z)=C(0)-C(z)$ and let $\dot F_1$ be defined as in \autoref{lem:flux_functions} for $s=1$.
    Let $G$ denote the Green function associated to $(-\Delta)^{-1}$.
    Then there exist constants $K_1,K_2>0$, depending on the parameters $\alpha,c_1',c_2'$ such that
    \begin{equation}\label{eq:good_bound_-1}
        \mathrm{Tr} [\widehat{Q\,D^2 G}](\xi)  = \dot F_1(\xi) \leq -K_1 | \xi|^{-2\alpha} + K_2 |\xi|^{-2}, \quad \forall \xi\in\R^d.
    \end{equation}
    Moreover one has $|F_1(\xi)|\lesssim 1$ uniformly in $\xi\in\R^d$.
\end{cor}

\subsection{Besov-type spaces}\label{app:besov}
Given $f\in L^2_{\omega, T,x}$, $\beta \in (0,1)$ and $\varrho_0>0$, we defined in \eqref{eq:relevant_seminorms} in \autoref{sec:preliminaries} the quantities $g_\beta (\varrho,f)$, $\seminorm{\theta}_{\beta,\varrho_0}^2$ in order to estimate regularity in Besov-type spaces $\tilde L^2_{\omega,T} \tilde B^\beta_{2,\infty}$.
In particular, by \cite[Lemma 2.18]{DrGaPa25}, it holds $\seminorm{f}_{\beta,+\infty}\sim\llbracket f\rrbracket_{\tilde L^2_{\omega,T} \tilde B^\beta_{2,\infty}}$.

We study here more in details useful estimates associated to such quantities; see \autoref{lem:final_besov} for one of the most important examples. For the sake of generality, we allow here any dimension $d\geq 2$.

It turns out convenient, for any $T_0 \in [0,T]$, we consider their analogues conditionally on $\mathcal{F}_{T_0}$:
\begin{equation}\label{eq:conditional_seminorms_appendix}\begin{split}
    \mathbb{E}_{T_0} &:= \mathbb{E} [\, \cdot \mid \mathcal{F}_{T_0}],
    \\
    G_{\beta,T_0} (\varrho,f) 
    &:= 
    \mathbb{E}_{T_0} \left[ \int_{[T_0, T] \times \mathbb{R}^d}
   \int_{\mathbb{S}^{d-1}} \frac{| \delta_{\varrho z} f_s (y) |^2}{\varrho^{2\beta}} \sigma (\dd z) \mathd y \mathd s \right],
   \\
   \seminorm{\theta }_{\beta,\varrho_0 \mid \mathcal{F}_{T_0}}^2 &:=
   \sup_{\varrho \in (0,\varrho_0)}G_{\beta,T_0}(\varrho,f).
\end{split}\end{equation}
When $T_0=0$, if $\mathcal{F}_{T_0}$ is (the completion of) the trivial $\sigma$-algebra, then we recover $G_{\beta,0}(\varrho,f)=g_\beta(\varrho,f)$ and $\seminorm{f }_{\beta,\varrho_0 \mid \mathcal{F}_{0}} = \seminorm{f }_{\beta,\varrho_0}$.

Notice that, for any $f\in L^2_{\omega,T,x}$,  $\seminorm{f }_{\beta,\varrho_0 \mid \mathcal{F}_{T_0}}$ is a well-defined $\mathcal{F}_{T_0}$-measurable random variable, despite the supremum being taken over uncountably many $\varrho \in (0,\varrho_0)$. Indeed, since $f\in L^2_{\omega,T,x}$, by strong continuity of translations in $L^2_x$ and dominated convergence the map 
\begin{align*}
 \varrho \mapsto \int_{[T_0, T] \times \mathbb{R}^d}
   \int_{\mathbb{S}^{d-1}} \frac{| \delta_{\varrho z} f_s (y) |^2}{\varrho^{2\beta}} \sigma (\dd z) \mathd y \mathd s   
\end{align*}
is $\PP$-a.s. continuous at every $\varrho \in (0,+\infty)$; moreover, for every subinterval of the form $[1/n,\varrho_0-1/n] \subset (0,\varrho_0)$ with $n \in \N$, the integrand above is a Bochner integrable random variable taking values in the separable Banach space $C([1/n,\varrho_0-1/n])$. As such, the conditional expectation $G_{\beta,T_0} (\cdot,f)$ is a well-defined $C([1/n,\varrho_0-1/n])$-valued $\mathcal{F}_{T_0}$-measurable random variable, and its supremum on $\varrho\in (0,\varrho_0)$ coincides $\PP$-a.s. with the one taken over countably many values of $\varrho$.

We adopt the short-hand notation $\|f \|_{L^2_{[T_0,T],x}}^2 :=\| f\|_{L^2([T_0,T];L^2_x)}^2$.

\begin{lem} \label{lem:auxiliary.besov.1}
    For any $\beta \in (0,1)$, $\varrho_0\in (0,\infty)$, $T_0 \in [0,T]$ and for every $h \in \R^d$ with $|h| \in (0,\varrho_0)$ it holds $\PP$-a.s. 
    \begin{equation}\label{eq:restricted_increments}
        \mathbb{E}_{T_0} \left[ \| \delta_h f\|_{L^2_{[T_0,T],x}}^2 \right] \lesssim_d |h|^{2\beta} \seminorm{f  }_{\beta,\varrho_0 \mid \mathcal{F}_{T_0}}^2,
    \end{equation}
    where the hidden constant does not depend on $h,\varrho_0,T_0,T$.
\end{lem}

\begin{proof}
    We divide the proof in several passages.
    
    \emph{Step 1.} Let $\psi_{B_{\varrho}}:=|B_{\varrho}|^{-1} \bm{1}_{B_{\varrho}}$ denote the characteristic function of the closed unit ball of radius ${\varrho}>0$, renormalized to be a probability density. We claim that, for any ${\varrho} < {\varrho}_0$, one has the a.s. bound
    \begin{equation} \label{eq:localized_ball_estim}
      \mathbb{E}_{T_0}  \left[ \| f - f\ast\psi_{B_{\varrho}}\|^2_{L^2_{[T_0,T],x}} \right] 
        \leq 
        \mathbb{E}_{T_0} \left[
        \left(
        \fint_{B_{\varrho}} \| \delta_h f\|_{L^2_{[T_0,T],x}} \mathd h
        \right)^2 \right]
        \lesssim 
        \varrho^{2\beta} \seminorm{f }_{\beta,{\varrho}_0 \mid \mathcal{F}_{T_0}}^2.
    \end{equation}
    The first inequality is an immediate consequence of definitions and Minkowski's integral inequality.
    Concerning the second one, by Jensen's inequality and Cavalieri's principle, it holds $\PP$-a.s.
    \begin{align*}
        \mathbb{E}_{T_0} \left[ \left( \fint_{B_{\varrho}} \| \delta_h f\|_{L^2_{[T_0,T],x}} \mathd h \right)^2  \right]
        & \lesssim 
        \mathbb{E}_{T_0} \left[ 
        \frac{1}{{\varrho}^d} \int_{|h|\leq {\varrho}} \| \delta_h f\|_{L^2_{[T_0,T],x}}^2 \mathd h  \right]
        \\
        & \lesssim 
        \frac{1}{{\varrho}^d} \int_0^{\varrho} r^{d-1+2\beta}
        \mathbb{E}_{T_0} \left[ r^{-2\beta}
        \int_{\mathbb{S}^{d-1}}  \| \delta_{r z} f\|_{L^2_{[T_0,T],x}}^2  \sigma(\dd z) \right]
        \dd r
        \\
        & \lesssim
        \varrho^{2\beta}  \seminorm{f  }_{\beta,{\varrho}_0 \mid \mathcal{F}_{T_0}}^2 .
    \end{align*}
    
    \emph{Step 2.} 
    In what follows we will frequently denote translation operators by $\tau_h$, namely $\tau_h f:= f(h+\cdot)$.
    Here we prove the bound \eqref{eq:restricted_increments} for $h\in \R^d$ with $|h|={\varrho} < {\varrho}_0/3$.
    Rewrite
    \begin{align*}
        \delta_h f
        =
        \left( \tau_h f - \tau_h f\ast \psi_{B_{\varrho}}\right) 
        + 
        \left( \tau_h f\ast \psi_{B_{\varrho}} - f\ast \psi_{B_{\varrho}} \right) 
        + 
        \left( f\ast \psi_{B_{\varrho}} - f \right). 
    \end{align*}
    Observe that the translation $\tau_h$ and the convolution with $\psi_{B_{\varrho}}$ commute, so no ambiguity arises in the quantity $\tau_h f\ast \psi_{B_{\varrho}}$ above.
    By triangular inequality and applying \eqref{eq:localized_ball_estim} twice we get
    \begin{align*}
        \mathbb{E}_{T_0} \left[ \| \delta_h f\|_{L^2_{[T_0,T],x}}^2  \right]
        & \lesssim
        \mathbb{E}_{T_0} \left[ \| \tau_h f - \tau_h f\ast \psi_{B_{\varrho}}\|_{L^2_{[T_0,T],x}}^2 \right]
        \\
        &\quad+ 
        \mathbb{E}_{T_0} \left[ \| \tau_h f\ast \psi_{B_{\varrho}} - f\ast \psi_{B_{\varrho}}\|_{L^2_{[T_0,T],x}}^2 \right] 
        \\
        &\quad+ 
        \mathbb{E}_{T_0} \left[ \| f\ast \psi_{B_{\varrho}} - f \|_{L^2_{[T_0,T],x}}^2 \right] 
        \\
        & \lesssim |h|^{2\beta} \seminorm{f }_{\beta,{\varrho}_0 \mid \mathcal{F}_{T_0}}^2 
        + 
        \mathbb{E}_{T_0} \left[ \| \tau_h f\ast \psi_{B_{\varrho}} - f\ast \psi_{B_{\varrho}}\|_{L^2_{[T_0,T],x}}^2  \right].
    \end{align*}
    The first term in the right-hand side above is the same appearing in \eqref{eq:restricted_increments}, thus we only need to bound the second term.
    On the other hand, by Minkowski's integral inequality, we have $\PP$-almost surely
    \begin{align*}
      \mathbb{E}_{T_0} &\left[ \| \tau_h f\ast \psi_{B_{\varrho}} - f\ast \psi_{B_{\varrho}}\|_{L^2_{[T_0,T],x}}^2   \right]
        \\
        &\leq 
         \mathbb{E}_{T_0} \left[ \left(
         \fint_{B_{\varrho}} \fint_{B_{\varrho}} \| f(h+z+\cdot)-f(\tilde z+\cdot)\|_{L^2_{[T_0,T],x}} \mathd z \mathd \tilde z \right)^2   \right]
         \\
         & \lesssim 
         \mathbb{E}_{T_0} \left[ \left(
         \fint_{B_{\varrho}} \fint_{B_{3{\varrho}}} \| f(y+\tilde z +\cdot)-f(\tilde z+\cdot)\|_{L^2_{[T_0,T],x}} \mathd y \mathd \tilde z
         \right)^2  \right]
         \\
         & = 
         \mathbb{E}_{T_0} \left[ \left(
         \fint_{B_{3{\varrho}}} \| \delta_y f\|_{L^2_{[T_0,T],x}} \mathd y
         \right)^2  \right]
         \lesssim 
        \varrho^{2\beta} \seminorm{f }_{\beta,{\varrho}_0 \mid \mathcal{F}_{T_0}}^2
        =
        |h|^{2\beta} \seminorm{f }_{\beta,{\varrho}_0 \mid \mathcal{F}_{T_0}}^2,
    \end{align*}
    where we used that $y:= z+h-\tilde z$ satisfies $|y|\leq 3{\varrho} < {\varrho}_0$ in the intermediate passage, while in the last line we invoked \eqref{eq:localized_ball_estim}.

    \emph{Step 3.} For $|h|={\varrho}\in [{\varrho}_0/3,{\varrho}_0)$, we can just write
    \begin{align*}
        \mathbb{E}_{T_0} \left[ \| \delta_h f\|_{L^2_{[T_0,T],x}}^2  \right]
        &= 
        \mathbb{E}_{T_0} \left[
        \|\tau_{2h/3} (\delta_{h/3} f) + \tau_{h/3} (\delta_{h/3} f)  + \delta_{h/3}f\|_{L^2_{[T_0,T],x}}^2  \right]
        \\
        &\lesssim 
        \mathbb{E}_{T_0} \left[ \| \delta_{h/3} f\|_{L^2_{[T_0,T],x}}^2  \right]
        \lesssim 
        |h|^{2\beta} \seminorm{f }_{\beta,{\varrho}_0 \mid \mathcal{F}_{T_0}}^2,
    \end{align*}
    where the last estimate follows from the previous step, concluding the verification of \eqref{eq:restricted_increments}.    
\end{proof}

The next statement must be understood as an endpoint interpolation estimate: even though $\seminorm{f }_{\beta,+\infty}$ is slightly worse than $\| f\|_{L^2_T B^\beta_{2,\infty}}$, it provides similar estimates on intermediate norms $\| f\|_{L^2_T H^{\gamma \beta}_x}$.

\begin{lem}\label{lem:final_besov}
    Fix $\beta,\varrho_0,T_0$ as in \cref{lem:auxiliary.besov.1}.
    For any $\gamma\in [0,1)$ and any $h\in \R^d$ with $|h| < {\varrho}_0 \leq 1$, one has the $\PP$-a.s. bounds
    \begin{align} 
        & \int_{T_0}^T \EE_{T_0} \left[\| \delta_h f_t\|_{H^{\beta\gamma}_x}^2  \right] \mathd t \lesssim
        |h|^{2(1-\gamma)\beta} \seminorm{ f }_{\beta,{\varrho}_0 \mid \mathcal{F}_{T_0}}^2,\label{eq:final.besov1}\\
        & \int_{T_0}^T \EE_{T_0} \left[\| f_t\|_{\dot H^{\beta\gamma}_x}^2  \right] \mathd t 
        \lesssim \EE_{T_0}\left[\| f\|_{L^2_{[T_0,T],x}}^{2}\right]^{1-\gamma} \seminorm{ f    }_{\beta,+\infty \mid \mathcal{F}_{T_0}}^{2\gamma}. \label{eq:final.besov2}
    \end{align}
    where the hidden constants do not depend on $h,{\varrho}_0,T_0,T$.
\end{lem}

\begin{proof}
    Denote by $\{\Delta_j\}_{j \geq -1}$ the inhomogeneous Littlewood--Paley blocks in the space variable;
    arguing as in \cite[Lemma 2.18]{DrGaPa25} (up to taking conditional expectation instead of plain expectation in the proof thereof) it holds $\PP$-a.s.
    \begin{equation} \label{eq:LP_bound.aux}
        \sup_{j> -1} 2^{2\beta j} \mathbb{E}_{T_0} \left[ \| \Delta_j \delta_h f\|^2_{L^2_{[T_0,T],x}}
          \right]
        \lesssim  
        \sup_{\varepsilon>0} \frac{1}{\varepsilon^{2\beta}} \fint_{\mathbb{S}^{d-1}} 
        \mathbb{E}_{T_0} \left[ \|\delta_{\varepsilon z} (\delta_h f)\|^2_{L^2_{[T_0,T],x}}   \right]\sigma(\dd z)
        .
    \end{equation}
As in the case of $\sup_{\varrho \in (0,\varrho_0)}G_{\beta,T_0}(\varrho,f)$, the right-hand side of \eqref{eq:LP_bound.aux} is a well-defined $\mathcal{F}_{T_0}$-measurable random variable.

We focus on \eqref{eq:final.besov1}. Let us distinguish two cases: if $\varepsilon \leq  |h|$, since $\delta_{\varepsilon z}$ and $\delta_h$ commute, by \cref{lem:auxiliary.besov.1} we find
\begin{align*}
    \mathbb{E}_{T_0} \left[
    \|\delta_{\varepsilon z} (\delta_h f)\|^2_{L^2_{[T_0,T],x}}
     \right] 
    &=
    \mathbb{E}_{T_0} \left[
    \|\delta_h (\delta_{\varepsilon z} f)\|^2_{L^2_{[T_0,T],x}}
      \right]
    \\
    &\lesssim
    \mathbb{E}_{T_0} \left[
    \| \delta_{\varepsilon z} f \|^2_{L^2_{[T_0,T],x}}
      \right]
    \lesssim
    \varepsilon^{2\beta}  \seminorm{f  }_{\beta,\varrho_0 \mid \mathcal{F}_{T_0}}^2;
\end{align*}
Instead, if $\varepsilon > |h|$, one can use the trivial bound
\begin{align*}
   \mathbb{E}_{T_0} \left[
    \|\delta_{\varepsilon z} (\delta_h f)\|^2_{L^2_{[T_0,T],x}}
      \right]
    \lesssim
    \mathbb{E}_{T_0} \left[
    \| \delta_h f\|^2_{L^2_{[T_0,T],x}}
      \right]
    \lesssim
    \varepsilon^{2\beta}  \seminorm{f  }_{\beta,\varrho_0 \mid \mathcal{F}_{T_0}}^2.
\end{align*}
In either case, from \eqref{eq:LP_bound.aux} and the above estimates we deduce
\begin{align}\label{eq:LP_bound}
    \mathbb{E}_{T_0} \left[ \| \Delta_j \delta_h f\|^2_{L^2_{[T_0,T],x}}
         \right]
        \lesssim 
        2^{-2\beta j} \seminorm{f  }_{\beta,\varrho_0 \mid \mathcal{F}_{T_0}}^2.
\end{align}
By the Besov identifications $H^{\beta \gamma}_x= B^{\beta\gamma}_{2,2}$ and $L^2_x=B^{0}_{2,2}$, we have for any $N \in \N$
\begin{align*}
    \int_0^T \mathbb{E}_{T_0} \left[ \| \delta_h f_t\|_{H^{\beta\gamma}_x}^2   \right] \mathd t
        &\sim 
        \sum_{j\geq -1} 2^{2j\beta\gamma} \mathbb{E}_{T_0} \left[  \| \Delta_j \delta_h f_t\|_{L^2_{[T_0,T],x}}^2   \right]
        \\
        & \lesssim 
        2^{2N\beta\gamma} \sum_{-1\leq j\leq N} \mathbb{E}_{T_0} \left[  \| \Delta_j \delta_h f\|_{L^2_{[T_0,T],x}}^2   \right]
        + 
        \sum_{j> N} 2^{-2 j\beta (1-\gamma)}
        \seminorm{f  }_{\beta,\varrho_0 \mid \mathcal{F}_{T_0}}^2
        \\
        & \lesssim 
        2^{2N\beta\gamma} \mathbb{E}_{T_0} \left[  \| \delta_h f\|_{L^2_{[T_0,T],x}}^2  \right]
        + 
        2^{-2 N \beta  (1-\gamma) } \seminorm{f  }_{\beta,\varrho_0 \mid \mathcal{F}_{T_0}}^2
        \\
        & \lesssim 
        \left( 2^{2N\beta\gamma} |h|^{2\beta} + 2^{-2N\beta (1-\gamma)} \right) \seminorm{f  }_{\beta,\varrho_0 \mid \mathcal{F}_{T_0}}^2,
    \end{align*}
    where in the above passages we used \eqref{eq:LP_bound} and \cref{lem:auxiliary.besov.1}. Choosing $N$ such that $2^{-N}\sim |h|$ then yields the desired \eqref{eq:final.besov1}.

    The proof of \eqref{eq:final.besov2} is similar, using the homogeneous Littlewood--Paley blocks $\{\dot\Delta_j\}_{j \in\Z}$, the identification $\dot H^{\beta\gamma}_x=\dot B^{\beta\gamma}_{2,2}$ and similar optimization arguments with $N\in\Z$; the only difference is the use of the Bernstein estimate $\| \dot\Delta_j f\|_{L^2_{[T_0,T],x}}^2 \lesssim 2^{2j} \| f\|_{L^2_{[T_0,T],x}}^2$ valid for all $j\in\Z$.
\end{proof}

\begin{cor}\label{lem:besov.higher.time}
Under the same assumptions of \cref{lem:final_besov}, for every $t_\ast \in [0,T-T_0]$ it holds
    \begin{align*}   
        & \int_{T_0}^{T_0+t_\ast} \EE_{T_0} \left[ \| f_t\|_{H^{\beta\gamma}_x}^2   \right] \mathd t\\
        & \ \ \lesssim (1+\varrho_0^{-2\beta\gamma})\EE_{T_0}\left[\| f\|_{L^2_{[T_0,T],x}}^{2}\right] + \EE_{T_0}\left[\| f\|_{L^2_{[T_0,T],x}}^{2}\right]^{1-\gamma} \seminorm{ f  \mathbf{1}_{[T_0, T_0+t_\ast]}}_{\beta,+\infty \mid \mathcal{F}_{T_0}}^{2\gamma}\\
        & \ \ \lesssim t_\ast (1+\varrho_0^{-2\beta\gamma}) \left\| \EE_{T_0} \left[ \| f \|_{L^2_x}^2    \right] \right\|_{L^\infty_{[T_0,T_0+t_\ast]}}
        + t_\ast^{1-\gamma} \left\| \EE_{T_0} \left[ \| f \|_{L^2_x}^2    \right] \right\|_{L^\infty_{[T_0,T_0+t_\ast]}}^{1-\gamma} \seminorm{ f  \mathbf{1}_{[T_0, T_0+t_\ast]}}_{\beta,+\infty \mid \mathcal{F}_{T_0}}^{2\gamma}.
    \end{align*} 
    where the hidden constant does not depend on ${\varrho}_0,T_0,t_\ast,T$.
\end{cor}

\begin{proof}
    Without loss of generality we may assume $t_\ast=T-T_0$. The first estimate follows from combining \eqref{eq:final.besov2} with the bounds
    \begin{align*}
        \| f_t\|_{H^{\beta\gamma}_x}^2 \lesssim \| f_t\|_{L^2_x}^2 + \| f_t\|_{\dot H^{\beta\gamma}_x}^2,\quad
        \seminorm{ f    }_{\beta,+\infty \mid \mathcal{F}_{T_0}}^2 \lesssim \varrho_0^{-2\beta} \EE_{T_0}\left[ \| f\|_{L^2_{[T_0,T],x}}^2\right]+\seminorm{ f    }_{\beta,\varrho_0 \mid \mathcal{F}_{T_0}}^2
    \end{align*}
    where the second one follows from the definition of $\seminorm{ f    }_{\beta,+\infty \mid \mathcal{F}_{T_0}}$ and triangular inequality.
    The second estimate then follows from the first one combined with
    \begin{equation*}
        \EE_{T_0}\left[\| f\|_{L^2_{[T_0,T],x}}^{2}\right] \leq t_\ast \left\| \EE_{T_0} \left[ \| f \|_{L^2_x}^2    \right] \right\|_{L^\infty_{[T_0,T]}}. \qedhere
    \end{equation*}
\end{proof}

\bibliography{biblio}{}
\bibliographystyle{alpha}

\end{document}